\documentclass{amsart}
\usepackage{tikz-cd}
\usepackage{lscape}
\usepackage{enumitem}
\usepackage{placeins}
\usepackage{array}
\usepackage{subcaption}

\usepackage{Definitions}
\usepackage{Environments}
\usepackage{macros}
\usepackage{tikzdecls}

\title{Bicategorical traces and cotraces}
\author{Justin Barhite}

\begin{document}

\begin{abstract}
The familiar trace of a square matrix generalizes to a trace of an endomorphism of a dualizable object in a symmetric monoidal category. To extend these ideas to other settings, such as modules over non-commutative rings, the trace can be generalized to a bicategory equipped with additional structure called a shadow. We propose a notion of bicategorical cotrace of certain maps involving dualizable objects in a closed bicategory equipped with a coshadow, and we use this framework to draw connections to work of Lipman on residues and traces with Hochschild (co)homology, and to work of Ganter and Kapranov on 2-representations and 2-characters.
\end{abstract}

\maketitle

\section{Introduction}

A recurring theme in mathematics is that traces turn a complicated object into a simpler one, discarding much of the information in the original object but retaining just enough information to say something useful about it. In particular, many familiar invariants arise as traces of identity maps. The trace of the identity map on a real or complex vector space is the vector space's dimension, which classifies the vector space up to isomorphism. The trace of the identity map on a finite CW complex is its Euler characteristic (vertices plus edges minus faces, and so on in higher dimensions), which is not a complete invariant but is still a useful tool for distinguishing between topological spaces. Finally, the trace of the identity map on a group representation is its character, which forgets most of the data of the representation and retains only the trace of each group element's action on the underlying vector space; over a field of characteristic zero, the character determines the representation up to isomorphism.

The first two of these examples are generalized by a well-known notion of trace of an endomorphism of a dualizable object in a symmetric monoidal category \cite{DoldPuppe1978, PS_sym_mon_traces} (\emph{dualizability} generalizes the finiteness conditions in these examples). The character of a group representation, however, requires the trace in a bicategory equipped with a shadow, which was introduced by Ponto \cite{Ponto_thesis, Shadows_and_traces, Ponto2015equivariant, Ponto2011relative} to generalize fixed-point invariants such as the Reidemeister trace. The bicategorical trace also subsumes the Hattori-Stallings trace \cite{Hattori1965, Stallings1965}, which extends the familiar linear algebra trace to endomorphisms of modules over noncommutative rings, and of which the group character is an example. The price we pay for including noncommutative rings is that the trace of an $R$-module endomorphism takes values not in $R$ itself but only in the quotient $\overline{R}$ of $R$ by the subgroup generated by all elements of the form $rs - sr$. The passage from $R$ to $\overline{R}$ is an example of a \emph{shadow}, which is the additional structure a bicategory requires in order to support a notion of trace.

There are, however, some examples of trace-like constructions which are not fully explained by categorical notions of trace; incorporating them into this perspective requires a dual notion of \emph{cotrace}. The example that originally motivated our development of a bicategorical cotrace is the cotrace maps of \cite{Lipman1987}. Lipman provides a more elementary development of Grothendieck's residue symbol \cite{Hartshorne1966} by reframing it in terms of Hochschild homology, and he establishes a sort of adjointness between trace and cotrace, mediated by a pairing map on Hochschild homology and cohomology. The simplest case of this adjointness is expressed by the following commutative diagram, where $R$ is any ring and $F$ is a right $R$-module which is finitely generated and projective (this is the appropriate sort of ``finiteness'' for modules):

\begin{equation}
\label{eq:interplay_basic_example}
\begin{tikzpicture}[xscale=3.25, yscale=1.5]
	\node (A) at (1, 3) {$R^c \otimes \overline{\Hom_R(F, F)}$};
	\node (B) at (0, 2) {$\Hom_R(F, F)^c \otimes \overline{\Hom_R(F, F)}$};
	\node (C) at (0, 1) {$\overline{\Hom_R(F, F)}$};
	\node (D) at (1, 0) {$\overline{R}$};
	\node (E) at (2, 2) {$R^c \otimes \overline{R}$};
	\node (F) at (2, 1) {$\overline{R}$};

	\draw [->] (A) -- node[above left]{$\cotr \otimes 1$} (B);
	\draw [->] (A) -- node[above right]{$1 \otimes \tr$} (E);
	\draw [->] (B) -- node[left]{$\rho_{\Hom_R(F, F)}$} (C);
	\draw [->] (E) -- node[right]{$\rho_R$} (F);
	\draw [->] (C) -- node[below left]{$\tr$} (D);
	\draw [double equal sign distance] (F) -- (D);
\end{tikzpicture}
\end{equation}

Just as the Hattori-Stallings trace necessitates the passage from $R$ to its quotient $\overline{R}$, a cotrace compels us, dually, to replace $R$ with its center $R^c$. The trace map
\[
	\overline{\Hom_R(F, F)} \to \overline{R}\]
is the Hattori-Stallings trace; its domain is $\overline{\Hom_R(F, F)}$, rather than $\Hom_R(F, F)$, because of cyclicity of the trace, i.e.~the fact that $\tr(fg) = \tr(gf)$. The cotrace
\[
	R^c \to \Hom_R(F, F)^c
\]
sends $r$ to multiplication by $r$, i.e.~the map $\mu_r : F \to F, \, x \mapsto xr$. The pairing map $\rho_R$ is multiplication in $R$, and $\rho_{\Hom_R(F, F)}$ is composition. Thus the diagram asserts that $\tr(\mu_r \circ f) = r \tr(f)$, but writing this as
\[
	\tr(\rho(\cotr(r), f)) = \rho(r, \tr(f))
\]
makes it more strongly resemble an adjointness between $\cotr$ and $\tr$.

Lipman acknowledges that his description of residues is not fully satisfactory, and he suspects that ``there might well be a more fundamental approach to the subject, encompassing a great deal more than we have dealt with here'' \cite{Lipman1987}. In Sections~\ref{sec:coshadows_and_cotraces} and \ref{sec:interplay} we offer a candidate for the more fundamental approach Lipman imagined, by repackaging his traces and cotraces in terms of Ponto's bicategorical trace \cite{Ponto_thesis, Shadows_and_traces} and a new notion of bicategorical cotrace. In doing so, we have teased out the formal structure underlying Lipman's trace formulas, which includes (1) ``coshadows'' and ``cotraces'' dual to Ponto's shadows and traces and (2) an interplay between bicategorical traces and cotraces generalizing the results of \cite{Lipman1987}.

Before we describe this interplay, some comments are in order about the shape that traces and cotraces take. The trace of a linear transformation $f : V \to V$ is often thought of as a \emph{number} (i.e.~an element of the ground field), but we prefer to think of it as a \emph{map}, namely the linear transformation $k \to k$ sending $1$ to the number which is ordinarily thought of as the trace of $f$. While this distinction may seem inconsequential for the example of vector spaces, it turns out to be a crucial shift in perspective that brings many different mathematical constructs underneath the umbrella of ``trace.'' For example, the trace of an endomorphism of a $k$-linear representation $V$ of a group $G$ is a map $\overline{k[G]} \to k$ (note that $\overline{k} = k$ since $k$ is a field, and in particular is commutative). Such a linear transformation amounts to a $k$-valued class function on $G$, and in fact the trace of the identity map on $V$ is nothing other than the character of $V$ (Example~\ref{ex:character}). Thus by embracing this shift in perspective and viewing traces as maps, we can consolidate the entire data of a group character into a single trace, rather than a collection of many traces (one for each conjugacy class).

The reason that the trace of an endomorphism of $V$ has the form $\overline{k[G]} \to \overline{k}$ is that an endomorphism of an $(R, S)$-bimodule has a trace $\overline{R} \to \overline{S}$, and a $G$-representation can be viewed as a $(k[G], k)$-bimodule, where $k[G]$ is the group algebra. More generally, the operation $R \mapsto \overline{R}$ can be replaced by a \emph{shadow} $\sh{-}$, which abstracts the properties of this quotienting operation that facilitate a sensible notion of trace; the trace of an $(R, S)$-bimodule (or a 1-cell $R \to S$ in a bicategory) is then a map $\sh{R} \to \sh{S}$. Dually, the replacement of a ring by its center is generalized by a coshadow $\lsh{-}$, which supports the formation of cotraces $\lsh{R} \to \lsh{S}$. With this notation in place, we are finally prepared to state the bicategorical generalization of the cotrace-trace adjointness illustrated by (\ref{eq:interplay_basic_example}), albeit somewhat imprecisely (see Theorem~\ref{thm:interplay} for the precise statement).

\begin{thm}
\label{thm:main_result_intro}
Given suitable maps $f, g, h, i$ involving objects $A, A', B, B', C$, two shadows $\bluesh{-}$ and $\sh{-}$, a coshadow $\bluelsh{-}$, and a pairing map $\bluelsh{-} \otimes \bluesh{-} \xrightarrow{\rho} \sh{- \otimes -}$, the following commutes:
\[\begin{tikzcd}
	& \bluelsh{A} \otimes \bluesh{B'}
		\arrow[dl, "\cotr(f) \otimes 1"']
		\arrow[dr, "1 \otimes \tr(g)"]
	\\
	\bluelsh{A'} \otimes \bluesh{B'}
		\arrow[d, "\rho"']
	&& \bluelsh{A} \otimes \bluesh{B}
		\arrow[d, "\rho"]
	\\
	\sh{A' \otimes B'}
		\arrow[dr, "\tr(h)"']
	&& \sh{A \otimes B}
		\arrow[dl, "\tr(i)"]
	\\
	& \sh{C}
\end{tikzcd}\]
\end{thm}

This abstract perspective on traces is valuable because it often allows us to make mathematical constructs and theorems more accessible by extracting the core ideas from the technical details of their original presentation. Moreover, we are often able to prove vastly more general versions of these results (for example, Theorem~\ref{thm:main_result_intro}) and port them over to other mathematical contexts. For example, the tools that we have built to understand Lipman's work also have applications to the theory of group representations and 2-representations.


The aim of this paper is to lay the foundation for a theory of bicategorical coshadows and cotraces. The main result is Theorem~\ref{thm:main_result_intro} (stated precisely in Theorem~\ref{thm:interplay}), but along the way we establish properties of coshadows and cotraces analogous to many of the properties of trace described in \cite{Shadows_and_traces}, including cyclicity, functoriality, and Morita invariance. The overarching goal is to illuminate connections between diverse mathematical contexts, so we emphasize the application of this categorical machinery to examples such as the trace-cotrace interplay of \cite{Lipman1987} and the categorical trace of \cite{GK2006}.

We begin in Section~\ref{sec:duality_and_trace} by reviewing the theory of duality and trace in symmetric monoidal categories and in bicategories, along with the graphical calculus of string diagrams, and then in Section~\ref{sec:closed_bicategories} we specialize to closed bicategories, which will be the setting for our study of cotraces. In Sections~\ref{sec:coshadows_and_cotraces} and \ref{sec:cotrace_basic_properties} we begin to develop the theory of bicategorical cotraces, and we demonstrate that while cotraces can also be defined in symmetric monoidal categories, they coincide with symmetric monoidal traces. The main result of the paper, Theorem~\ref{thm:main_result_intro}, is proven in Section~\ref{sec:interplay}, showing that traces and cotraces are compatible with the additional structure of a pairing map from a shadow and coshadow to a second shadow. Finally, in Sections~\ref{sec:functoriality} and \ref{sec:Morita_invariance} we establish functoriality of cotrace and Morita invariance of coshadows.

\subsection*{Acknowledgements}

The author is grateful to Jonathan Campbell for bringing to our attention the paper \cite{Lipman1987} that motivated this work, to Cary Malkiewich for suggesting a better system of string diagrams and thereby helping to formulate cyclicity for cotraces, to Millie Rose Deaton, Niles Johnson, Inbar Klang, and Jordan Sawdy for helpful conversations, to his advisor Kate Ponto for all her support and encouragement, and to the referee for careful reading and helpful suggestions. This material is based upon work supported by the National Science Foundation under Award Nos.~DMS-2052905 and DMS-1810779.

\section{Duality and trace}
\label{sec:duality_and_trace}

We begin by reviewing the theory of duality and trace in symmetric monoidal categories and bicategories, which is developed in \cite{DoldPuppe1978, LMS, MaySigurdsson, Ponto_thesis, Shadows_and_traces}. The central idea is that the property of dualizability allows the extraction of interesting invariants via traces.

\subsection{Symmetric monoidal duality and trace}

If $V$ is a finite-dimensional vector space, then any linear transformation $f : V \to V$ has a trace, which is the sum of the diagonal entries in any matrix representation of $f$. Traces exist in much more general contexts though; instead of endomorphisms vector spaces, we can take the trace of an endomorphism of an object in a symmetric monoidal category, as long as that object satisfies a ``dualizability'' condition generalizing finite-dimensionality for vector spaces.

\begin{defn}[{\cite[Theorem 1.3]{DoldPuppe1978}}]
\label{defn:dual_pair_sm}
Let $(\mathscr{C}, \otimes, I)$ be a symmetric monoidal category. An object $M$ of $\mathscr{C}$ is \KEYWORD{dualizable} if there is an object $M^*$ of $\mathscr{C}$ and maps
\[
	\eta : I \to M \otimes M^* \qquad\text{and}\qquad \varepsilon : M^* \otimes M \to I
\]
such that the following \KEYWORD{triangle identities} hold:
\[\begin{tikzcd}
	M
		\arrow[r, "\eta \otimes \id"]
		\arrow[dr, "\id"']
	& M \otimes M^* \otimes M
		\arrow[d, "\id \otimes \varepsilon"]
	\\
	& M
\end{tikzcd} \qquad\qquad
\begin{tikzcd}
	M^*
		\arrow[r, "\id \otimes \eta"]
		\arrow[dr, "\id"']
	& M^* \otimes M \otimes M^* \arrow[d, "\varepsilon \otimes \id"]
	\\
	& M^*
\end{tikzcd}\]
We say that $M^*$ is a \KEYWORD{dual} for $M$, and we refer to $\eta$ and $\varepsilon$ as the \KEYWORD{coevaluation} and \KEYWORD{evaluation} (respectively) of the dual pair (they are sometimes called the \emph{unit} and \emph{counit}, but we avoid that terminology so as not to overload the word ``unit'').
\end{defn}

In the previous definition and often elsewhere, we suppress unit isomorphisms; for instance, we write $M \xrightarrow{\eta \otimes \id} M \otimes M^* \otimes M$ as shorthand for the composite $M \cong I \otimes M \xrightarrow{\eta \otimes \id} M \otimes M^* \otimes M$.

\begin{rmk}
\label{rmk:dual_pair_other_way}
If $(M, M^*)$ is a dual pair with coevaluation and evaluation $\eta$ and $\varepsilon$, then $(M^*, M)$ is also a dual pair, with coevaluation and evaluation
\[
	I \xrightarrow{\eta} M \otimes M^* \xrightarrow{\cong} M^* \otimes M \qquad\text{and}\qquad M \otimes M^* \xrightarrow{\cong} M^* \otimes M \xrightarrow{\varepsilon} I.
\]
\end{rmk}

\begin{example}
\label{ex:vector_spaces}
A vector space $V$ over a field $k$ is dualizable if and only if it is finite-dimensional. If $V$ is finite-dimensional, then $V^* := \Hom_k(V, k)$ is a dual for $V$; if $e_1, \ldots, e_n$ is a basis for $V$ and $e_1^*, \ldots, e_n^*$ is the corresponding dual basis for $V^*$ (so that $e_i^*(e_j) = \delta_{ij}$), then the maps
\[\begin{tikzcd}[row sep=0.5]
	k
		\arrow[r, "\eta"]
	& V \otimes V^*
	\\
	1
		\arrow[r, |->]
	& \sum_{i=1}^n e_i \otimes e_i^*
\end{tikzcd} \qquad\qquad
\begin{tikzcd}[row sep=0.5]
	V^* \otimes V
		\arrow[r, "\varepsilon"]
	& k
	\\
	\phi \otimes v
		\arrow[r, |->]
	& \phi(v)
\end{tikzcd}\]
exhibit $(V, V^*)$ as a dual pair. Conversely, if $V$ is a dualizable vector space, then $\eta(1) =: \sum_{i=1}^n v_i \otimes \phi_i$ is some finite sum of simple tensors, and one of the triangle identities implies that $v_1, \ldots, v_n$ generate $V$, so $V$ is finite-dimensional.
\end{example}

Dualizability allows us to extract useful information about an object; for example, if a vector space is dualizable (i.e.~finite-dimensional), we can use the structure of a dual pair to determine its dimension:

\begin{example}
\label{ex:vector_space_Euler_characteristic}
If $V$ is a finite-dimensional vector space over $k$ and $V^*$ is its dual, with $\eta$ and $\varepsilon$ as in Example~\ref{ex:vector_spaces}, then the composite
\begin{equation}
\label{eq:vector_space_Euler_characteristic}
	k \xrightarrow{\eta} V \otimes V^* \xrightarrow{\cong} V^* \otimes V \xrightarrow{\varepsilon} k
\end{equation}
is the element of $\Hom_k(k, k)$ which is multiplication by $\dim V$.
\end{example}

While there may be many choices of $V^*$, $\eta$, and $\varepsilon$ satisfying Definition~\ref{defn:dual_pair_sm}, the map (\ref{eq:vector_space_Euler_characteristic}) is independent of such choices, since it is determined solely by the dimension of $V$. We could have instead defined $\varepsilon$ as a map $V \otimes V^* \to k$, making the symmetry isomorphism in (\ref{eq:vector_space_Euler_characteristic}) unnecessary; however, we cannot avoid the need for symmetry isomorphisms entirely, as they would then appear in the triangle identities instead. Besides, the way we have formulated the definitions generalizes more naturally to the bicategorical setting.

A similar story plays out in topological settings as well, though we need to pass to a category of spectra since the category of spaces has no nontrivial dualizable objects. Having done so, we produce an endomorphism of the monoidal unit analogous to (\ref{eq:vector_space_Euler_characteristic}) and find that it recovers a familiar topological invariant:

\begin{example}[\cite{DoldPuppe1978}]
\label{ex:CW_Euler_characteristic}
If $X$ is a compact CW complex, then $\Sigma^{\infty}_+ X$ is dualizable in the stable homotopy category $\mathrm{Ho}(\mathbf{Sp})$, with dual $DX$ (the Spanier-Whitehead dual), and the map
\begin{equation}
\label{eq:CW_Euler_characteristic}
	S \xrightarrow{\eta} \Sigma^{\infty}_+ X \wedge DX \cong DX \wedge \Sigma^{\infty}_+ X \xrightarrow{\varepsilon} S
\end{equation}
is the element of $\Hom_{\mathrm{Ho}(\mathbf{Sp})}(S, S) = \pi_0(S) \cong \mathbb{Z}$ which is multiplication by the Euler characteristic of $X$.
\end{example}

In either of these examples, we can insert an endomorphism $f$ of the dualizable object in order to obtain information about $f$.

\begin{example}
\label{ex:vector_space_trace}
Let $f : V \to V$ be an endomorphism of a finite-dimensional vector space $V$ over $k$. Then the composite
\[
	k \xrightarrow{\eta} V \otimes V^* \xrightarrow{f \otimes 1} V \otimes V^* \xrightarrow{\cong} V^* \otimes V \xrightarrow{\varepsilon} k
\]
is multiplication by the trace of $f$.
\end{example}

\begin{example}
\label{ex:Lefschetz_number}
If $X$ is a compact CW complex and $f : X \to X$, then
\[
	S \xrightarrow{\eta} \Sigma^{\infty}_+ X \wedge DX \xrightarrow{\Sigma^{\infty}_+ f \wedge 1} \Sigma^{\infty}_+ X \wedge DX \cong DX \wedge \Sigma^{\infty}_+ X \xrightarrow{\varepsilon} S
\]
corresponds to multiplication by the Lefschetz number of $f$ \cite{DoldPuppe1978}.
\end{example}

These are examples of the following general notion of trace in a symmetric monoidal category:

\begin{defn}[\cite{DoldPuppe1978}]
\label{defn:symmetric_monoidal_trace}
Let $M$ be a dualizable object in a symmetric monoidal category, with dual $M^*$. The \KEYWORD{trace} of $f : M \to M$ is
\[
	I \xrightarrow{\eta} M \otimes M^* \xrightarrow{f \otimes 1} M \otimes M^* \xrightarrow{\cong} M^* \otimes M \xrightarrow{\varepsilon} I.
\]
\end{defn}
If $f$ is an identity map, we call its trace the \KEYWORD{Euler characteristic} of $M$, by analogy with Example~\ref{ex:CW_Euler_characteristic}. The trace of $f$ is independent of the choices of $M^*$, $\eta$, and $\varepsilon$.

\subsection{Bicategorical duality and trace}

The machinery of symmetric monoidal traces does not apply to modules over noncommutative rings, since $\MOD_R$ does not have a monoidal tensor product if $R$ is not commutative. There is, however, a sensible notion of trace for modules over noncommutative rings, the Hattori-Stallings trace \cite{Hattori1965, Stallings1965}, but it takes values in a quotient of the ring rather than the ring itself. The appropriate category-theoretic setting for describing this trace is a bicategory.

A bicategory $\mathscr{B}$ has a collection of objects (also called \emph{0-cells}), and between any two objects it has not merely a \emph{set} of morphisms but rather a \emph{category}, whose objects and morphisms are called \emph{1-cells} and \emph{2-cells} of $\mathscr{B}$, respectively; for a complete definition of a bicategory see \cite{Leinster_bicategories}. We denote composition of 1-cells in a bicategory $\mathscr{B}$ by $\odot$, and we write composition in diagrammatic order. If $R$ is a 0-cell in $\mathscr{B}$, we denote the identity 1-cell for $R$ by $U_R$. Sometimes we write $M : R \pto S$ to indicate that $M$ is a 1-cell from $R$ to $S$, i.e.~$M \in \mathscr{B}(R, S)$. The most useful bicategory to keep in mind is $\MORITABICAT$, the Morita bicategory of rings, bimodules, and bimodule homomorphisms. In this bicategory, $U_R$ is $R$ with the standard $(R, R)$-bimodule structure, and $\odot$ is the tensor product. We will sometimes also work in $\MORITABICATK$, the bicategory of algebras over a commutative ring $k$, bimodules, and bimodule homomorphisms.

The following definition first appeared as \cite[Definition 16.4.1]{MaySigurdsson}.

\begin{defn}
\label{defn:dual_pair_bicat}
Let $M$ be a 1-cell in a bicategory $\mathscr{B}(R, S)$. We say $M$ is \KEYWORD{right dualizable} if there is a 1-cell $M^*$ together with 2-cells
\[
	\eta : U_R \to M \odot M^* \qquad\text{and}\qquad \varepsilon : M^* \odot M \to U_S
\]
such that the triangle identities hold. We say that $M^*$ is \KEYWORD{right dual} to $M$, that $(M, M^*)$ is a \KEYWORD{dual pair}, that $M^*$ is \KEYWORD{left dualizable}, and that $M$ is its \KEYWORD{left dual}.
\end{defn}

Unlike dual pairs in symmetric monoidal categories, bicategorical dual pairs have a sidedness: right dualizability is not the same as left dualizability (cf.~Remark~\ref{rmk:dual_pair_other_way}).

\begin{rmk}
\label{rmk:uniqueness_of_duals}
As in the symmetric monoidal setting, duals are unique up to isomorphism. For example, suppose that $(M, N)$ is a dual pair with coevaluation and evaluation $\eta$ and $\varepsilon$ and that $(M, N')$ is a dual pair with coevaluation and evaluation $\eta'$ and $\varepsilon'$. Then there are maps
\[N \xrightarrow{\id \odot \eta'} N \odot M \odot N' \xrightarrow{\varepsilon \odot \id} N' \qquad\text{and}\qquad N' \xrightarrow{\id \odot \eta} N' \odot M \odot N \xrightarrow{\varepsilon' \odot \id} N\]
which are mutually inverse by the triangle identities for both dual pairs.
\end{rmk}

\begin{example}
\label{ex:dualizable_bimodule}
An $(R, S)$-bimodule is right dualizable in $\MORITABICAT$ if and only if it is finitely generated and projective as a right $S$-module. The argument is similar to Example~\ref{ex:vector_spaces} and uses the fact that a right $S$-module $M$ is finitely generated and projective if and only if there are $m_1, \ldots, m_n \in M$ and $m_1^*, \ldots, m_n^* \in \Hom_S(M, S)$ such that
\[
	\sum_{i=1}^n m_i m_i^*(m) = m
\]
for all $m \in M$.
\end{example}

\begin{example}
In the bicategory of categories, functors, and natural transformations, the dual pairs are the adjoint pairs of functors, but in an unfortunate clash of nomenclature, \emph{left} adjoints are \emph{right} duals as a consequence of the order in which we have chosen to write composition.
\end{example}

\begin{prop}[{\cite[Theorem 16.5.1]{MaySigurdsson}}]
\label{prop:composite_of_dual_pairs}
If $(M, M^*)$ and $(N, N^*)$ are dual pairs, then so is $(M \odot N, N^* \odot M^*)$.
\end{prop}

\begin{proof}
Let $M \in \mathscr{B}(R, S)$ and $N \in \mathscr{B}(S, T)$. If the dual pairs $(M, M^*)$ and $(N, N^*)$ have coevaluations $\eta$ and $\eta'$ and evaluations $\varepsilon$ and $\varepsilon'$, respectively, then
\[
U_R \xrightarrow{\eta} M \odot M^* \xrightarrow{1 \odot \eta' \odot 1} M \odot N \odot N^* \odot M^*
\]
\[
N^* \odot M^* \odot M \odot N \xrightarrow{1 \odot \varepsilon \odot 1} N^* \odot N \xrightarrow{\varepsilon'} U_T
\]
are a coevaluation and an evaluation for $(M \odot N, N^* \odot M^*)$.
\end{proof}

If we attempt to write down a composite like the symmetric monoidal trace of Definition~\ref{defn:symmetric_monoidal_trace}, we quickly encounter a problem: there is no symmetry isomorphism to get us from $M \odot M^*$ to $M^* \odot M$, and in fact these are not even objects in the same category. One possible remedy is to apply a functor to push $M \odot M^*$ and $M^* \odot M$ into a third category, where we can ask for an isomorphism between their images. This is the notion of a \emph{bicategorical shadow}, which was introduced in \cite{Ponto_thesis}:

\begin{defn}
\label{defn:shadow}
A \KEYWORD{shadow} for a bicategory $\mathscr{B}$ is a category $\mathbf{T}$ and functors
\[
	\sh{-} : \mathscr{B}(R, R) \to \mathbf{T}
\]
for each 0-cell $R$ of $\mathscr{B}$, equipped with natural isomorphisms
\[
	\theta_{M, N} : \sh{M \odot N} \xrightarrow{\cong} \sh{N \odot M}
\]
for each $M \in \mathscr{B}(R, S)$ and $N \in \mathscr{B}(S, R)$, such that the following diagrams commute whenever they make sense:
\[
\begin{tikzcd}
	\sh{(M \odot N) \odot P}
		\arrow[r, "\theta", "\cong"']
		\arrow[d, "\sh{a}"', "\cong"]
	& \sh{P \odot (M \odot N)}
		\arrow[r, "\sh{a^{-1}}", "\cong"']
	& \sh{(P \odot M) \odot N}
		\arrow[d, "\theta", "\cong"']
	\\
	\sh{M \odot (N \odot P)}
		\arrow[r, "\theta"', "\cong"]
	& \sh{(N \odot P) \odot M}
		\arrow[r, "\sh{a}"', "\cong"]
	& \sh{N \odot (P \odot M)}
\end{tikzcd}
\]
\[
\begin{tikzcd}
	\sh{M \odot U_R}
		\arrow[r, "\theta", "\cong"']
		\arrow[dr, "\sh{r}"', "\cong"]
	& \sh{U_R \odot M}
		\arrow[r, "\theta", "\cong"']
		\arrow[d, "\sh{l}"', "\cong"]
	& \sh{M \odot U_R}
		\arrow[dl, "\sh{r}", "\cong"']
	\\
	& \sh{M}
\end{tikzcd}
\]
\end{defn}

Here and elsewhere, the phrase ``whenever they make sense'' means that the 1-cells have source and target 0-cells which make all tensor products and shadows (and, later, hom-objects and coshadows) valid. For example, the hexagon in Definition~\ref{defn:shadow} ``makes sense'' if the targets of $M$, $N$, and $P$ equal the sources of $N$, $P$, and $M$, respectively. The isomorphisms $\theta_{M, N}$ and $\theta_{N, M}$ are mutually inverse \cite[Proposition 4.3]{Shadows_and_traces}, so we usually drop the subscripts and simply write $\theta$.

\begin{example}
The zeroth Hochschild homology $\HH_0(R, M)$ is the quotient (of abelian groups) of $M$ by the subgroup generated by all elements of the form $mr - rm$. This quotient can also be described as $M \otimes_{R \otimes R^{\op}} R$. Hochschild homology defines a shadow on $\MORITABICAT$ with target $\mathbf{Ab}$ (or a shadow on $\MORITABICATK$ with target $\mathbf{Vect}_k$). This boils down to the fact that
\[
	\HH_0(R, M \otimes_S N) \cong \HH_0(S, N \otimes_R M)
\]
for bimodules $_RM_S$ and $_SN_R$, since the relation $rm \otimes n \sim m \otimes nr$ imposed by Hochschild homology mirrors the relation $ms \otimes n \sim m \otimes sn$ imposed by the passage from $M \otimes_{\mathbb{Z}} N$ to $M \otimes_S N$.
\end{example}

\begin{example}
There is a bicategory $\CHBICAT$ whose 0-cells are rings, 1-cells are nonnegatively graded chain complexes of bimodules, and 2-cells are chain maps. This bicategory has a shadow, which we also call $\HH_0$, that takes values in chain complexes of abelian groups and is given by
\[
	\HH_0(R, C) = C \otimes_{R \otimes R^{\op}} R[0],
\]
where $R[0]$ is the complex which has $R$ in degree zero and is zero in positive degree; that is, $\HH_0(R, C)_n \cong \HH_0(R, C_n)$. The cyclicity isomorphism
\[
	\theta : \HH_0(R, C \otimes_S D) \cong \HH_0(S, D \otimes_R C)
\]
has a sign: $\theta(c \otimes d) = (-1)^{|c||d|} d \otimes c$.
\end{example}

The following generalizes the symmetric monoidal Euler characteristic of Definition~\ref{defn:symmetric_monoidal_trace}.

\begin{defn}[\cite{Ponto_thesis}]
\label{defn:Euler_characteristic_bicat}
The \KEYWORD{Euler characteristic} $\chi(M)$ of a right dualizable 1-cell $M \in \mathscr{B}(R, S)$ is the map
\[\begin{tikzcd}
	\sh{U_R} \arrow[r, "\sh{\eta}"]
	& \sh{M \odot M^*} \arrow[r, "\theta", "\cong"']
	& \sh{M^* \odot M} \arrow[r, "\varepsilon"]
	& \sh{U_S}.
\end{tikzcd}\]
\end{defn}

\begin{example}
\label{ex:character}
A $k$-linear representation of a group $G$ can be viewed as module over the group algebra $k[G]$; this is a 1-cell $k[G] \pto k$ in $\MORITABICATK$. When the underlying vector space of $V$ is finite-dimensional, $V$ is right dualizable (Example~\ref{ex:dualizable_bimodule}) and $\chi(V)$ is a map $\sh{k[G]} \to \sh{k}$. If we use $\HH_0$ as the shadow, the quotient
\[
	k[G] \to \HH_0(k[G], k[G])
\]
identifies pairs of group elements $gh$ and $hg$; equivalently, it identifies conjugate elements, and thus $\HH_0(k[G], k[G])$ is $k[\mathrm{cl}(G)]$, the free $k$-vector space on the conjugacy classes of $G$. Since $k$ is commutative, $\HH_0(k, k) \cong k$, and thus $\chi(V)$ is a linear map $k[\mathrm{cl}(G)] \to k$, which amounts to a class function $G \to k$. By describing a right dual, coevaluation, and evaluation similar to those in Example~\ref{ex:vector_spaces}, we compute that $\chi(V)$ is the character of $V$; that is,
\[
	\chi(V)(g) = \tr(V \xrightarrow{g \cdot} V).
\]
\end{example}

Definition~\ref{defn:Euler_characteristic_bicat} generalizes in two ways. First, we can introduce an endomorphism of the dualizable object, similar to the transition from Examples~\ref{ex:vector_space_Euler_characteristic} and \ref{ex:CW_Euler_characteristic} to Examples~\ref{ex:vector_space_trace} and \ref{ex:Lefschetz_number}, respectively. Second, we can twist the endomorphism by 1-cells $Q$ and $P$, making the trace a map $\sh{Q} \to \sh{P}$ rather than simply a map between shadows of unit 1-cells.

\begin{defn}[{\cite[Definition 4.5.1]{Ponto_thesis}}]
\label{defn:bicategorical_trace}
Let $f : Q \odot M \to M \odot P$ be a 2-cell where $M$ is right dualizable. The \KEYWORD{trace} of $f$ is the composite
\[\begin{tikzcd}[column sep=2.1em]
	\sh{Q} \arrow[r, "\sh{1 \odot \eta}"]
	& \sh{Q \odot M \odot M^*} \arrow[r, "\sh{f \odot 1}"]
	& \sh{M \odot P \odot M^*} \arrow[r, "\theta", "\cong"']
	& \sh{M^* \odot M \odot P} \arrow[r, "\sh{\varepsilon \odot 1}"]
	& \sh{P}.
\end{tikzcd}\]
\end{defn}

The Euler characteristic for a right dualizable 1-cell $M \in \mathscr{B}(R, S)$ is the trace of the canonical isomorphism 2-cell $U_R \odot M \xrightarrow{\cong} M \odot U_S$; this is the sense in which we think of the Euler characteristic as a ``trace of identity map.''

Suppose we have right dualizable 1-cells $M \in \mathscr{B}(R, S)$ and $N \in \mathscr{B}(S, T)$ and 1-cells $Q, P, L$ which twist endomorphisms of $M$ and $N$:
\[
	f : Q \odot M \to M \odot P \qquad
	g : P \odot N \to N \odot L.
\]
The traces of $f$ and $g$ are maps $\sh{Q} \to \sh{P}$ and $\sh{P} \to \sh{L}$. The following theorem says that we can obtain the composite $\sh{Q} \to \sh{P} \to \sh{L}$ as a single trace with respect to $M \odot N$, which is dualizable by Proposition~\ref{prop:composite_of_dual_pairs}.

\begin{thm}[{\cite[Proposition 7.5]{Shadows_and_traces}}]
\label{thm:composite_of_dual_pairs_traces}
Let $M, N, Q, P, L, f, g$ be as above. Then the trace of
\[
	Q \odot M \odot N \xrightarrow{f \odot \id_N} M \odot P \odot N \xrightarrow{\id_M \odot g} M \odot N \odot L
\]
is
\[
	\sh{Q} \xrightarrow{\tr(f)} \sh{P} \xrightarrow{\tr(g)} \sh{L}.
\]
\end{thm}

When applied to the isomorphisms
\[
	U_R \odot M \odot N \xrightarrow{\cong} M \odot U_S \odot N \xrightarrow{\cong} M \odot N \odot U_T
\]
this theorem gives the following.

\begin{cor}
\label{cor:composite_of_dual_pairs_Euler_characteristics}
If $M \in \mathscr{B}(R, S)$ and $N \in \mathscr{B}(S, T)$ are right dualizable then
\[
	\chi(M \odot N) = \chi(N) \circ \chi(M).
\]
\end{cor}

\begin{example}
\label{ex:res_and_ind_characters}
The preceding corollary can be used to obtain the formulas for characters of restricted and induced representations. Given a subgroup $H \leq G$ and a $G$-representation $V$, the restricted $H$-representation is
\[
	\Res_H^G(V) := {}_{\varphi}k[G] \otimes_{k[G]} V,
\]
where $\varphi : k[H] \to k[G]$ is the ring homomorphism induced by the inclusion of $H$ into $G$. The bimodule ${}_{\varphi}k[G]$ is right dualizable, since $(_{\varphi}S, S_{\varphi})$ is always a dual pair given a ring homomorphism $\varphi : R \to S$ \cite[Example 16.4.2]{MaySigurdsson}. By Corollary~\ref{cor:composite_of_dual_pairs_Euler_characteristics}, $\chi(\Res_H^G(H))$ equals the composite
\[
	\sh{k[H]} \xrightarrow{\chi({}_{\varphi}k[G])} \sh{k[G]} \xrightarrow{\chi(V)} \sh{k}.
\]
The first map simply takes $[h] \in \HH_0(k[H])$ to $[h] \in \HH_0(k[G])$, which is well-defined since elements which are conjugate in $H$ are also conjugate in $G$. Now, given an $H$-representation $W$, the induced $G$-representation is
\[
	\Ind_H^G(W) := k[G]_{\varphi} \otimes_{k[H]} V.
\]
If $[G : H] < \infty$, then $k[G]_{\varphi}$ is right dualizable with right dual $_{\varphi}k[G]$ (this is \emph{not} true for a general ring homomorphism $\varphi$). A coevaluation and evaluation for the dual pair $(k[G]_{\varphi}, {}_{\varphi}k[G])$ are given by
\[\begin{aligned}
	k[G] &\xrightarrow{\eta} k[G]_{\varphi} \otimes_{k[H]} {}_{\varphi}k[G] &\qquad\qquad {}_{\varphi}k[G] \otimes_{k[G]} k[G]_{\varphi} &\xrightarrow{\varepsilon} k[H] \\
	g &\mapsto \sum_{g_iH \in G/H} gg_i \otimes g_i^{-1} & g \otimes g' &\mapsto \begin{cases} gg', & gg' \in H \\ 0, & gg' \notin H \end{cases}
\end{aligned}\]
where the $g_i$ are a choice of coset representatives for $G/H$. By Corollary~\ref{cor:composite_of_dual_pairs_Euler_characteristics}, $\chi(\Ind_H^G(V))$ is equal to
\[
	\sh{k[G]} \xrightarrow{\chi(k[G]_{\varphi})} \sh{k[H]} \xrightarrow{\chi(V)} \sh{k}.
\]
The first map takes $[g]$ to
\[
	\sum_{g_iH \in G/H, \, g_i^{-1}gg_i \in H} [g_i^{-1}gg_i] = \frac{1}{|H|} \sum_{s \in G, \, s^{-1}gs \in H} [s^{-1}gs];
\]
then after applying $\chi(V)$ we get the usual character induction formula
\[
	\chi(\Ind_H^G(V))(g) = \frac{1}{|H|} \sum_{s \in G, \, s^{-1}gs \in H} \chi(V)(s^{-1}gs).
\]
\end{example}

\begin{defn}
\label{defn:mate}
Let $f : Q \odot M \to N \odot P$ be a 2-cell where $(M, M^*)$ and $(N, N^*)$ are dual pairs. The \KEYWORD{mate} of $f$ is the map $f^* : N^* \odot Q \to P \odot M^*$ given by
\[
	N^* \odot Q \xrightarrow{\id \odot \eta} N^* \odot Q \odot M \odot M^* \xrightarrow{\id \odot f \odot \id} N^* \odot N \odot P \odot M^* \xrightarrow{\varepsilon \odot \id} P \odot M^*.
\]
\end{defn}

There is a construction analogous to Definition~\ref{defn:bicategorical_trace} for a twisted endomorphism $M^* \odot Q \to P \odot M^*$ of a \emph{left} dualizable 1-cell $M^*$. Using this, we obtain the following.

\begin{prop}[{\cite[Lemma 4.5.3]{Ponto_thesis}}]
\label{prop:trace_of_mate}
Let $f : Q \odot M \to M \odot P$ be a 2-cell where $M$ is right dualizable. Then $\tr(f) = \tr(f^*)$.
\end{prop}

\subsection{String diagrams for bicategories}

String diagrams provide a useful graphical calculus for working with composites of many morphisms in monoidal categories or bicategories. They were first used by Penrose \cite{Penrose_negdimten, Penrose_spinors} and given a rigorous foundation by Joyal and Street \cite{JoyalStreet_geotencal, JoyalStreetVerity_tracedmoncat}. For a comprehensive overview of string diagrams in various kinds of monoidal categories, see \cite{Selinger_graphical}. Ponto and Shulman introduced string diagrams for shadowed bicategories in \cite{Shadows_and_traces} and rigorously justified their use, describing a way to assign a unique value to a string diagram which is invariant under deformation. Consequently, string diagrams can be used to prove results like Theorem~\ref{thm:composite_of_dual_pairs_traces} or Proposition~\ref{prop:trace_of_mate}; these pictorial proofs are usually easier and more illuminating than the alternative (a long string of equations or a large commutative diagram).

In a string diagram for a bicategory, 0-cells are represented by colored regions, 1-cells by strings, and 2-cells by nodes, as in Figure~\ref{fig:string_diagrams_bicat}. String diagrams are read from top to bottom; we think of each 2-cell as a ``machine,'' with its source 1-cell as the ``input'' strings feeding into it from above and its target 1-cell as the ``output'' strings emerging downward from it. The composite of 1-cells is formed by placing their strings next to each other; by coloring the region between the strings, we ensure that two 1-cells can only be composed if the target 0-cell of the first matches the source 0-cell of the second. We do not draw strings for unit 1-cells.

\begin{figure}[h]
	\centering
	\begin{tabular}{m{17mm}m{30mm}cm{17mm}m{20mm}}
	\begin{center}Object\\\ \\$R$\end{center} &
	\begin{center}\begin{tikzpicture}
		\fill[color1fill] (0,0) rectangle (2,2);
		\node[anchor=north west,color1dark,re] at (0,2) {$R$};
	\end{tikzpicture}\end{center}
	&&
	\begin{center}1-cell\\\ \\$R\xrightarrow{M} S$\end{center} &
	\begin{center}\begin{tikzpicture}
		\fill[color1fill] (0,0) rectangle (1,2);
		\fill[color2fill] (1,0) rectangle (2,2);
		\node[anchor=north west,color1dark,re] at (0,2) {$R$};
		\node[anchor=north east,color2dark,re] at (2,2) {$S$};
		\draw (1,2) -- node[ed] {$M$} +(0,-2);
	\end{tikzpicture}\end{center}
	\\\\
	\begin{center}Composite\\\ \\$M\odot N$\end{center} &
	\begin{center}\begin{tikzpicture}
		\fill[color1fill] (0,0) rectangle (1,2);
		\fill[color2fill] (1,0) rectangle (2,2);
		\fill[color3fill] (2,0) rectangle (3,2);
		\node[anchor=north west,color1dark,re] at (0,2) {$R$};
		\node[anchor=north,color2dark,re] at (1.5,2) {$S$};
		\node[anchor=north east,color3dark,re] at (3,2) {$T$};
		\draw (1,2) -- node[ed,swap] {$M$} +(0,-2);
		\draw (2,2) -- node[ed] {$N$} +(0,-2);
	\end{tikzpicture}\end{center}
	&&
	\begin{center}2-cell\\$\xymatrix{R\rtwocell^{M}_N{f} & S}$\end{center} &
	\begin{center}\begin{tikzpicture}
		\fill[color1fill] (0,0) rectangle (1,2);
		\fill[color2fill] (1,0) rectangle (2,2);
		\node[anchor=north west,color1dark,re] at (0,2) {$R$};
		\node[anchor=north east,color2dark,re] at (2,2) {$S$};
		\node[fill=white,draw,circle,inner sep=1pt] (f) at (1,1) {$f$};
		\draw (1,2) -- node[ed] {$M$} (f) -- node[ed] {$N$} +(0,-1);
	\end{tikzpicture}\end{center}
\end{tabular}
	\caption{String diagrams for bicategories}
	\label{fig:string_diagrams_bicat}
\end{figure}
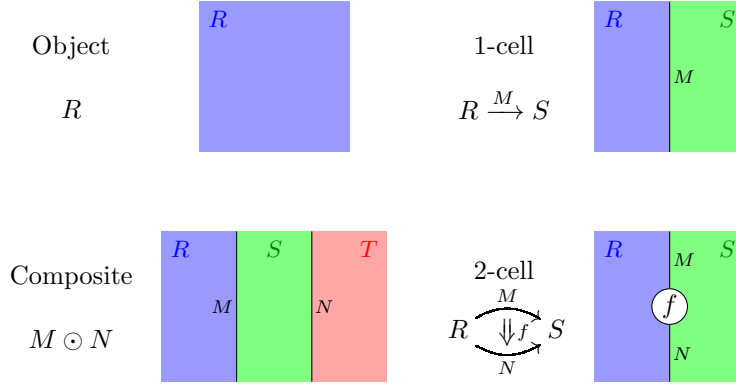

The coevaluation $\eta : U_R \to M \odot M^*$ and evaluation $\varepsilon : M^* \odot M \to U_S$ for a dual pair $(M, M^*)$ (Definition~\ref{defn:dual_pair_bicat}) are shown in Figure~\ref{fig:dual_pair}, and the triangle identities (see Definition~\ref{defn:dual_pair_sm}) in Figure~\ref{fig:triangle_identities}. As the figures show, we usually omit the nodes labeled $\eta$ and $\varepsilon$, so that it looks like the string is simply making a $180^{\circ}$ turn; we think of an $M^*$ string as a ``wrong-way'' $M$ string (i.e.~traveling upward instead of downward). With this convention, the triangle identities essentially say that bent strings can be straightened out.

\begin{figure}[h]
	\centering
	\begin{subfigure}[b]{0.5\textwidth}
	\centering
	\begin{tabular}{m{26mm}m{10mm}m{26mm}}
		\begin{center}\begin{tikzpicture}
			\fill[color1fill] (0, 0) rectangle (2.6, 2.6);
			\filldraw[color2fill] (0.8, 0) -- node[ed]{$M$} ++(0, 1.2) arc(180:90:0.5) node[vert]{$\eta$} arc(90:0:0.5) -- node[ed]{$M^*$} ++(0, -1.2);
			\node[anchor=north west,color1dark,re] at (0, 2.6) {$R$};
			\node[anchor=south,color2dark,re] at (1.3, 0) {$S$};
		\end{tikzpicture}\end{center}
		& \begin{center} or \end{center} &
		\begin{center}\begin{tikzpicture}
			\fill[color1fill] (0, 0) rectangle (2.6, 2.6);
			\filldraw[color2fill] (0.8, 0) -- node[ed]{$M$} ++(0, 1.2) arc(180:0:0.5) -- node[ed]{$M^*$} ++(0, -1.2);
			\node[anchor=north west,color1dark,re] at (0, 2.6) {$R$};
			\node[anchor=south,color2dark,re] at (1.3, 0) {$S$};
		\end{tikzpicture}\end{center}
	\end{tabular}
	\caption{Coevaluation $\eta : U_R \to M \odot M^*$}
\end{subfigure}
\par\bigskip
\begin{subfigure}[b]{0.5\textwidth}
	\centering
	\begin{tabular}{m{26mm}m{10mm}m{26mm}}
		\begin{center}\begin{tikzpicture}
			\fill[color2fill] (0, 0) rectangle (2.6, 2.6);
			\filldraw[color1fill] (0.8, 2.6) -- node[ed,swap]{$M^*$} ++(0, -1.2) arc(-180:-90:0.5) node[vert]{$\varepsilon$} arc(-90:0:0.5) -- node[ed,swap]{$M$} ++(0, 1.2);
			\node[anchor=south west,color2dark,re] at (0, 0) {$S$};
			\node[anchor=north,color1dark,re] at (1.3, 2.6) {$R$};
		\end{tikzpicture}\end{center}
		& \begin{center} or \end{center} &
		\begin{center}\begin{tikzpicture}
			\fill[color2fill] (0, 0) rectangle (2.6, 2.6);
			\filldraw[color1fill] (0.8, 2.6) -- node[ed,swap]{$M^*$} ++(0, -1.2) arc(-180:0:0.5) -- node[ed,swap]{$M$} ++(0, 1.2);
			\node[anchor=south west,color2dark,re] at (0, 0) {$S$};
			\node[anchor=north,color1dark,re] at (1.3, 2.6) {$R$};
		\end{tikzpicture}\end{center}
	\end{tabular}
	\caption{Evaluation $\varepsilon : M^* \odot M \to U_S$}
\end{subfigure}
	\caption{The coevaluation and evaluation for a dual pair}
	\label{fig:dual_pair}
\end{figure}
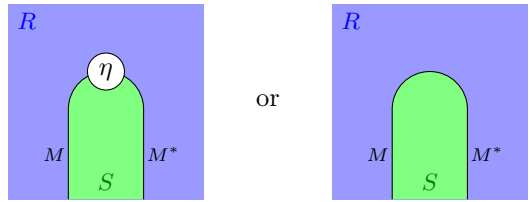
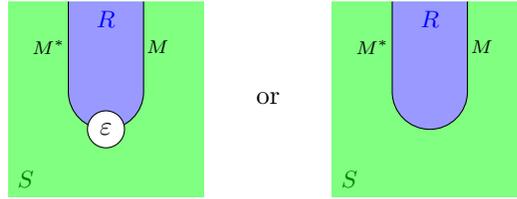

\begin{figure}[h]
	\centering
	\begin{tabular}{m{36mm}m{4mm}m{20mm}}
	\begin{center}\begin{tikzpicture}
		\clip (0, 0) -- (3.6, 0) -- (3.6, 3.6) -- (0, 3.6);
		\fill[color1fill] (0, 0) rectangle (3.6, 3.6);
		\filldraw[color2fill] (0.8, 0) -- ++(0, 1.4) -- node[ed]{$M$} ++(0, 0.8) arc(180:0:0.5) -- node[ed]{$M^*$} ++(0, -0.8) arc(-180:0:0.5) -- node[ed,swap]{$M$} ++(0, 0.8) -- (2.8, 3.7) -- (3.7, 3.7) -- (3.7, -0.1) -- (-0.1, -0.1);
		\node[anchor=north west,color1dark,re] at (0, 3.6) {$R$};
		\node[anchor=south east,color2dark,re] at (3.6, 0) {$S$};
	\end{tikzpicture}\end{center}
	& \begin{center} = \end{center} &
	\begin{center}\begin{tikzpicture}
		\fill[color1fill] (0, 0) rectangle (1, 3.6);
		\fill[color2fill] (1, 0) rectangle (2, 3.6);
		\draw (1, 0) -- node[ed]{$M$} (1, 3.6);
		\node[anchor=north west,color1dark,re] at (0, 3.6) {$R$};
		\node[anchor=south east,color2dark,re] at (2, 0) {$S$};
	\end{tikzpicture}\end{center}
	\\ \\
	\begin{center}\begin{tikzpicture}
		\clip (0, 0) -- (3.6, 0) -- (3.6, 3.6) -- (0, 3.6);
		\fill[color2fill] (0, 0) rectangle (3.6, 3.6);
		\filldraw[color1fill] (0.8, 3.6) -- ++(0, -1.4) -- node[ed,swap]{$M^*$} ++(0, -0.8) arc(-180:0:0.5) -- node[ed]{$M$} ++(0, 0.8) arc(180:0:0.5) -- node[ed]{$M^*$} ++(0, -0.8) -- (2.8, -0.1) -- (3.7, -0.1) -- (3.7, 3.7) -- (-0.1, 3.7);
		\node[anchor=north west,color2dark,re] at (0, 3.6) {$S$};
		\node[anchor=south east,color1dark,re] at (3.6, 0) {$R$};
	\end{tikzpicture}\end{center}
	& \begin{center} = \end{center} &
	\begin{center}\begin{tikzpicture}
		\fill[color2fill] (0, 0) rectangle (1, 3.6);
		\fill[color1fill] (1, 0) rectangle (2, 3.6);
		\draw (1, 0) -- node[ed]{$M^*$} (1, 3.6);
		\node[anchor=north west,color2dark,re] at (0, 3.6) {$S$};
		\node[anchor=south east,color1dark,re] at (2, 0) {$R$};
	\end{tikzpicture}\end{center}
\end{tabular}
	\caption{The triangle identities for the dual pair $(M, M^*)$}
	\label{fig:triangle_identities}
\end{figure}
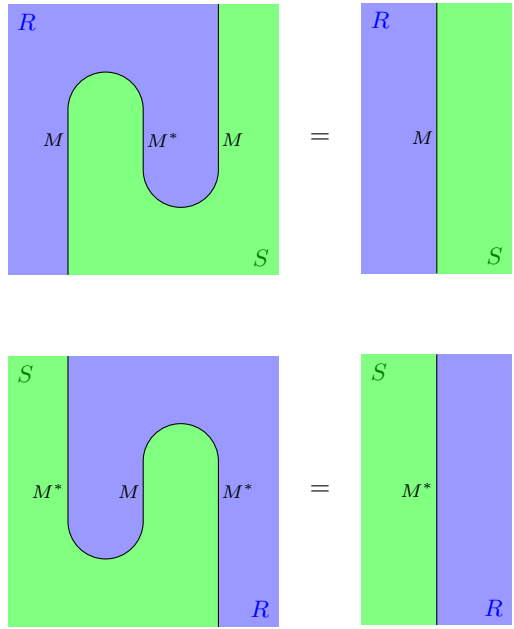

Shadows are represented graphically by closing up planar string diagrams into cylinders, as in Figure~\ref{fig:shadow}. The cyclicity isomorphism $\theta_{M,N} : \sh{M \odot N} \xrightarrow{\cong} \sh{N \odot M}$ is represented by allowing either the $M$ or $N$ string to travel around the back of the cylinder (it makes no difference which, since $\theta_{M,N} = \theta_{N,M}^{-1}$). A string diagram of the bicategorical trace (Definition~\ref{defn:bicategorical_trace}) appears in Figure~\ref{fig:bicat_trace}.

\begin{figure}[h]
	\centering
	\begin{subfigure}[b]{0.3\textwidth}
	\centering
	\begin{tikzpicture}
		\bgcylinder{0,0}{2}{1.2}{.3}{color1}{color1}
		\node[anchor=south west,color1dark,re] at (dl) {$R$};
		\draw (top) -- node[ed] {$M$} (bot);
	\end{tikzpicture}
	\caption{Shadow $\sh{M}$}
\end{subfigure}
\begin{subfigure}[b]{0.3\textwidth}
	\centering
	\begin{tikzpicture}
		\bgcylinder{0,0}{3.7}{1.2}{.3}{color2}{color1}
		\begin{pgfonlayer}{foreground}
			\drawtheta{(0,2.1)}{0}{1}{th}{}
		\end{pgfonlayer}
		\node[anchor=south east,color2dark,re] at (dr) {$S$};
		\filldraw[color1fill] let \p1 = ($(ul)!.65!(ur)$) in
			(thR) -- ++(0.1,0) -- (ur') -- ($(ul)!.65!(ur)$) -- (\x1, 3.1) to[out=-90,in=90] node[ed] {$N$} (thR);
		\filldraw[color1fill] let \p1 = ($(dl)!.35!(dr)$), \p2 = ($(dl)!.65!(dr)$) in
			(thL) to[out=-90,in=90] (\x1, 0.0) -- (\p1 |- bot') -- (\p2 |- bot') -- (\x2, 0.0) to[out=90,in=-90] node[ed,near start] {$M$} (\x1, 3.29) -- ($(ul)!.35!(ur)$) -- (ul') -- (ul' |- thL) -- (thL);
		\node[anchor=north west,color1dark,re] at ($(ul)!.1!(dl)$) {$R$};
	\end{tikzpicture}
	\caption{Cyclicity $\theta_{M,N}$}
\end{subfigure}
	\caption{String diagrams for shadows}
	\label{fig:shadow}
\end{figure}

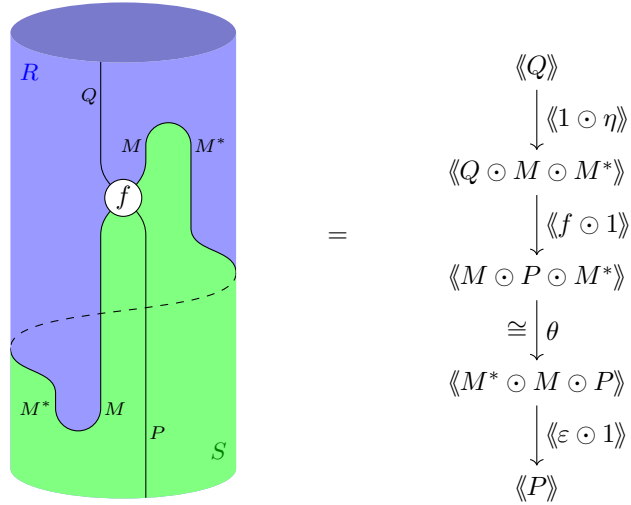
\begin{figure}[h]
	\centering
	\begin{tabular}{m{30mm}m{20mm}m{25mm}}
	\begin{tikzpicture}
		\bgcylinder{0,0}{5.8}{1.5}{.4}{color1}{color1}

		\coordinate (f) at (1.5, 3.6) {};
		\path (f) -- ++(-0.3, 0.5) coordinate (ful);
		\path (f) -- ++(0.3, 0.5) coordinate (fur);
		\path (f) -- ++(-0.3, -0.5) coordinate (fdl);
		\path (f) -- ++(0.3, -0.5) coordinate (fdr);

		\begin{pgfonlayer}{foreground}
			\drawtheta{(f)}{-1.0}{1}{th}{}
		\end{pgfonlayer}

		\filldraw[color2fill] (f) to[out=-135, in=90] (fdl) -- (fdl |- thL) -- node[ed,pos=1]{$M$} ++(0, -0.8) arc (0:-180:0.3) -- node[ed,pos=0]{$M^*$} ++(0, 0.2) to[out=90, in=-90] (thL) -- (thL') -- (thL' |- bot') -- (thR' |- bot') -- (thR') -- (thR) to[out=90, in=-90] ++(-0.6, 0.6) -- ($(fur) + (0.6, 0)$) -- node[ed,swap,pos=1]{$M^*$} ++(0, 0.2) arc(0:180:0.3) -- node[ed,swap,pos=0]{$M$} (fur) to[out=-90, in=45] (f);
		\draw (f) to[out=135, in=-90] (ful) -- node[ed]{$Q$} (ful |- ul);
		\draw (f) to[out=-45, in=90] (fdr) -- node[ed,near end]{$P$} (fdr |- bot);

		\node[vert] at (f) {$f$};
		\node[anchor=north west,color1dark,re] at ($(ul)!.05!(dl)$) {$R$};
		\node[anchor=south east,color2dark,re] at (dr) {$S$};
	\end{tikzpicture}
	& \begin{center} = \end{center} &
	\begin{tikzpicture}
		\node at (0, 6.5) {}; 
		\node (A) at (0, 5.6) {$\sh{Q}$};
		\node (B) at (0, 4.2) {$\sh{Q \odot M \odot M^*}$};
		\node (C) at (0, 2.8) {$\sh{M \odot P \odot M^*}$};
		\node (D) at (0, 1.4) {$\sh{M^* \odot M \odot P}$};
		\node (E) at (0, 0) {$\sh{P}$};

		\draw[->] (A) -- node[right]{$\sh{1 \odot \eta}$} (B);
		\draw[->] (B) -- node[right]{$\sh{f \odot 1}$} (C);
		\draw[->] (C) -- node[right]{$\theta$} node[left]{$\cong$} (D);
		\draw[->] (D) -- node[right]{$\sh{\varepsilon \odot 1}$} (E);
	\end{tikzpicture}
\end{tabular}
	\caption{The bicategorical trace}
	\label{fig:bicat_trace}
\end{figure}

\section{Closed bicategories}
\label{sec:closed_bicategories}

The bicategory $\MORITABICAT$ has additional structure that we have not yet made use of, namely the hom-functors. Axiomatizing this structure is essential for generalizing notions of \emph{cotrace} that appear in the literature, just as the symmetric monoidal and bicategorical trace generalize many familiar examples.

\begin{defn}[{\cite[Definition 16.3.1]{MaySigurdsson}}]
\label{defn:closed_bicategory}
A \KEYWORD{(left and right) closed bicategory} is a bicategory $\mathscr{B}$ equipped with left and right internal hom-functors
\[
	- \triangleleft - : \mathscr{B}(R, T) \times \mathscr{B}(R, S)^{\op} \to \mathscr{B}(S, T)
\]
and
\[
	- \triangleright - : \mathscr{B}(S, T)^{\op} \times \mathscr{B}(R, T) \to \mathscr{B}(R, S)
\]
for all triples of 0-cells $R, S, T$ and natural isomorphisms
\begin{equation}
\label{eq:closed_bicat_tensor_hom}
	\mathscr{B}(S, T)(N, P \triangleleft M) \cong \mathscr{B}(R, T)(M \odot N, P) \cong \mathscr{B}(R, S)(M, N \triangleright P)
\end{equation}
for all triples of 1-cells $M : R \pto S$, $N : S \pto T$, and $P : R \pto T$.
\end{defn}

Instead of the functors of two variables $- \triangleright -$, it is enough to ask for a right adjoint $N \triangleright -$ to the functor $- \odot N$ for each 1-cell $N$; these right adjoints automatically assemble into unique functors $- \triangleright -$ making the second isomorphism of (\ref{eq:closed_bicat_tensor_hom}) natural in $N$ (as well as $M$ and $P$). Similar remarks apply to $- \triangleleft -$.

Our most frequently used example of a bicategory, $\MORITABICAT$ (or $\MORITABICATK$), is a closed bicategory. For bimodules $_RM_S$ and $_RP_T$, the left internal hom-object is the $(S, T)$-bimodule
\[
	P \triangleleft M := \Hom_R(M, P),
\]
and for bimodules $_SN_T$ and $_RP_T$, the right internal hom-object is the $(R, S)$-bimodule
\[
	N \triangleright P := \Hom_T(N, P).
\]
We remember the hom-functors for bimodules by noting that triangle ($\triangleleft$ or $\triangleright$) always points from source to target, and the direction it points indicates on which side the maps are linear (e.g.~the \emph{right}-pointing triangle $\triangleright$ indicates that $N \triangleright P$ is the set of \emph{right} $T$-linear maps from $N$ to $P$).

There are \KEYWORD{evaluation} maps
\[
	(N \triangleright P) \odot N \xrightarrow{\ev} P \qquad \text{and} \qquad M \odot (P \triangleleft M) \xrightarrow{\ev} P
\]
(the transposes of $\id_{N \triangleright P}$ and $\id_{P \triangleleft M}$, respectively), which are natural in $P$ and extranatural in $N$ and $M$ (respectively). For bimodules, these are the familiar evaluation maps $\varphi \otimes n \mapsto \varphi(n)$ and $m \otimes \psi \mapsto \psi(m)$. Similarly, there are \KEYWORD{coevaluation} maps
\[
	M \xrightarrow{\coev} N \triangleright (M \odot N) \qquad \text{and} \qquad N \xrightarrow{\coev} (M \odot N) \triangleleft M
\]
(the transposes of $\id_{M \odot N}$), which are natural in $M$ and $N$ (respectively) and extranatural in $N$ and $M$ (respectively). There is a natural isomorphism
\[\begin{tikzcd}
	(M \odot N) \triangleright P
		\arrow[r, "t", "\cong"']
	& M \triangleright (N \triangleright P)
\end{tikzcd}\]
(which we call $t$ for ``transpose'' or ``tensor-hom adjunction'') given by the transpose of the transpose of
\[
	((M \odot N) \triangleright P) \odot M \odot N \xrightarrow{\ev} P
\]
and whose inverse is the transpose of
\[
	(M \triangleright (N \triangleright P)) \odot M \odot N \xrightarrow{\ev \odot 1} (N \triangleright P) \odot N \xrightarrow{\ev} P.
\]

There is a similar natural isomorphism for $\triangleleft$, as well as a natural isomorphism
\[\begin{tikzcd}
	(N \triangleright P) \triangleleft M
		\arrow[r, "a", "\cong"']
	& N \triangleright (P \triangleleft M),
\end{tikzcd}\]
which we call $a$ for ``associator.''

One of the axioms of a bicategory is a pentagon ensuring that any two ways of reparenthesizing four composed 1-cells through associators are equal. In a closed bicategory, there are several more associativity pentagons---not axioms but rather provably commuting diagrams---since there are now three ways to join two 1-cells together ($\odot$, $\triangleleft$, and $\triangleright$). We describe some of them below since we will often need them later; note that they come in pairs since there are both left and right internal hom-functors.

\begin{lem}
\label{lem:pentagon_a_and_t_1}
In a closed bicategory, the following diagrams commute for any 1-cells $W, X, Y, Z$ for which the diagrams make sense:
\[
\begin{tikzpicture}[xscale=1.8]
	\newcommand{\RADIUS}{2}
	\node(A) at ({\RADIUS*cos(90+0*72)}, {\RADIUS*sin(90+0*72)}){$((W \odot X) \odot Y) \triangleright Z$};
	\node(B) at ({\RADIUS*cos(90+1*72)}, {\RADIUS*sin(90+1*72)}){$(W \odot (X \odot Y)) \triangleright Z$};
	\node(C) at ({\RADIUS*cos(90+2*72)}, {\RADIUS*sin(90+2*72)}){$W \triangleright ((X \odot Y) \triangleright Z)$};
	\node(D) at ({\RADIUS*cos(90+3*72)}, {\RADIUS*sin(90+3*72)}){$W \triangleright (X \triangleright (Y \triangleright Z))$};
	\node(E) at ({\RADIUS*cos(90+4*72)}, {\RADIUS*sin(90+4*72)}){$(W \odot X) \triangleright (Y \triangleright Z)$};
	\draw [->] (A) -- node[above right]{$t$} node[below left]{$\cong$} (E);
	\draw [->] (A) -- node[above left]{$(a^{-1})^*$} node[below right]{$\cong$} (B);
	\draw [->] (B) -- node[below left]{$t$} node[above right]{$\cong$} (C);
	\draw [->] (C) -- node[below]{$t_*$} node[above]{$\cong$} (D);
	\draw [->] (E) -- node[below right]{$t$} node[above left]{$\cong$} (D);
\end{tikzpicture}
\]
\[
\begin{tikzpicture}[xscale=1.8]
	\newcommand{\RADIUS}{2}
	\node(A) at ({\RADIUS*cos(90+0*72)}, {\RADIUS*sin(90+0*72)}){$((W \triangleleft X) \triangleleft Y) \triangleleft Z$};
	\node(B) at ({\RADIUS*cos(90+1*72)}, {\RADIUS*sin(90+1*72)}){$(W \triangleleft (X \odot Y)) \triangleleft Z$};
	\node(C) at ({\RADIUS*cos(90+2*72)}, {\RADIUS*sin(90+2*72)}){$W \triangleleft ((X \odot Y) \odot Z)$};
	\node(D) at ({\RADIUS*cos(90+3*72)}, {\RADIUS*sin(90+3*72)}){$W \triangleleft (X \odot (Y \odot Z))$};
	\node(E) at ({\RADIUS*cos(90+4*72)}, {\RADIUS*sin(90+4*72)}){$(W \triangleleft X) \triangleleft (Y \odot Z)$};
	\draw [->] (A) -- node[above right]{$t^{-1}$} node[below left]{$\cong$} (E);
	\draw [->] (A) -- node[above left]{$t^{-1}_*$} node[below right]{$\cong$} (B);
	\draw [->] (B) -- node[below left]{$t^{-1}$} node[above right]{$\cong$} (C);
	\draw [->] (C) -- node[below]{$(a^{-1})^*$} node[above]{$\cong$} (D);
	\draw [->] (E) -- node[below right]{$t^{-1}$} node[above left]{$\cong$} (D);
\end{tikzpicture}
\]
\end{lem}

The reader is warned that the two diagrams above will appear as squares rather than pentagons when we make use of them later on, since we typically suppress the associator for $\odot$. To verify that diagrams like these commute, the easiest approach is usually to take transposes until no hom-objects remain in the terminal object of the diagram. For example, each side of the first pentagon above is thrice transposed by tensoring with the identity map on $W \odot X \odot Y$ and then composing with
\[
	(W \triangleright (X \triangleright (Y \triangleright Z))) \odot W \odot X \odot Y \xrightarrow{\ev \odot 1^2} (X \triangleright (Y \triangleright Z)) \odot X \odot Y \xrightarrow{\ev \odot 1} (Y \triangleright Z) \odot Y \xrightarrow{\ev} Z.
\]
Making use of naturality of $\ev$ and the definitions of $a$ and $t$ in terms of their transposes, one simplifies both sides of the pentagon (after transposing, that is) to the same map.

\begin{lem}
\label{lem:pentagon_a_and_t_2}
In a closed bicategory, the following diagrams commute for any 1-cells $W, X, Y, Z$ for which the diagrams make sense:
\[
\begin{tikzpicture}[xscale=1.8]
	\newcommand{\RADIUS}{2}
	\node(A) at ({\RADIUS*cos(90+0*72)}, {\RADIUS*sin(90+0*72)}){$((W \odot X) \triangleright Y) \triangleleft Z$};
	\node(B) at ({\RADIUS*cos(90+1*72)}, {\RADIUS*sin(90+1*72)}){$(W \triangleright (X \triangleright Y)) \triangleleft Z$};
	\node(C) at ({\RADIUS*cos(90+2*72)}, {\RADIUS*sin(90+2*72)}){$W \triangleright ((X \triangleright Y) \triangleleft Z)$};
	\node(D) at ({\RADIUS*cos(90+3*72)}, {\RADIUS*sin(90+3*72)}){$W \triangleright (X \triangleright (Y \triangleleft Z))$};
	\node(E) at ({\RADIUS*cos(90+4*72)}, {\RADIUS*sin(90+4*72)}){$(W \odot X) \triangleright (Y \triangleleft Z)$};
	\draw [->] (A) -- node[above right]{$a$} node[below left]{$\cong$} (E);
	\draw [->] (A) -- node[above left]{$t_*$} node[below right]{$\cong$} (B);
	\draw [->] (B) -- node[below left]{$a$} node[above right]{$\cong$} (C);
	\draw [->] (C) -- node[below]{$a_*$} node[above]{$\cong$} (D);
	\draw [->] (E) -- node[below right]{$t$} node[above left]{$\cong$} (D);
\end{tikzpicture}
\]
\[
\begin{tikzpicture}[xscale=1.8]
	\newcommand{\RADIUS}{2}
	\node(A) at ({\RADIUS*cos(90+0*72)}, {\RADIUS*sin(90+0*72)}){$((W \triangleright X) \triangleleft Y) \triangleleft Z)$};
	\node(B) at ({\RADIUS*cos(90+1*72)}, {\RADIUS*sin(90+1*72)}){$(W \triangleright (X \triangleleft Y)) \triangleleft Z)$};
	\node(C) at ({\RADIUS*cos(90+2*72)}, {\RADIUS*sin(90+2*72)}){$W \triangleright ((X \triangleleft Y) \triangleleft Z)$};
	\node(D) at ({\RADIUS*cos(90+3*72)}, {\RADIUS*sin(90+3*72)}){$W \triangleright (X \triangleleft (Y \odot Z))$};
	\node(E) at ({\RADIUS*cos(90+4*72)}, {\RADIUS*sin(90+4*72)}){$(W \triangleright X) \triangleleft (Y \odot Z)$};
	\draw [->] (A) -- node[above right]{$t^{-1}$} node[below left]{$\cong$} (E);
	\draw [->] (A) -- node[above left]{$a_*$} node[below right]{$\cong$} (B);
	\draw [->] (B) -- node[below left]{$a$} node[above right]{$\cong$} (C);
	\draw [->] (C) -- node[below]{$t^{-1}_*$} node[above]{$\cong$} (D);
	\draw [->] (E) -- node[below right]{$a$} node[above left]{$\cong$} (D);
\end{tikzpicture}
\]
\end{lem}

Just as there are natural isomorphisms $(M \triangleright N) \triangleleft P \cong M \triangleright (N \triangleleft P)$ analogous to the associators $(M \odot N) \odot P \cong M \odot (N \odot P)$, there are natural isomorphisms $U_S \triangleright M \cong M \cong M \triangleleft U_R$ analogous to the unitors $U_R \odot M \cong M \cong M \odot U_S$.

\begin{lem}
\label{lem:hom_unit_isomorphisms}
Given a 1-cell $M : R \pto S$, the transpose $\overline{l} : M \to M \triangleleft U_R$ of $l : U_R \odot M \xrightarrow{\cong} M$ is an isomorphism, as is the transpose $\overline{r} : M \to U_S \triangleright M$ of $r : M \odot U_S \xrightarrow{\cong} M$. Moreover, these isomorphisms are natural in $M$.
\end{lem}

\begin{proof}
The inverse of $\overline{l}$ is
\[\begin{tikzcd}
	M \triangleleft U_R
		\arrow[r, "l^{-1}", "\cong"']
	& U_R \odot (M \triangleleft U_R)
		\arrow[r, "\mathrm{ev}"]
	& M.
\end{tikzcd}\]
The inverse of $\overline{r}$ is similar, and naturality is straightforward to check.
\end{proof}

These maps $\overline{l}$ and $\overline{r}$ satisfy properties analogous to the bicategory axiom relating the associator and unitors.

\begin{lem}
\label{lem:hom_unit_iso_compatibility}
The following diagrams commute whenever they make sense:
\begin{enumerate}[ref={\thecor.\arabic*}]
\item \label{lem:hom_unit_iso_compatibility_with_t}
\[\begin{tikzcd}
	X \triangleright Y
		\arrow[r, "r^*", "\cong"']
		\arrow[dr, "\overline{r}_*"', "\cong"]
	& (X \odot U_S) \triangleright Y
		\arrow[d, "t", "\cong"']
	\\
	& X \triangleright (U_S \triangleright Y)
\end{tikzcd} \qquad\qquad \begin{tikzcd}
	Y \triangleleft X
		\arrow[r, "l^*", "\cong"']
		\arrow[dr, "\overline{l}_*"', "\cong"]
	& Y \triangleleft (U_R \odot X)
		\arrow[d, "t", "\cong"']
	\\
	& (Y \triangleleft U_R) \triangleleft X
\end{tikzcd}\]
\item \label{lem:hom_unit_iso_compatibility_with_t_2}
\[\begin{tikzcd}
	X \triangleright Y
		\arrow[r, "l^*", "\cong"']
		\arrow[dr, "\overline{r}"', "\cong"]
	& (U_R \odot X) \triangleright Y
		\arrow[d, "t", "\cong"']
	\\
	& U_R \triangleright (X \triangleright Y)
\end{tikzcd} \qquad\qquad \begin{tikzcd}
	Y \triangleleft X
		\arrow[r, "r^*", "\cong"']
		\arrow[dr, "\overline{l}"', "\cong"]
	& Y \triangleleft (X \odot U_S)
		\arrow[d, "t", "\cong"']
	\\
	& (Y \triangleleft X) \triangleleft U_S
\end{tikzcd}\]
\item \label{lem:hom_unit_iso_compatibility_with_a}
\[\begin{tikzcd}
	X \triangleright Y
		\arrow[r, "\overline{l}_*", "\cong"']
		\arrow[dr, "\overline{l}"', "\cong"]
	& X \triangleright (Y \triangleleft U)
		\arrow[d, "a^{-1}", "\cong"']
	\\
	& (X \triangleright Y) \triangleleft U
\end{tikzcd} \qquad \qquad \begin{tikzcd}
	Y \triangleleft X
		\arrow[r, "\overline{r}_*", "\cong"']
		\arrow[dr, "\overline{r}"', "\cong"]
	& (U \triangleright Y) \triangleleft X
		\arrow[d, "a", "\cong"']
	\\
	& U \triangleright (Y \triangleleft X)
\end{tikzcd}\]
\end{enumerate}
\end{lem}

\begin{lem}
\label{lem:hom_unit_iso_equality}
For any 1-cell $M : R \pto S$, the maps $\overline{r}, \overline{r}_* : U_S \triangleright M \to U_S \triangleright (U_S \triangleright M)$ are equal, as are the maps $\overline{l}, \overline{l}_* : M \triangleleft U_R \to (M \triangleleft U_R) \triangleleft U_R$.
\end{lem}

\subsection{String diagrams for closed bicategories}

We represent hom-objects in string diagrams as shown in Figure~\ref{fig:hom_objects}. This is similar to how we join strings $M$ and $N$ together to represent their composite $M \odot N$ in Figure~\ref{fig:string_diagrams_bicat}, but with three key differences: (1) the 0-cell regions on either side of the source 1-cell (i.e.~the $M$ or $N$ in Figure~\ref{fig:hom_objects}) are interchanged, (2) the 1-cells appear in reverse order (e.g.~for $N \triangleright P$, the $P$ string appears to the left of the $N$ string), and (3) the string of the target 1-cell (i.e.~$P$) is drawn thicker.

\begin{rmk}
\label{rmk:drawing_homs}
The first two of these observations are explained by the fact that $N \triangleright P \cong P \odot N^*$ if $N$ is right dualizable (Corollary~\ref{cor:hom_from_dualizable}); thus even if $N$ does not have a right dual $N^*$, we draw the string diagram for $N \triangleright P$ nearly the same way that we would draw $P \odot N^*$ if it did.
\end{rmk}

\begin{figure}[h]
	\centering
	\begin{tabular}{m{13mm}m{20mm}cm{13mm}m{20mm}cm{13mm}m{20mm}}
	\begin{center}1-cell \\\ \\ $R \xrightarrow{M} S$ \end{center} &
	\begin{center}\begin{tikzpicture}
		\fill[color1fill] (0,0) rectangle (1,2);
		\fill[color2fill] (1,0) rectangle (2,2);
		\node[anchor=north west,color1dark,re] at (0,2) {$R$};
		\node[anchor=north east,color2dark,re] at (2,2) {$S$};
		\draw (1,2) -- node[ed] {$M$} +(0,-2);
	\end{tikzpicture}\end{center}
	&&
	\begin{center}1-cell \\\ \\ $R \xrightarrow{P} T$ \end{center} &
	\begin{center}\begin{tikzpicture}
		\fill[color1fill] (0,0) rectangle (1,2);
		\fill[color3fill] (1,0) rectangle (2,2);
		\node[anchor=north west,color1dark,re] at (0,2) {$R$};
		\node[anchor=north east,color3dark,re] at (2,2) {$T$};
		\draw (1,2) -- node[ed] {$P$} +(0,-2);
	\end{tikzpicture}\end{center}
	&&
	\begin{center}1-cell \\\ \\ $S \xrightarrow{N} T$ \end{center} &
	\begin{center}\begin{tikzpicture}
		\fill[color2fill] (0,0) rectangle (1,2);
		\fill[color3fill] (1,0) rectangle (2,2);
		\node[anchor=north west,color2dark,re] at (0,2) {$S$};
		\node[anchor=north east,color3dark,re] at (2,2) {$T$};
		\draw (1,2) -- node[ed] {$N$} +(0,-2);
	\end{tikzpicture}\end{center}
\end{tabular}
\begin{tabular}{m{20mm}m{30mm}cm{20mm}m{30mm}}
	\begin{center}Hom-object \\\ \\ $P \triangleleft M$ \end{center} &
	\begin{center}\begin{tikzpicture}
		\fill[color2fill] (0,0) rectangle (1,2);
		\fill[color1fill] (1,0) rectangle (2,2);
		\fill[color3fill] (2,0) rectangle (3,2);
		\node[anchor=north west,color2dark,re] at (0,2) {$S$};
		\node[anchor=north,color1dark,re] at (1.5,2) {$R$};
		\node[anchor=north east,color3dark,re] at (3,2) {$T$};
		\draw (1,2) -- node[ed,swap] {$M$} +(0,-2);
		\draw[ds] (2,2) -- node[ed] {$P$} +(0,-2);
	\end{tikzpicture}\end{center}
	&&
	\begin{center}Hom-object \\\ \\ $N \triangleright P$ \end{center} &
	\begin{center}\begin{tikzpicture}
		\fill[color1fill] (0,0) rectangle (1,2);
		\fill[color3fill] (1,0) rectangle (2,2);
		\fill[color2fill] (2,0) rectangle (3,2);
		\node[anchor=north west,color1dark,re] at (0,2) {$R$};
		\node[anchor=north,color3dark,re] at (1.5,2) {$T$};
		\node[anchor=north east,color2dark,re] at (3,2) {$S$};
		\draw[ds] (1,2) -- node[ed,swap] {$P$} +(0,-2);
		\draw (2,2) -- node[ed] {$N$} +(0,-2);
	\end{tikzpicture}\end{center}
\end{tabular}
	\caption{String diagrams for hom-objects}
	\label{fig:hom_objects}
\end{figure}
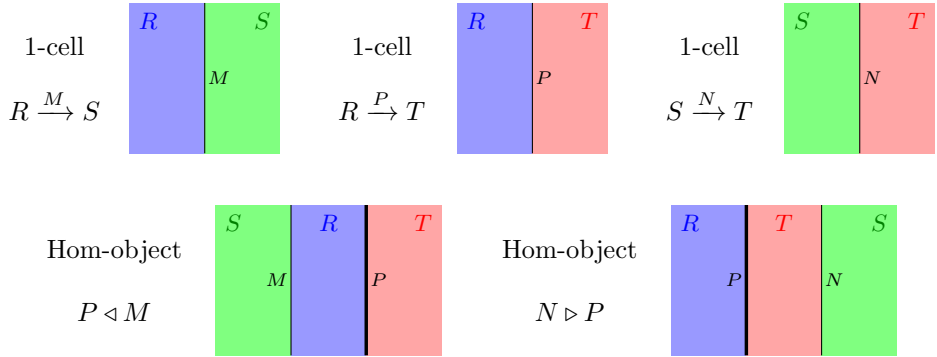

In general, a string diagram for a closed bicategory will have a single distinguished string, drawn thicker than the others, and all the other strings will have the 0-cell regions on their left and right swapped. Each string diagram can be interpreted as one of several different but isomorphic objects, just like in an ordinary bicategory, where, for example, a series of strings labeled $X$, $Y$, and $Z$ could represent either $(X \odot Y) \odot Z$ or $X \odot (Y \odot Z)$. Before describing how to interpret a general such string diagram, we will look at some examples with three and four strings.

Figure~\ref{fig:hom_3_objects} shows two examples with three strings (as a warning, the source and target 0-cells of $Y$ and $Z$ are not consistent between these two pictures). In the first, $Z$ is the distinguished string, indicating that $Z$ is the target of one or more hom-objects from the other 1-cells. This could mean either $(X \odot Y) \triangleright Z$ or $X \triangleright (Y \triangleright Z)$, which are isomorphic. In the second string diagram, $Y$ is the distinguished string, so we understand $Y$ to be the target of hom-objects from $X$ and $Z$; this could mean either of the isomorphic 1-cells $(X \triangleright Y) \triangleleft Z$ or $X \triangleright (Y \triangleleft Z)$. As in Figure~\ref{fig:hom_objects}, the 1-cells appear in the opposite order ($X, Y, Z$) to the order of the strings ($Z, Y, X$).

\begin{figure}[h]
	\centering
	\begin{tabular}{m{20mm}m{40mm}}
	\begin{center}$(X \odot Y) \triangleright Z$ \\ or \\ $X \triangleright (Y \triangleright Z)$ \end{center} &
	\begin{center}
		\begin{tikzpicture}
			\fill[color1fill] (0,0) rectangle (1,2);
			\fill[color2fill] (1,0) rectangle (2,2);
			\fill[color3fill] (2,0) rectangle (3,2);
			\fill[color4fill] (3,0) rectangle (4,2);
			\draw[ds] (1,2) -- node[ed] {$Z$} +(0,-2);
			\draw (2,2) -- node[ed] {$Y$} +(0,-2);
			\draw (3,2) -- node[ed] {$X$} +(0,-2);
		\end{tikzpicture}
	\end{center}
	\\ \\
	\begin{center}$(X \triangleright Y) \triangleleft Z$ \\ or \\ $X \triangleright (Y \triangleleft Z)$ \end{center} &
	\begin{center}
		\begin{tikzpicture}
			\fill[color1fill] (0,0) rectangle (1,2);
			\fill[color2fill] (1,0) rectangle (2,2);
			\fill[color3fill] (2,0) rectangle (3,2);
			\fill[color4fill] (3,0) rectangle (4,2);
			\draw (1,2) -- node[ed] {$Z$} +(0,-2);
			\draw[ds] (2,2) -- node[ed] {$Y$} +(0,-2);
			\draw (3,2) -- node[ed] {$X$} +(0,-2);
		\end{tikzpicture}
	\end{center}
\end{tabular}
	\caption{Closed bicategory string diagrams with three strings}
	\label{fig:hom_3_objects}
\end{figure}
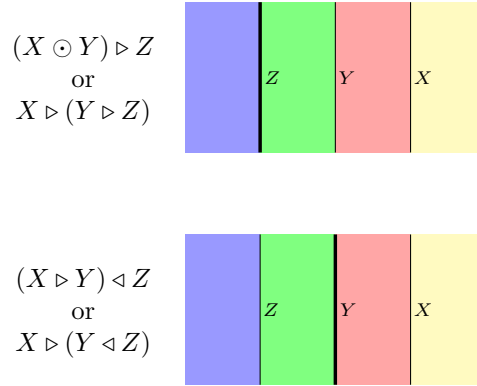

Figure~\ref{fig:hom_4_objects} shows two examples of diagrams with four strings (again, the source and target 0-cells of $Y$ and $Z$ are not consistent between the two pictures). With four strings, there are five different but isomorphic 1-cells that the diagram could represent. These two string diagrams are justified by Lemmas~\ref{lem:pentagon_a_and_t_1} and \ref{lem:pentagon_a_and_t_2}, respectively, which show that the five 1-cells are not just isomorphic but coherently isomorphic, i.e.~any two ways of reparenthesizing via the maps $a$ and $t$ are equal.

\begin{figure}[h]
	\centering
	\begin{tabular}{m{30mm}m{50mm}}
	\begin{center}
		$((W \odot X) \odot Y) \triangleright Z$, \\
		$(W \odot (X \odot Y)) \triangleright Z$, \\
		$W \triangleright ((X \odot Y) \triangleright Z)$, \\
		$W \triangleright (X \triangleright (Y \triangleright Z))$, \\
		$(W \odot X) \triangleright (Y \triangleright Z)$
	\end{center} &
	\begin{center}\begin{tikzpicture}
		\fill[color1fill] (0,0) rectangle (1,2);
		\fill[color2fill] (1,0) rectangle (2,2);
		\fill[color3fill] (2,0) rectangle (3,2);
		\fill[color4fill] (3,0) rectangle (4,2);
		\fill[color5fill] (4,0) rectangle (5,2);
		\draw[ds] (1,2) -- node[ed] {$Z$} +(0,-2);
		\draw (2,2) -- node[ed] {$Y$} +(0,-2);
		\draw (3,2) -- node[ed] {$X$} +(0,-2);
		\draw (4,2) -- node[ed] {$W$} +(0,-2);
	\end{tikzpicture}\end{center}
	\\ \\
	\begin{center}
		$((W \odot X) \triangleright Y) \triangleleft Z$, \\
		$(W \triangleright (X \triangleright Y)) \triangleleft Z$, \\
		$W \triangleright ((X \triangleright Y) \triangleleft Z)$, \\
		$W \triangleright (X \triangleright (Y \triangleleft Z))$, \\
		$(W \odot X) \triangleright (Y \triangleleft Z)$
	\end{center} &
	\begin{center}\begin{tikzpicture}
		\fill[color1fill] (0,0) rectangle (1,2);
		\fill[color2fill] (1,0) rectangle (2,2);
		\fill[color3fill] (2,0) rectangle (3,2);
		\fill[color4fill] (3,0) rectangle (4,2);
		\fill[color5fill] (4,0) rectangle (5,2);
		\draw (1,2) -- node[ed] {$Z$} +(0,-2);
		\draw[ds] (2,2) -- node[ed] {$Y$} +(0,-2);
		\draw (3,2) -- node[ed] {$X$} +(0,-2);
		\draw (4,2) -- node[ed] {$W$} +(0,-2);
	\end{tikzpicture}\end{center}
\end{tabular}
	\caption{Closed bicategory string diagrams with four strings}
	\label{fig:hom_4_objects}
\end{figure}
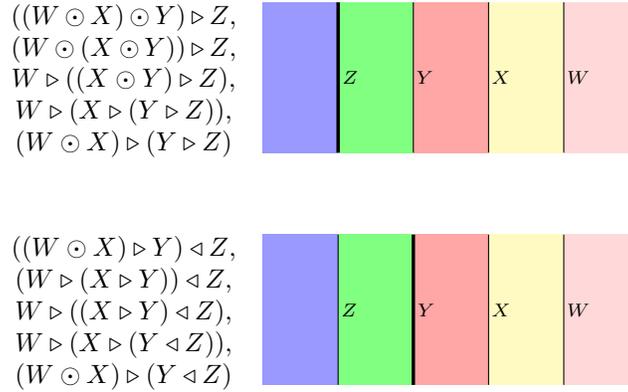

The general procedure for interpreting a diagram with $n$ strings, one of them distinguished, is:
\begin{enumerate}[label=\arabic*.]
	\item Write the 1-cells in the order opposite their appearance in the diagram.
	\item Choose any of the $\frac{1}{n} \binom{2n - 2}{n - 1}$ ways to parenthesize the $n$ 1-cells.
	\item In accordance with that parenthesization, place between each adjacent pair $A$ and $B$ of 1-cells (or parenthesized groups of 1-cells) the symbol
	\[\begin{cases}
		\triangleleft, &\text{if $A$ contains the distinguished 1-cell} \\
		\triangleright, &\text{if $B$ contains the distinguished 1-cell} \\
		\odot, &\text{if neither contains the distinguished 1-cell} \\
	\end{cases}\]
\end{enumerate}

For example, if $Y$ is the distinguished 1-cell, the parenthesization $((WX)Y)Z$ yields $((W \odot X) \triangleright Y) \triangleleft Z$. Our system of string diagrams is somewhat limiting; 1-cells like $(X \triangleright Y) \triangleright Z$ and $(X \triangleright Y) \odot X$ cannot be represented with string diagrams, for example, nor can the coevaluation map $X \to Y \triangleright (X \odot Y)$ or the evaluation map $(Y \triangleright Z) \odot Y \to Z$.

We have two reasons for adopting this system of string diagrams, despite its limitations. One is that the subset of 1-cells made available by these string diagrams is sufficient to produce cotraces and to describe their basic properties in Section~\ref{sec:cotrace_basic_properties}. The other reason is that this system of string diagrams is very similar to that of \cite{Shadows_and_traces}. The only essential difference is a choice of distinguished string; while there may be more ways to combine 1-cells ($\triangleleft$, $\triangleright$, and $\odot$), there are still the same Catalan number's worth of ways to interpret each horizontal slice of a string diagram. While we do not offer a rigorous justification of string diagram manipulations like Ponto and Shulman do for their string diagrams in \cite[Appendix A]{Shadows_and_traces}, the similarity of our string diagrams to theirs suggests that ours could be formalized via a straightforward adaptation of their work. It also gives us confidence that manipulations of our string diagrams, though we do not regard them as formal proofs, do generally lead to true statements about bicategorical cotraces.

As an aid toward proving more complex results such as Theorem~\ref{thm:interplay}, which contain objects inaccessible to this system of string diagrams, we used a variation of the string diagrams of \cite[Section 2.6]{BaezStay} which represent hom-objects with ``clasps.'' However, these diagrams quickly become unwieldy when they contain several nested hom-objects, so we do not reproduce them here. Once again, we view a deformation of these string diagrams as a guide toward a commutative diagram rather than a proof itself.

String diagrams can depict functoriality of $- \triangleright -$ and $-\triangleleft -$, which are covariant in the target and contravariant in the source. For example, a 2-cell $f : P \to P'$ induces a 2-cell $f_* : N \triangleright P \to N \triangleright P'$, and $g : N \to N'$ induces $g^* : N' \triangleright P \to N \triangleright P$. Figure~\ref{fig:hom_functoriality} shows how these are represented as string diagrams.

\begin{figure}[h]
	\centering
	\begin{tabular}{m{17mm}m{30mm}cm{17mm}m{30mm}}
	\begin{center}2-cell\\$\xymatrix{R\rtwocell^P_{P'}{f} & T}$\end{center} &
	\begin{center}\begin{tikzpicture}
		\fill[color1fill] (0,0) rectangle (1,2);
		\fill[color3fill] (1,0) rectangle (2,2);
		\node[anchor=north west,color1dark,re] at (0,2) {$R$};
		\node[anchor=north east,color3dark,re] at (2,2) {$T$};
		\node[fill=white,draw,circle,minimum size=14pt,inner sep=0pt] (f) at (1,1) {$f$};
		\draw (1,2) -- node[ed,swap] {$P$} (f) -- node[ed,swap] {$P'$} +(0,-1);
	\end{tikzpicture}\end{center}
	&&
	\begin{center}2-cell\\$\xymatrix{S\rtwocell^N_{N'}{g} & T}$\end{center} &
	\begin{center}\begin{tikzpicture}
		\fill[color2fill] (0,0) rectangle (1,2);
		\fill[color3fill] (1,0) rectangle (2,2);
		\node[anchor=north west,color2dark,re] at (0,2) {$S$};
		\node[anchor=north east,color3dark,re] at (2,2) {$T$};
		\node[fill=white,draw,circle,minimum size=14pt,inner sep=0pt] (g) at (1,1) {$g$};
		\draw (1,2) -- node[ed,swap] {$N$} (g) -- node[ed,swap] {$N'$} +(0,-1);
	\end{tikzpicture}\end{center}
	\\\\
	\begin{center}2-cell\\$\xymatrix{R\rtwocell^{N \triangleright P}_{N \triangleright P'}{\,\,f_*} & S}$\end{center} &
	\begin{center}\begin{tikzpicture}
		\fill[color1fill] (0,0) rectangle (1,2);
		\fill[color3fill] (1,0) rectangle (2,2);
		\fill[color2fill] (2,0) rectangle (3,2);
		\node[anchor=north west,color1dark,re] at (0,2) {$R$};
		\node[anchor=north,color3dark,re] at (1.5,2) {$T$};
		\node[anchor=north east,color2dark,re] at (3,2) {$S$};
		\node[fill=white,draw,circle,minimum size=14pt,inner sep=0pt] (f) at (1,1) {$f$};
		\draw[ds] (1,2) -- node[ed,swap] {$P$} (f) --node[ed,swap] {$P'$} +(0,-1);
		\draw (2,2) -- node[ed] {$N$} +(0,-2);
	\end{tikzpicture}\end{center}
	&&
	\begin{center}2-cell\\$\xymatrix{R\rtwocell^{N' \triangleright P}_{N \triangleright P}{\,\,g^*} & S}$\end{center} &
	\begin{center}\begin{tikzpicture}
		\fill[color1fill] (0,0) rectangle (1,2);
		\fill[color3fill] (1,0) rectangle (2,2);
		\fill[color2fill] (2,0) rectangle (3,2);
		\node[anchor=north west,color1dark,re] at (0,2) {$R$};
		\node[anchor=north,color3dark,re] at (1.5,2) {$T$};
		\node[anchor=north east,color2dark,re] at (3,2) {$S$};
		\node[fill=white,draw,circle,minimum size=14pt,inner sep=0pt] (g) at (2,1) {$g$};
		\draw[ds] (1,2) -- node[ed,swap] {$P$} +(0,-2);
		\draw (2,2) -- node[ed] {$N'$} (g) -- node[ed] {$N$} +(0,-1);
	\end{tikzpicture}\end{center}
\end{tabular}
	\caption{String diagrams for functoriality of hom}
	\label{fig:hom_functoriality}
\end{figure}
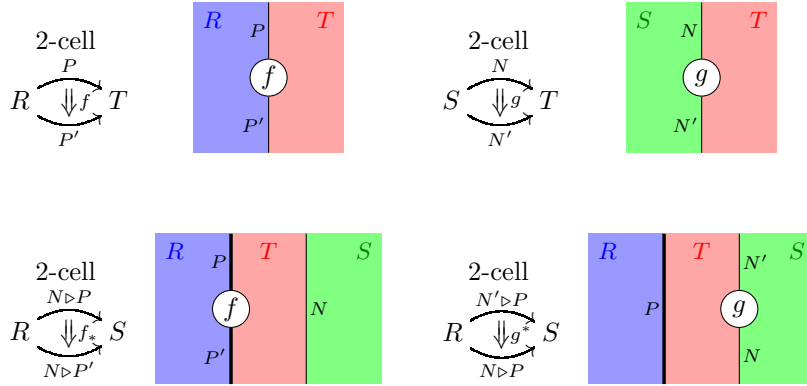

Contravariance of hom in the source makes clear that the reversal of the 0-cell regions of $N$ and $N'$ (or, generally, any string other than the distinguished one) occurs because those strings should be thought of as traveling upward, while the distinguished string alone travels in the usual downward direction. This is consistent with the identification of $N \triangleright P$ with $P \odot N^*$ when $N$ is right dualizable (Remark~\ref{rmk:drawing_homs}), since we think of an $N^*$ string as an $N$ string traveling the wrong way (that is, upward).

As usual, we draw no strings for unit 1-cells; their omission is justified by Lemmas~\ref{lem:hom_unit_iso_compatibility} and \ref{lem:hom_unit_iso_equality}. For example, the first picture in Figure~\ref{fig:hom_objects} could represent $M$, $U_S \triangleright M$, or $M \triangleleft U_R$, or $(M \triangleleft U_R) \triangleleft U_R$, among other possibilities.

\subsection{Duality in closed bicategories}
\label{sec:duality_closed_bicat}

When internal hom-functors are present, they are intimately related with duality. The dual of a finite-dimensional $k$-vector space $V$, for example, is the hom-space $\Hom_k(V, k)$; in fact, in a closed symmetric monoidal category or closed bicategory, the dual of $M$, when it exists, always takes the form (up to isomorphism) of a hom-object from $M$ into a unit object. As a prerequisite to proving this, we consider the maps, natural in $M$, $N$, and $P$,
\[
	M \odot (P \triangleright N) \xrightarrow{\mu} P \triangleright (M \odot N) \qquad \text{and} \qquad (M \triangleleft P) \odot N \xrightarrow{\nu} (M \odot N) \triangleleft P,
\]
which are the transposes of
\[
	M \odot (P \triangleright N) \odot P \xrightarrow{1 \odot \ev} M \odot N \qquad \text{and} \qquad P \odot (M \triangleleft P) \odot N \xrightarrow{\ev \odot 1} M \odot N.
\]

\begin{prop}[\cite{MaySigurdsson}]
\label{prop:mu_iso}
The map $\mu : M \odot (P \triangleright N) \to P \triangleright (M \odot N)$ is an isomorphism if either $M$ or $P$ is right dualizable. Similarly, $\nu : (M \triangleleft P) \odot N \to (M \odot N) \triangleleft P$ is an isomorphism if either $N$ or $P$ is left dualizable.
\end{prop}
\begin{proof}
If $P$ is right dualizable, then
\begin{multline*}
	P \triangleright (M \odot N) \xrightarrow{1 \odot \eta} (P \triangleright (M \odot N)) \odot P \odot P^* \xrightarrow{\ev \odot 1} M \odot N \odot P^* \\ \xrightarrow{1 \odot \coev} M \odot (P \triangleright (N \odot P^* \odot P)) \xrightarrow{1 \odot (1 \odot \varepsilon)_*} M \odot (P \triangleright N)
\end{multline*}
is inverse to $\mu$. If $M$ is right dualizable, then
\begin{multline*}
	P \triangleright (M \odot N) \xrightarrow{\eta \odot 1} M \odot M^* \odot (P \triangleright (M \odot N)) \xrightarrow{1 \odot \mu} M \odot (P \triangleright (M^* \odot M \odot N)) \\ \xrightarrow{1 \odot (\varepsilon \odot 1)_*} M \odot (P \triangleright N)
\end{multline*}
is inverse to $\mu$. Inverses to $\nu$ are constructed similarly when either $N$ or $P$ is left dualizable.
\end{proof}

\begin{prop}
\label{prop:dual_is_hom}
If $M \in \mathscr{B}(R, S)$ is right dualizable, then $M \triangleright U_S$ is a right dual for $M$.
\end{prop}
\begin{proof}
If $M$ is right dualizable, then
\[
	U_R \xrightarrow{\coev} M \triangleright (U_R \odot M) \cong M \triangleright (M \odot U_S) \underset{\cong}{\xrightarrow{\mu^{-1}}} M \odot (M \triangleright U_S)
\]
and $\ev : (M \triangleright U_S) \odot M \to U_S$ exhibit $(M, M \triangleright U_S)$ as a dual pair.
\end{proof}

Similarly, if $M$ is left dualizable, then its left dual is $U_R \triangleleft M$. Together with the isomorphism $\mu$ of Proposition~\ref{prop:mu_iso} in the case that $N$ is a unit 1-cell (or the isomorphism $\nu$ in the case that $M$ is a unit 1-cell), this implies the following.

\begin{cor}
\label{cor:hom_from_dualizable}
If $(P, P^*)$ is a dual pair, then $P \triangleright M \cong M \odot P^*$ for any 1-cell $M$ and $N \triangleleft P^* \cong P \odot N$ for any 1-cell $N$.
\end{cor}

The corollary supplies a slick argument that induction and coinduction coincide for representations of finite-index subgroups:

\begin{example}
Let $V$ be a $k$-linear $H$-representation, where $H \leq G$ is a subgroup of finite index. We saw in Example~\ref{ex:res_and_ind_characters} that $(k[G]_{\varphi}, {}_{\varphi}k[G])$ is a dual pair, and therefore
\[
	\Ind_H^G(V) = k[G]_{\varphi} \odot V \cong V \triangleleft {}_{\varphi}k[G] = \Coind_H^G(V).
\]
\end{example}

A map of the appropriate form for taking traces has a mate (Definition~\ref{defn:mate}), and dually a map $M \triangleright Q \to P \triangleleft N$, which has the appropriate form for taking cotraces (Definition~\ref{defn:cotrace}), has a mate:

\begin{defn}
\label{defn:mate_2}
Let $(M, M^*)$ and $(N, N^*)$ be dual pairs in a closed bicategory. A map $f : M \triangleright Q \to P \triangleleft N$ has a \KEYWORD{mate} $f^* : Q \triangleleft N^* \to M^* \triangleright P$ given by
\begin{multline*}
	Q \triangleleft N^* \underset{\cong}{\xrightarrow{\overline{r}_*}} (U_S \triangleright Q) \triangleleft N^* \xrightarrow{(\varepsilon^*)_*} ((M^* \odot M) \triangleright Q) \triangleleft N^* \underset{\cong}{\xrightarrow{t_*}} (M^* \triangleright (M \triangleright Q)) \triangleleft N^* \\
	\xrightarrow{f_{**}} (M^* \triangleright (P \triangleleft N)) \triangleleft N^* \underset{\cong}{\xrightarrow{a}} M^* \triangleright ((P \triangleleft N) \triangleleft N^*) \underset{\cong}{\xrightarrow{t^{-1}_*}} M^* \triangleright (P \triangleleft (N \odot N^*)) \\
	\xrightarrow{(\eta^*)_*} M^* \triangleright (P \triangleleft U_R) \underset{\cong}{\xrightarrow{\overline{l}^{-1}_*}} M^* \triangleright P.
\end{multline*}
\end{defn}

Similarly, a map $g : Q \triangleleft N^* \to M^* \triangleright P$ has a mate $g^* : M \triangleright Q \to P \triangleleft N$. Moreover, $f^{**} = f$ and $g^{**} = g$ for any $f : M \triangleright Q \to P \triangleleft N$ and $g : Q \triangleleft N^* \to M^* \triangleright P$. While this definition of $f^*$ resembles Definition~\ref{defn:mate_2}, there is an alternate description of $f^*$ which is often easier to work with by virtue of involving fewer hom-objects; it is the transpose of
\[
	(Q \triangleleft N^*) \odot M^* \xrightarrow{\cong} N \odot Q \odot M^* \xrightarrow{\cong} N \odot (M \triangleright Q) \xrightarrow{1 \odot f} N \odot (P \triangleleft N) \xrightarrow{\ev} P,
\]
where the first two maps come from Corollary~\ref{cor:hom_from_dualizable}.

\section{Coshadows and cotraces}
\label{sec:coshadows_and_cotraces}

There are certain constructions appearing in the literature under the name of \emph{cotrace}, which resemble traces in some ways but differ in others. In this section, we develop a theory of bicategorical cotraces which generalizes these examples and draws them into the framework of bicategorical duality and trace.

The prototypical example of a shadow is Hochschild homology
\[
	\HH_0(R, M) \cong M \otimes_{R \otimes R^{\op}} R,
\]
which has the property that
\[
	\HH_0(R, M \otimes_S N) \cong \HH_0(S, N \otimes_R M)
\]
for bimodules ${}_RM_S, {}_SN_R$. Hochschild cohomology
\[
	\HH^0(R, M) \cong \Hom_{R-R}(R, M)
\]
does not have this property, but it does have the property that
\[
	\HH^0(R, \Hom_S(M, N)) \cong \Hom_{R-S}(M, N) \cong \HH^0(S, \Hom_R(M, N))
\]
for bimodules ${}_RM_S, {}_RN_S$. This suggests that the appropriate setting for studying cotraces is a closed bicategory and that the analogue of a shadow functor should be the following.

\begin{defn}
\label{defn:coshadow}
A \KEYWORD{coshadow} for a closed bicategory $\mathscr{B}$ is a category $\mathbf{T}$ and functors
\[
	\lsh{-} : \mathscr{B}(R, R) \to \mathbf{T}
\]
for each 0-cell $R$ of $\mathscr{B}$, equipped with natural isomorphisms
\[
	\theta : \lsh{M \triangleright N} \xrightarrow{\cong} \lsh{N \triangleleft M}
\]
for each $M, N \in \mathscr{B}(R, S)$, such that the following diagrams commute whenever they make sense:
\[
\begin{tikzcd}
	\lsh{(M \odot N) \triangleright P}
		\arrow[d, "\lsh{t}"', "\cong"]
		\arrow[r, "\theta", "\cong"']
	& \lsh{P \triangleleft (M \odot N)}
		\arrow[r, "\lsh{t}", "\cong"']
	& \lsh{(P \triangleleft M) \triangleleft N}
		\arrow[d, "\theta", "\cong"']
	\\
	\lsh{M \triangleright (N \triangleright P)}
		\arrow[r, "\theta"', "\cong"]
	& \lsh{(N \triangleright P) \triangleleft M}
		\arrow[r, "\lsh{a}"', "\cong"]
	& \lsh{N \triangleright (P \triangleleft M)}
\end{tikzcd}
\]
\[
\begin{tikzcd}
	\lsh{U_R \triangleright M}
		\arrow[r, "\theta", "\cong"']
		\arrow[dr, "\lsh{\overline{r}^{-1}}"', "\cong"]
	& \lsh{M \triangleleft U_R}
		\arrow[r, "\theta", "\cong"']
		\arrow[d, "\lsh{\overline{l}^{-1}}"', "\cong"]
	& \lsh{U_R \triangleright M}
		\arrow[dl, "\lsh{\overline{r}^{-1}}", "\cong"']
	\\
	& \lsh{M}
\end{tikzcd}
\]
\end{defn}

Our diagrammatic treatment of coshadows (Figure~\ref{fig:coshadow}) is identical to that of shadows (Figure~\ref{fig:shadow}); the only additional comment we make is that the distinguished string and the undistinguished strings are equally free to travel around the back of the cylinder.

\begin{figure}[h]
	\centering
	\begin{subfigure}[b]{0.3\textwidth}
	\centering
	\begin{tikzpicture}
		\bgcylinder{0,0}{2}{1.2}{.3}{color1}{color1}
		\node[anchor=south west,color1dark,re] at (dl) {$R$};
		\draw[ds] (top) -- node[ed] {$M$} (bot);
	\end{tikzpicture}
	\caption{Coshadow $\lsh{M}$}
	\label{fig:coshadow_single_object}
\end{subfigure}
\begin{subfigure}[b]{0.3\textwidth}
	\centering
	\begin{tikzpicture}
		\bgcylinder{0,0}{3.7}{1.2}{.3}{color1}{color2}
		\begin{pgfonlayer}{foreground}
			\drawtheta{(0,2.1)}{0}{1}{th}{}
		\end{pgfonlayer}

		\filldraw[color2fill] let \p1 = ($(ul)!.65!(ur)$) in
			(thR) -- ++(0.1,0) -- (ur') -- ($(ul)!.65!(ur)$) -- (\x1, 3.1) to[out=-90,in=90] node[ed] {$M$} (thR);
		\filldraw[color2fill] let \p1 = ($(dl)!.35!(dr)$), \p2 = ($(dl)!.65!(dr)$) in
			(thL) to[out=-90,in=90] (\x1, 0.0) -- (\p1 |- bot') -- (\p2 |- bot') -- (\x2, 0.0) to[out=90,in=-90] node[ed,near start] {$N$} (\x1, 3.29) -- ($(ul)!.35!(ur)$) -- (ul') -- (ul' |- thL) -- (thL);
		\draw[ds] let \p1 = ($(dl)!.35!(dr)$), \p2 = ($(dl)!.65!(dr)$) in
			(\p2 |- bot') -- (\x2, 0.0) to[out=90,in=-90] (\x1, 3.29) -- ($(ul)!.35!(ur)$);

		\node[anchor=south east,color1dark,re] at (dr) {$R$};
		\node[anchor=north west,color2dark,re] at ($(ul)!.1!(dl)$) {$S$};
	\end{tikzpicture}
	\caption{Cyclicity $\theta_{M,N}$}
\end{subfigure}
	\caption{String diagrams for coshadows}
	\label{fig:coshadow}
\end{figure}
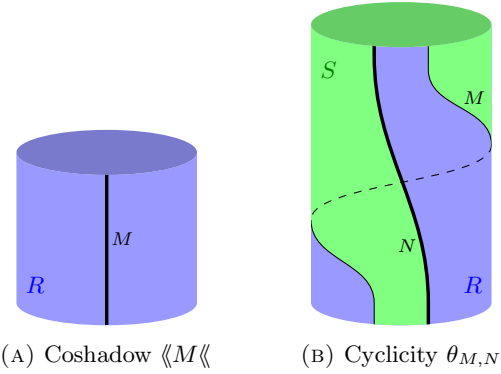

\begin{example}
Zeroth Hochschild cohomology
\[
	\HH^0(R, M) \cong \{m \in M : mr = rm \,\forall r \in R\} \cong \Hom_{R-R}(R, M)
\]
defines a coshadow on $\MORITABICAT$ with target $\mathbf{Ab}$ (or a coshadow on $\MORITABICATK$ with target $\mathbf{Vect}_k$).
\end{example}

\begin{example}
\label{ex:categorical_trace}
The \emph{categorical trace} of \cite[Definition~3.1]{GK2006}, which sends a 1-cell $M \in \mathscr{B}(R, R)$ to the set $\Hom_{\mathscr{B}(R, R)}(U_R, M)$, is an example of a coshadow. The key observation is that for $M, N \in \mathscr{B}(R, S)$, the required isomorphism $\theta$ comes from the tensor-hom adjunctions
\[
\Hom_{\mathscr{B}(R, R)}(U_R, M \triangleright N) \cong \Hom_{\mathscr{B}(R, S)}(M, N) \cong \Hom_{\mathscr{B}(S, S)}(U_S, N \triangleleft M).
\]
This coshadow takes values in the category of sets, but if the bicategory is enriched in a symmetric monoidal category $\mathscr{V}$ (in the sense that categories $\mathscr{B}(R, S)$ are $\mathscr{V}$-enriched categories in a way that is compatible with the horizontal composition of $\mathscr{B}$), then the categorical trace defines a $\mathscr{V}$-valued coshadow. It is a curious fact that the categorical trace provides a simple example of a coshadow on any closed bicategory whatsoever.
\end{example}

\begin{defn}
\label{defn:cotrace}
Let $\mathscr{B}$ be a closed bicategory with a coshadow and $(M, M^*)$ a dual pair with $M \in \mathscr{B}(R, S)$. The \KEYWORD{cotrace} of a 2-cell $f : M \triangleright Q \to P \triangleleft M$, denoted $\cotr(f)$, is the composite
\begin{multline*}
	\lsh{Q} \underset{\cong}{\xrightarrow{\overline{r}}} \lsh{U_S \triangleright Q} \xrightarrow{\lsh{\varepsilon^*}} \lsh{(M^* \odot M) \triangleright Q} \underset{\cong}{\xrightarrow{\lsh{t}}} \lsh{M^* \triangleright (M \triangleright Q)} \\
	\xrightarrow{\lsh{f_*}} \lsh{M^* \triangleright (P \triangleleft M)} \underset{\cong}{\xrightarrow{\theta}} \lsh{(P \triangleleft M) \triangleleft M^*} \underset{\cong}{\xrightarrow{\lsh{t^{-1}}}} \lsh{P \triangleleft (M \odot M^*)} \\
	\xrightarrow{\lsh{\eta^*}} \lsh{P \triangleleft U_R} \underset{\cong}{\xrightarrow{\lsh{\overline{l}^{-1}}}} \lsh{P}.
\end{multline*}
\end{defn}

This definition mirrors that of the bicategorical trace (Definition~\ref{defn:bicategorical_trace}); the only reason it appears to be composed of more maps than the trace is that in the trace we usually suppress the associators
\[
	Q \odot (M \odot M^*) \cong (Q \odot M) \odot M^* \qquad\text{and}\qquad M^* \odot (M \odot P) \cong (M^* \odot M) \odot P,
\]
whereas for the cotrace we always explicitly write out the isomorphisms
\[
	(M^* \odot M) \triangleright Q \cong M^* \triangleright (M \triangleright Q) \qquad\text{and}\qquad (P \triangleleft M) \triangleleft M^* \cong P \triangleleft (M \odot M^*).
\]

When there are multiple dualizable objects in play, we will sometimes subscript $\cotr$ (or $\tr$) with the dualizable object being used for that cotrace (or trace); that is, we might write $\cotr_M(f)$ for the cotrace in Definition~\ref{defn:cotrace}.

A string diagram of the cotrace is shown in Figure~\ref{fig:bicat_cotrace}; the only difference between this picture and that of the trace (Figure~\ref{fig:bicat_trace}) is the greater thickness distinguishing the $Q$ and $P$ strings.

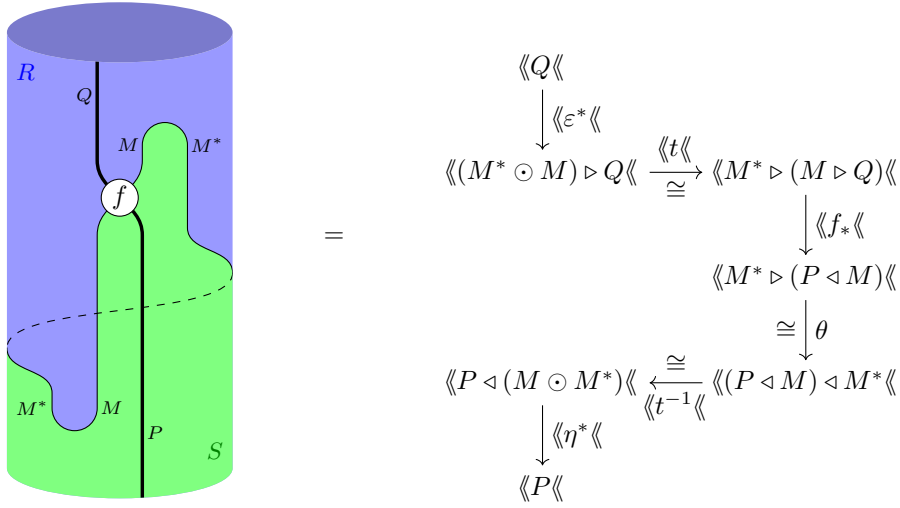
\begin{figure}[h]
	\centering
	\begin{tabular}{m{30mm}m{20mm}m{60mm}}
	\begin{tikzpicture}
		\bgcylinder{0,0}{5.8}{1.5}{.4}{color1}{color1}

		\coordinate (f) at (1.5, 3.6) {};
		\path (f) -- ++(-0.3, 0.5) coordinate (ful);
		\path (f) -- ++(0.3, 0.5) coordinate (fur);
		\path (f) -- ++(-0.3, -0.5) coordinate (fdl);
		\path (f) -- ++(0.3, -0.5) coordinate (fdr);

		\begin{pgfonlayer}{foreground}
			\drawtheta{(f)}{-1.0}{1}{th}{}
		\end{pgfonlayer}

		\filldraw[color2fill] (f) to[out=-135, in=90] (fdl) -- (fdl |- thL) -- node[ed,pos=1]{$M$} ++(0, -0.8) arc (0:-180:0.3) -- node[ed,pos=0]{$M^*$} ++(0, 0.2) to[out=90, in=-90] (thL) -- (thL') -- (thL' |- bot') -- (thR' |- bot') -- (thR') -- (thR) to[out=90, in=-90] ++(-0.6, 0.6) -- ($(fur) + (0.6, 0)$) -- node[ed,swap,pos=1]{$M^*$} ++(0, 0.2) arc(0:180:0.3) -- node[ed,swap,pos=0]{$M$} (fur) to[out=-90, in=45] (f);
		\draw[ds] (f) to[out=135, in=-90] (ful) -- node[ed]{$Q$} (ful |- ul);
		\draw[ds] (f) to[out=-45, in=90] (fdr) -- node[ed,near end]{$P$} (fdr |- bot);

		\node[vert] at (f) {$f$};
		\node[anchor=north west,color1dark,re] at ($(ul)!.05!(dl)$) {$R$};
		\node[anchor=south east,color2dark,re] at (dr) {$S$};
	\end{tikzpicture}
	& \begin{center} = \end{center} &
	\begin{tikzpicture}
		\node at (0, 6.5) {}; 
		\node (A) at (0, 5.6) {$\lsh{Q}$};
		\node (B) at (0, 4.2) {$\lsh{(M^* \odot M) \triangleright Q}$};
		\node (C) at (3.5, 4.2) {$\lsh{M^* \triangleright (M \triangleright Q)}$};
		\node (D) at (3.5, 2.8) {$\lsh{M^* \triangleright (P \triangleleft M)}$};
		\node (E) at (3.5, 1.4) {$\lsh{(P \triangleleft M) \triangleleft M^*}$};
		\node (F) at (0, 1.4) {$\lsh{P \triangleleft (M \odot M^*)}$};
		\node (G) at (0, 0) {$\lsh{P}$};

		\draw[->] (A) -- node[right]{$\lsh{\varepsilon^*}$} (B);
		\draw[->] (B) -- node[above]{$\lsh{t}$} node[below]{$\cong$} (C);
		\draw[->] (C) -- node[right]{$\lsh{f_*}$} (D);
		\draw[->] (D) -- node[right]{$\theta$} node[left]{$\cong$} (E);
		\draw[->] (E) -- node[below]{$\lsh{t^{-1}}$} node[above]{$\cong$} (F);
		\draw[->] (F) -- node[right]{$\lsh{\eta^*}$} (G);
	\end{tikzpicture}
\end{tabular}
	\caption{The bicategorical cotrace}
	\label{fig:bicat_cotrace}
\end{figure}

\begin{lem}
\label{lem:cotrace_well_defined}
The cotrace of $f$ is independent of the choices of $M^*$, $\eta$, and $\varepsilon$.
\end{lem}

\begin{proof}
Let $f : M \triangleright Q \to P \triangleleft M$, where $M$ is right dualizable. Suppose that $M$ participates in two dual pairs: $(M, N)$, with coevaluation and evaluation $\eta$ and $\varepsilon$, and $(M, N')$, with coevaluation and evaluation $\eta'$ and $\varepsilon'$. As in Remark~\ref{rmk:uniqueness_of_duals}, there is an isomorphism $N \cong N'$ given by
\[N \xrightarrow{\id \odot \eta'} N \odot M \odot N' \xrightarrow{\varepsilon \odot \id} N';\]
call this isomorphism $\alpha$. Then in the diagram of Figure~\ref{fig:cotrace_well_defined}, the composite around the outside is the cotrace of $f$ with respect to the dual pair $(M, N)$, and the composite around the inside is the cotrace with respect to $(M, N')$. To see that the upper triangle in the diagram commutes, we expand $\alpha$ and appeal to a triangle identity for $(M, N')$:
\[\begin{tikzcd}
	N \odot M
		\arrow[r, "1 \odot \eta' \odot 1"]
		\arrow[dr, "1"']
		\arrow[rr, bend left=30, "\alpha \odot 1", "\cong"']
	& N \odot M \odot N' \odot M
		\arrow[r, "\varepsilon \odot 1^2"]
		\arrow[d, "1^2 \odot \varepsilon'"]
	& N' \odot M
		\arrow[d, "\varepsilon'"]
	\\
	& N \odot M
		\arrow[r, "\varepsilon"']
	& U_S
\end{tikzcd}\]
The other triangle in Figure~\ref{fig:cotrace_well_defined} commutes for similar reasons, and the rest of the diagram commutes by functoriality of $- \triangleright -$ or naturality of $t$ or $\theta$. Since the diagram commutes, the cotrace of $f$ does not depend on the choice of dual pair.
\end{proof}

\begin{figure}[h]
	\centering
	\resizebox{\textwidth}{!}{
		\begin{tikzpicture}[xscale=3.5, yscale=2]
	\newcommand{\Xa}{0}
	\newcommand{\Xb}{\Xa+0.7}
	\newcommand{\Xc}{\Xb+0.9}
	\newcommand{\Xd}{\Xc+1.0}
	\newcommand{\Xe}{\Xd+1.0}
	\newcommand{\Ya}{0}
	\newcommand{\Yb}{\Ya+1.0}
	\newcommand{\Yc}{\Yb+1.0}
	\newcommand{\Yd}{\Yc+1.0}
	\newcommand{\Ye}{\Yd+1.0}

	\node(A) at (\Xa, \Ye) {$\lsh{Q}$};
	\node(B) at (\Xb, \Ye) {$\lsh{U_S \triangleright Q}$};
	\node(C) at (\Xc, \Ye) {$\lsh{(N \odot M) \triangleright Q}$};
	\node(D) at (\Xe, \Ye) {$\lsh{N \triangleright (M \triangleright Q)}$};
	\node(E) at (\Xe, \Yc) {$\lsh{N \triangleright (P \triangleleft M)}$};
	\node(F) at (\Xe, \Ya) {$\lsh{(P \triangleleft M) \triangleleft N}$};
	\node(G) at (\Xc, \Ya) {$\lsh{P \triangleleft (M \odot N)}$};
	\node(H) at (\Xb, \Ya) {$\lsh{P \triangleleft U_R}$};
	\node(I) at (\Xa, \Ya) {$\lsh{P}$};
	\node(J) at (\Xc, \Yd) {$\lsh{(N' \odot M) \triangleright Q}$};
	\node(K) at (\Xd, \Yd) {$\lsh{N' \triangleright (M \triangleright Q)}$};
	\node(L) at (\Xd, \Yc) {$\lsh{N' \triangleright (P \triangleleft M)}$};
	\node(M) at (\Xd, \Yb) {$\lsh{(P \triangleleft M) \triangleleft N'}$};
	\node(N) at (\Xc, \Yb) {$\lsh{P \triangleleft (M \odot N')}$};

	\draw [->] (A) -- node[above]{$\lsh{\overline{r}}$} node[below]{$\cong$} (B);
	\draw [->] (B) -- node[above]{$\lsh{\varepsilon^*}$} (C);
	\draw [->] (C) -- node[above]{$\lsh{t}$} node[below]{$\cong$} (D);
	\draw [->] (D) -- node[right]{$\lsh{f_*}$} (E);
	\draw [->] (E) -- node[right]{$\lsh{\theta}$} node[left]{$\cong$} (F);
	\draw [->] (F) -- node[below]{$\lsh{t}$} node[above]{$\cong$} (G);
	\draw [->] (G) -- node[below]{$\lsh{\eta^*}$} (H);
	\draw [->] (H) -- node[below]{$\lsh{\overline{l}^{-1}}$} node[above]{$\cong$} (I);

	\draw [->] (B) -- node[below left]{$\lsh{(\varepsilon')^*}$} (J);
	\draw [->] (J) -- node[below]{$\lsh{t}$} node[above]{$\cong$} (K);
	\draw [->] (K) -- node[left]{$\lsh{f_*}$} (L);
	\draw [->] (L) -- node[left]{$\lsh{\theta}$} node[right]{$\cong$} (M);
	\draw [->] (M) -- node[above]{$\lsh{t}$} node[below]{$\cong$} (N);
	\draw [->] (N) -- node[above left]{$\lsh{(\eta')^*}$} (H);

	\draw [->] (J) -- node[right]{$\lsh{(\alpha \odot 1)^*}$} node[left]{$\cong$} (C);
	\draw [->] (K) -- node[below right]{$\lsh{\alpha^*}$} node[above left]{$\cong$} (D);
	\draw [->] (L) -- node[above]{$\lsh{\alpha^*}$} node[below]{$\cong$} (E);
	\draw [->] (M) -- node[above right]{$\lsh{\alpha^*}$} node[below left]{$\cong$} (F);
	\draw [->] (N) -- node[right]{$\lsh{(1 \odot \alpha)^*}$} node[left]{$\cong$} (G);
\end{tikzpicture}
	}
	\caption{Diagram for Lemma~\ref{lem:cotrace_well_defined}}
	\label{fig:cotrace_well_defined}
\end{figure}

\begin{example}
Let $M : R \pto S$ be a right dualizable 1-cell in $\MORITABICAT$, i.e.~an $(R, S)$-bimodule which is finitely generated and projective as a right $S$-module. Using $\HH^0$ as the coshadow, the cotrace of a 2-cell $f : \Hom_S(M, Q) \to \Hom_R(M, P)$ is the map
\[
	\HH^0(S, Q) \to \HH^0(R, P)
\]
taking $q$ to $\sum_i f(qe_i^*(-))(e_i)$, where $\{e_i\}$ and $\{e_i^*\}$ are a pair of dual bases for $M_S$ as in Example~\ref{ex:dualizable_bimodule}.
\end{example}

\begin{example}
Given a representation $V$ of a finite group $G$, there is a homomorphism $f : V \triangleright k \to k[G] \triangleleft V$ given by
\[
	f(\phi)(v) = \sum_{g \in G} \phi(g^{-1}v)g.
\]
In fact, this is an isomorphism, with inverse $f^{-1} : k[G] \triangleleft V \to V \triangleright k$ given by $f^{-1}(\varphi)(v) = e^*(\varphi(v))$. If $V$ is finite-dimensional (i.e.~right dualizable), then $f$ has a cotrace $\HH^0(k) \to \HH^0(k[G])$; since $k$ is commutative $\HH^0(k)$ is just $k$, and $\HH^0(k[G])$ is the subset of $k[G]$ consisting of linear combinations $\sum_{g \in G} a_g g$ such that $a_g = a_{g'}$ whenever $g$ and $g'$ are conjugate. The cotrace is
\[
	\cotr(f)(1) = \sum_{g \in G} \chi(V)(g^{-1}) g,
\]
where $\chi(V)$ is as in Example~\ref{ex:character}. Note that $\cotr(f)$ contains precisely the same information as the character $\chi(V)$: an element of $\HH^0(k[G])$ amounts to a scalar for each conjugacy class of $G$, and the scalars picked out by this cotrace are the values of $\chi(V)$.
\end{example}

We can define a symmetric monoidal cotrace, but it turns out be a trace. We write $[-, -]$ for the internal hom in a closed symmetric monoidal category (i.e. $[Y, -]$ is the right adjoint to $- \otimes Y$). We also make use of the map $\mu : X \otimes [Y, Z] \to [Y, X \otimes Z]$, which is an isomorphism if $X$ or $Y$ is dualizable (cf.~Proposition~\ref{prop:mu_iso}).

\begin{defn}
\label{defn:symmetric_monoidal_cotrace}
Let $(\mathscr{C}, \otimes, I)$ be a closed symmetric monoidal category and $(M, M^*)$ a dual pair. The \KEYWORD{cotrace} of a map $f : [M, Q] \to [M, P]$ is the composite:
\begin{multline*}
	Q \cong \left[I, Q\right] \xrightarrow{\varepsilon^*} \left[M^* \otimes M, Q\right] \underset{\cong}{\xrightarrow{t}} \left[M^*, \left[M, Q\right]\right] \xrightarrow{f_*} \left[M^*, \left[M, P\right]\right] \\
	\underset{\cong}{\xrightarrow{t^{-1}}} \left[M^* \otimes M, P\right] \underset{\cong}{\xrightarrow{s^*}} \left[M \otimes M^*, P\right] \xrightarrow{\eta^*} \left[I, P\right] \cong P
\end{multline*}
\end{defn}

This is similar to the bicategorical cotrace (Definition~\ref{defn:cotrace}), but the symmetry isomorphism obviates the need for a coshadow.

\begin{prop}
\label{prop:sym_mon_cotraces_are_traces}
Let $M$ be a dualizable object in a symmetric monoidal category, and let $f : [M, Q] \to [M, P]$. Then $\cotr(f) = \tr(\tilde{f})$, where $M^* := [M, I]$ and $\tilde{f}$ is the unique map making the following commute:
\[
\begin{tikzcd}
	Q \otimes M^*
		\arrow[d, "\tilde{f}"']
		\arrow[rr, "\mu", "\cong"']
	&& \left[M, Q\right]
		\arrow[d, "f"]
	\\
	M^* \otimes P
		\arrow[r, "s"', "\cong"]
	& P \otimes M^*
		\arrow[r, "\mu"', "\cong"]
	& \left[M, P\right]
\end{tikzcd}
\]
\end{prop}

\begin{proof}
In the diagram of Figure~\ref{fig:sym_mon_cotraces_are_traces}, the left-hand side is $\tr(\tilde{f})$ and the right-hand side is $\cotr(f)$. As in a closed bicategory, if $M$ is dualizable its dual is isomorphic to $[M, I]$ (see Proposition~\ref{prop:dual_is_hom}), so we take $M^*$ to be $[M, I]$, with coevaluation and evaluation
\[\begin{tikzcd}
	\eta : I
		\arrow[r, "\overline{l}"]
	& \left[M, M\right]
		\arrow[r, "\mu^{-1}", "\cong"']
	& M \otimes \left[M, I\right]
\end{tikzcd} \qquad\text{and}\qquad
\varepsilon : \left[M, I\right] \otimes M \xrightarrow{\ev} I.
\]
We let $M^{**}$ be $[M^*, I]$ and define the coevaluation $\eta'$ and evaluation $\varepsilon'$ similarly. To verify that a subdiagram of Figure~\ref{fig:sym_mon_cotraces_are_traces} commutes (if it is not for a straightforward reason like properties of symmetric monoidal categories or naturality of $\mu$), the simplest approach is, as usual, to take transposes of both sides until no hom-objects remain in the target.
\end{proof}

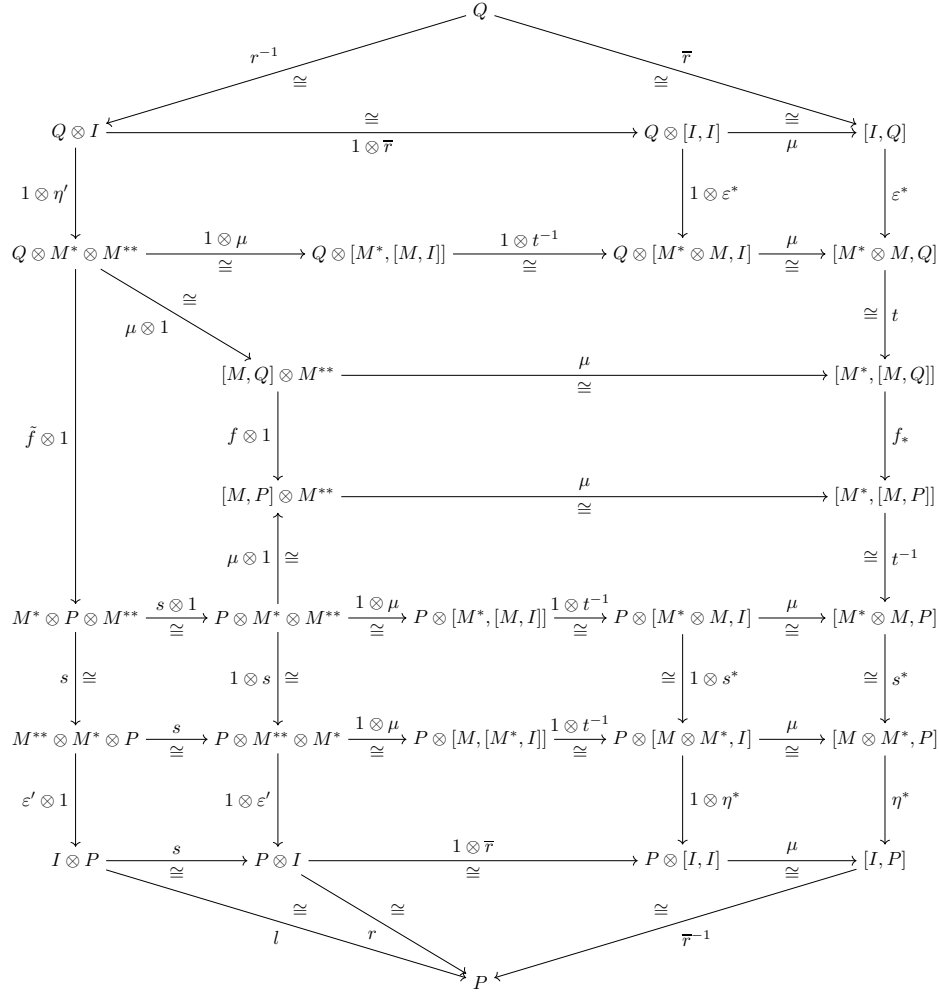
\begin{figure}[h]
	\centering
	\resizebox{\textwidth}{!}{
		\begin{tikzpicture}[xscale=3.75, yscale=2.25]
	\node(A) at (2, 8){$Q$};
	\node(B) at (0, 7){$Q \otimes I$};
	\node(C) at (3, 7){$Q \otimes \left[I, I\right]$};
	\node(D) at (4, 7){$\left[I, Q\right]$};
	\node(E) at (0, 6){$Q \otimes M^* \otimes M^{**}$};
	\node(F) at (1.5, 6){$Q \otimes \left[M^*, \left[M, I\right]\right]$};
	\node(G) at (3, 6){$Q \otimes \left[M^* \otimes M, I\right]$};
	\node(H) at (4, 6){$\left[M^* \otimes M, Q\right]$};
	\node(I) at (1, 5){$\left[M, Q\right] \otimes M^{**}$};
	\node(J) at (4, 5){$\left[M^*, \left[M, Q\right]\right]$};
	\node(M) at (1, 4){$\left[M, P\right] \otimes M^{**}$};
	\node(N) at (4, 4){$\left[M^*, \left[M, P\right]\right]$};
	\node(K) at (0, 3){$M^* \otimes P \otimes M^{**}$};
	\node(L) at (1, 3){$P \otimes M^* \otimes M^{**}$};
	\node(O) at (2, 3){$P \otimes \left[M^*, \left[M, I\right]\right]$};
	\node(P) at (3, 3){$P \otimes \left[M^* \otimes M, I\right]$};
	\node(Q) at (4, 3){$\left[M^* \otimes M, P\right]$};
	\node(R) at (0, 2){$M^{**} \otimes M^* \otimes P$};
	\node(S) at (1, 2){$P \otimes M^{**} \otimes M^*$};
	\node(T) at (2, 2){$P \otimes \left[M, \left[M^*, I\right]\right]$};
	\node(U) at (3, 2){$P \otimes \left[M \otimes M^*, I\right]$};
	\node(V) at (4, 2){$\left[M \otimes M^*, P\right]$};
	\node(W) at (0, 1){$I \otimes P$};
	\node(X) at (1, 1){$P \otimes I$};
	\node(Y) at (3, 1){$P \otimes \left[I, I\right]$};
	\node(Z) at (4, 1){$\left[I, P\right]$};
	\node(A1) at (2, 0){$P$};

	\draw [->] (A) -- node[above left]{$r^{-1}$} node[below right]{$\cong$} (B);
	\draw [->] (A) -- node[above right]{$\overline{r}$} node[below left]{$\cong$} (D);
	\draw [->] (B) -- node[below]{$1 \otimes \overline{r}$} node[above]{$\cong$} (C);
	\draw [->] (C) -- node[below]{$\mu$} node[above]{$\cong$} (D);
	\draw [->] (B) -- node[left]{$1 \otimes \eta'$} (E);
	\draw [->] (E) -- node[above]{$1 \otimes \mu$} node[below]{$\cong$} (F);
	\draw [->] (F) -- node[above]{$1 \otimes t^{-1}$} node[below]{$\cong$} (G);
	\draw [->] (C) -- node[right]{$1 \otimes \varepsilon^*$} (G);
	\draw [->] (G) -- node[above]{$\mu$} node[below]{$\cong$} (H);
	\draw [->] (D) -- node[right]{$\varepsilon^*$} (H);
	\draw [->] (E) -- node[below left]{$\mu \otimes 1$} node[above right]{$\cong$} (I);
	\draw [->] (I) -- node[above]{$\mu$} node[below]{$\cong$} (J);
	\draw [->] (H) -- node[right]{$t$} node[left]{$\cong$} (J);
	\draw [->] (E) -- node[left]{$\tilde{f} \otimes 1$} (K);
	\draw [->] (K) -- node[above]{$s \otimes 1$} node[below]{$\cong$} (L);
	\draw [->] (L) -- node[left]{$\mu \otimes 1$} node[right]{$\cong$} (M);
	\draw [->] (I) -- node[left]{$f \otimes 1$} (M);
	\draw [->] (M) -- node[above]{$\mu$} node[below]{$\cong$} (N);
	\draw [->] (J) -- node[right]{$f_*$} (N);
	\draw [->] (K) -- node[left]{$s$} node[right]{$\cong$} (R);
	\draw [->] (R) -- node[above]{$s$} node[below]{$\cong$} (S);
	\draw [->] (L) -- node[left]{$1 \otimes s$} node[right]{$\cong$} (S);
	\draw [->] (L) -- node[above]{$1 \otimes \mu$} node[below]{$\cong$} (O);
	\draw [->] (O) -- node[above]{$1 \otimes t^{-1}$} node[below]{$\cong$} (P);
	\draw [->] (P) -- node[above]{$\mu$} node[below]{$\cong$} (Q);
	\draw [->] (N) -- node[right]{$t^{-1}$} node[left]{$\cong$} (Q);
	\draw [->] (S) -- node[above]{$1 \otimes \mu$} node[below]{$\cong$} (T);
	\draw [->] (T) -- node[above]{$1 \otimes t^{-1}$} node[below]{$\cong$} (U);
	\draw [->] (P) -- node[right]{$1 \otimes s^*$} node[left]{$\cong$} (U);
	\draw [->] (U) -- node[above]{$\mu$} node[below]{$\cong$} (V);
	\draw [->] (Q) -- node[right]{$s^*$} node[left]{$\cong$} (V);
	\draw [->] (R) -- node[left]{$\varepsilon' \otimes 1$} (W);
	\draw [->] (W) -- node[above]{$s$} node[below]{$\cong$} (X);
	\draw [->] (S) -- node[left]{$1 \otimes \varepsilon'$} (X);
	\draw [->] (X) -- node[above]{$1 \otimes \overline{r}$} node[below]{$\cong$} (Y);
	\draw [->] (U) -- node[right]{$1 \otimes \eta^*$} (Y);
	\draw [->] (Y) -- node[above]{$\mu$} node[below]{$\cong$} (Z);
	\draw [->] (V) -- node[right]{$\eta^*$} (Z);
	\draw [->] (W) -- node[below left]{$l$} node[above right]{$\cong$} (A1);
	\draw [->] (X) -- node[below left]{$r$} node[above right]{$\cong$} (A1);
	\draw [->] (Z) -- node[below right]{$\overline{r}^{-1}$} node[above left]{$\cong$} (A1);
\end{tikzpicture}
	}
	\caption{Diagram for Proposition~\ref{prop:sym_mon_cotraces_are_traces}}
	\label{fig:sym_mon_cotraces_are_traces}
\end{figure}

In a bicategory, however, traces and cotraces truly are different, for the simple reason that shadows and coshadows are different; Hochschild homology and cohomology, for example, are not the same thing.

\section{Properties of cotrace}
\label{sec:cotrace_basic_properties}

The cotrace has properties analogous to those of the bicategorical trace, which are catalogued in Section 7 of \cite{Shadows_and_traces}. Some of the diagrams proving these properties get quite large, so to make them a bit more manageable we sometimes omit the symbol $\odot$. We adopt the convention that $\triangleleft$ and $\triangleright$ bind more loosely than composition by juxtaposition, e.g.~$AB \triangleright C$ is to be understood as $(A \odot B) \triangleright C$, not $A \odot (B \triangleright C)$.

The following is analogous to \cite[Proposition~7.1]{Shadows_and_traces}:

\begin{prop}
\label{prop:cotrace_tightening}
Let $M$ be a right dualizable 1-cell, let $f : M \triangleright Q \to P \triangleleft M$ be a 2-cell, and let $g : Q' \to Q$ and $h : P \to P'$ be 2-cells. Then
\[
	\lsh{h} \circ \cotr(f) \circ \lsh{g} = \cotr(h_* \circ f \circ g_*).
\]
\end{prop}

\begin{proof}
The composite around the outside top, right, and bottom of Figure~\ref{fig:CD_tightening} is $\cotr(h_* \circ f \circ g_*)$. Each square commutes because of functoriality of the internal hom or naturality of $\theta$, $t$, $\overline{r}$, or $\overline{l}$.
\end{proof}

This equality is shown graphically in Figure~\ref{fig:SD_tightening}. A deformation between the two figures is easy to visualize: we simply slide the $g$ and $h$ nodes closer to $f$, contracting the $Q$ and $P$ strings and lengthening the $Q'$ and $P'$ strings.

\begin{figure}[h]
	\centering
	\begin{tabular}{m{3cm}m{20mm}m{3cm}}
	\begin{tikzpicture}
		\bgcylinder{0,0}{7.0}{1.5}{.4}{color1}{color1}

		\coordinate (f) at (1.5, 4.1);
		\node (g) at (1.2, 5.8) {};
		\node (h) at (1.8, 0.4) {};
		\path (f) -- ++(-0.3, 0.5) coordinate (ful);
		\path (f) -- ++(0.3, 0.5) coordinate (fur);
		\path (f) -- ++(-0.3, -0.5) coordinate (fdl);
		\path (f) -- ++(0.3, -0.5) coordinate (fdr);

		\begin{pgfonlayer}{foreground}
			\drawtheta{(f)}{-1}{1}{th}{}
		\end{pgfonlayer}

		\def\ecs{0.2} 
		\filldraw[color2fill] (f) to[out=-135, in=90] (fdl) -- (fdl |- thL) -- ++(0, -0.5) -- node[ed,pos=1]{$M$} ++(0, -\ecs) arc (0:-180:0.3) -- node[ed,pos=0]{$M^*$} ++(0, \ecs) to[out=90, in=-90] (thL) -- (thL') -- (thL' |- bot') -- (thR' |- bot') -- (thR') -- (thR) to[out=90, in=-90] ++(-0.6, 0.5) -- ($(fur) + (0.6, 0)$) -- node[ed,swap,pos=1]{$M^*$} ++(0, \ecs) arc(0:180:0.3) -- node[ed,swap,pos=0]{$M$} (fur) to[out=-90, in=45] (f);
		\draw[ds] (f) to[out=135, in=-90] (ful) -- node[ed]{$Q$} (g) -- node[ed]{$Q'$} (ful |- ul);
		\draw[ds] (f) to[out=-45, in=90] (fdr) -- node[ed]{$P$} (h) -- node[ed]{$P'$} (fdr |- bot);

		\node[vert] at (f) {$f$};
		\node[vert] at (g) {$g$};
		\node[vert] at (h) {$h$};
	\end{tikzpicture}
	& \begin{center} = \end{center} &
	\begin{tikzpicture}
		\bgcylinder{0,0}{7.0}{1.5}{.4}{color1}{color1}

		\node (f) at (1.5, 4.3) {};
		\path (f) -- ++(-0.3, 0.5) coordinate (ful) -- ++(0, 0.5) coordinate(g);
		\path (f) -- ++(0.3, 0.5) coordinate (fur);
		\path (f) -- ++(-0.3, -0.5) coordinate (fdl);
		\path (f) -- ++(0.3, -0.5) coordinate (fdr) -- ++(0, -0.5) coordinate(h);

		\begin{pgfonlayer}{foreground}
			\drawtheta{(f)}{-2.0}{1}{th}{}
		\end{pgfonlayer}

		\def\ecs{0.2} 
		\filldraw[color2fill] (f) to[out=-135, in=90] (fdl) -- (fdl |- thL) -- ++(0, -0.5) -- node[ed,pos=1]{$M$} ++(0, -\ecs) arc (0:-180:0.3) -- node[ed,pos=0]{$M^*$} ++(0, \ecs) to[out=90, in=-90] (thL) -- (thL') -- (thL' |- bot') -- (thR' |- bot') -- (thR') -- (thR) to[out=90, in=-90] ++(-0.6, 0.5) -- node[ed,swap,pos=1]{$M^*$} ++(0, 2.9) arc(0:180:0.3) -- node[ed,swap,pos=0]{$M$} (fur) to[out=-90, in=45] (f);
		\draw[ds] (f) to[out=135, in=-90] (ful) -- node[ed,pos=0]{$Q$} (g) -- node[ed]{$Q'$} (ful |- ul);
		\draw[ds] (f) to[out=-45, in=90] (fdr) -- node[ed,pos=0]{$P$} (h) -- node[ed]{$P'$} (fdr |- bot);

		\node[vert] at (f) {$f$};
		\node[vert] at (g) {$g$};
		\node[vert] at (h) {$h$};
	\end{tikzpicture}
\end{tabular}
	\caption{String diagram picture of Proposition~\ref{prop:cotrace_tightening}}
	\label{fig:SD_tightening}
\end{figure}

\begin{figure}[h]
	\centering
	\resizebox{\textwidth}{!}{
		\begin{tikzpicture}[xscale=3.5, yscale=2]
	\newcommand{\Xa}{0}
	\newcommand{\Xb}{\Xa+0.7}
	\newcommand{\Xc}{\Xb+1.0}
	\newcommand{\Xd}{\Xc+1.2}
	\newcommand{\Xe}{\Xd+0.8}
	\newcommand{\Xf}{\Xe+1.2}
	\newcommand{\Ya}{0}
	\newcommand{\Yb}{\Ya+1.0}
	\newcommand{\Yc}{\Yb+1.0}
	\newcommand{\Yd}{\Yc+1.0}
	\newcommand{\Ye}{\Yd+1.0}

	\node(A0) at (\Xa, \Ye){$\lsh{Q'}$};
	\node(B0) at (\Xb, \Ye){$\lsh{U_S \triangleright Q'}$};
	\node(C0) at (\Xc, \Ye){$\lsh{(M^* \odot M) \triangleright Q'}$};
	\node(D0) at (\Xd, \Ye){$\lsh{M^* \triangleright (M \triangleright Q')}$};
	\node(E0) at (\Xf, \Yc){$\lsh{M^* \triangleright (P' \triangleleft M)}$};
	\node(F0) at (\Xd, \Ya){$\lsh{(P' \triangleleft M) \triangleleft M^*}$};
	\node(G0) at (\Xc, \Ya){$\lsh{P' \triangleleft (M \odot M^*)}$};
	\node(H0) at (\Xb, \Ya){$\lsh{P' \triangleleft U_R}$};
	\node(I0) at (\Xa, \Ya){$\lsh{P'}$};

	\node(A1) at (\Xa, \Yd){$\lsh{Q}$};
	\node(B1) at (\Xb, \Yd){$\lsh{U_S \triangleright Q}$};
	\node(C1) at (\Xc, \Yd){$\lsh{(M^* \odot M) \triangleright Q}$};
	\node(D1) at (\Xd, \Yd){$\lsh{M^* \triangleright (M \triangleright Q)}$};
	\node(E1) at (\Xe, \Yc){$\lsh{M^* \triangleright (P \triangleleft M)}$};
	\node(F1) at (\Xd, \Yb){$\lsh{(P \triangleleft M) \triangleleft M^*}$};
	\node(G1) at (\Xc, \Yb){$\lsh{P \triangleleft (M \odot M^*)}$};
	\node(H1) at (\Xb, \Yb){$\lsh{P \triangleleft U_R}$};
	\node(I1) at (\Xa, \Yb){$\lsh{P}$};

	\draw [->] (A0) -- node[left]{$\lsh{g}$} (A1);
	\draw [->] (B0) -- node[left]{$\lsh{g_*}$} (B1);
	\draw [->] (C0) -- node[left]{$\lsh{g_*}$} (C1);
	\draw [->] (D0) -- node[left]{$\lsh{g_{**}}$} (D1);
	\draw [->] (E1) -- node[above]{$\lsh{h_{**}}$} (E0);
	\draw [->] (I1) -- node[left]{$\lsh{h}$} (I0);
	\draw [->] (H1) -- node[left]{$\lsh{h_*}$} (H0);
	\draw [->] (G1) -- node[left]{$\lsh{h_*}$} (G0);
	\draw [->] (F1) -- node[left]{$\lsh{h_{**}}$} (F0);
	\draw [->] (A1) -- node[left]{$\cotr(f)$} (I1);

	\draw [->] (A0) -- node[above]{$\lsh{\overline{r}}$} node[below]{$\cong$} (B0);
	\draw [->] (B0) -- node[above]{$\lsh{\varepsilon^*}$} (C0);
	\draw [->] (C0) -- node[above]{$\lsh{t}$} node[below]{$\cong$} (D0);
	\draw [->] (D0) -- node[above right]{$\lsh{(h_* \circ f \circ g_*)_*}$} (E0);
	\draw [->] (E0) -- node[below right]{$\lsh{\theta}$} node[above left]{$\cong$} (F0);
	\draw [->] (F0) -- node[below]{$\lsh{t^{-1}}$} node[above]{$\cong$} (G0);
	\draw [->] (G0) -- node[below]{$\lsh{\eta^*}$} (H0);
	\draw [->] (H0) -- node[below]{$\lsh{\overline{l}^{-1}}$} node[above]{$\cong$} (I0);

	\draw [->] (A1) -- node[above]{$\lsh{\overline{r}}$} node[below]{$\cong$} (B1);
	\draw [->] (B1) -- node[above]{$\lsh{\varepsilon^*}$} (C1);
	\draw [->] (C1) -- node[above]{$\lsh{t}$} node[below]{$\cong$} (D1);
	\draw [->] (D1) -- node[above right]{$\lsh{f_*}$} (E1);
	\draw [->] (E1) -- node[below right]{$\lsh{\theta}$} node[above left]{$\cong$} (F1);
	\draw [->] (F1) -- node[below]{$\lsh{t^{-1}}$} node[above]{$\cong$} (G1);
	\draw [->] (G1) -- node[below]{$\lsh{\eta^*}$} (H1);
	\draw [->] (H1) -- node[below]{$\lsh{\overline{l}^{-1}}$} node[above]{$\cong$} (I1);
\end{tikzpicture}
	}
	\caption{Diagram for Proposition~\ref{prop:cotrace_tightening}}
	\label{fig:CD_tightening}
\end{figure}

The following is analogous to \cite[Proposition~7.4]{Shadows_and_traces}:

\begin{prop}
\label{prop:cotrace_with_unit}
If $f : U_R \triangleright Q \to P \triangleleft U_R$ is any 2-cell, then $\cotr(f)$ is
\[\begin{tikzcd}
	\lsh{Q}
		\arrow[r, "\lsh{\overline{r}}", "\cong"']
	& \lsh{U_R \triangleright Q}
		\arrow[r, "\lsh{f}"]
	& \lsh{P \triangleleft U_R}
		\arrow[r, "\lsh{\overline{l}^{-1}}", "\cong"']
	& \lsh{P}
\end{tikzcd}\]
\end{prop}

\begin{proof}
In the diagram in Figure~\ref{fig:CD_cotrace_with_unit}, the map around the outside from $\lsh{Q}$ to $\lsh{P}$ is $\cotr(f)$. Note that two of the arrows are labeled two different ways, making use of Lemma~\ref{lem:hom_unit_iso_equality}.
\end{proof}

A string diagram picture for this equality would be tautological, since we do not draw unit 1-cells.

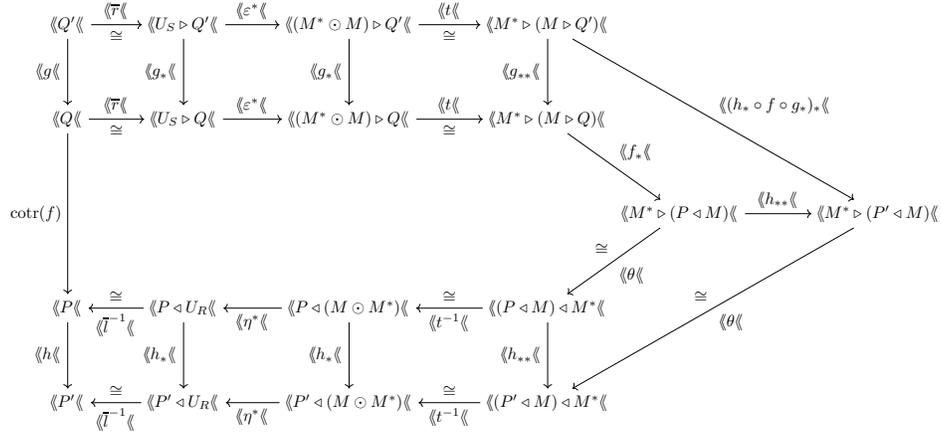
\begin{figure}[h]
	\centering
	\resizebox{\textwidth}{!}{
		\begin{tikzpicture}[xscale=3.25, yscale=2.25]
	\newcommand{\Xa}{0}
	\newcommand{\Xb}{\Xa+1.0}
	\newcommand{\Xc}{\Xb+1.0}
	\newcommand{\Xd}{\Xc+1.0}
	\newcommand{\Xe}{\Xd+1.0}
	\newcommand{\Ya}{0}
	\newcommand{\Yb}{\Ya+1.0}
	\newcommand{\Yc}{\Yb+1.0}
	\newcommand{\Yd}{\Yc+1.0}
	\newcommand{\Ye}{\Yd+1.0}

	\node(A) at (\Xa, \Yd){$\lsh{Q}$};
	\node(B) at (\Xb, \Yd){$\lsh{U_R \triangleright Q}$};
	\node(C) at (\Xc, \Ye){$\lsh{(U_R \odot U_R) \triangleright Q}$};
	\node(D) at (\Xd, \Yd){$\lsh{U_R \triangleright (U_R \triangleright Q)}$};
	\node(E) at (\Xe, \Yc){$\lsh{U_R \triangleright (P \triangleleft U_R)}$};
	\node(F) at (\Xd, \Yb){$\lsh{(P \triangleleft U_R) \triangleleft U_R}$};
	\node(G) at (\Xc, \Ya){$\lsh{P \triangleleft (U_R \odot U_R)}$};
	\node(H) at (\Xb, \Yb){$\lsh{P \triangleleft U_R}$};
	\node(I) at (\Xa, \Yb){$\lsh{P}$};
	\node(J) at (\Xc, \Yc){$\lsh{P \triangleleft U_R}$};

	\node at (barycentric cs:D=1,J=1){\small Lemma \ref{lem:hom_unit_isomorphisms}};
	\node at (barycentric cs:F=1,J=1){\small Definition \ref{defn:coshadow}};

	\draw [->] (A) -- node[above]{$\lsh{\overline{r}}$} node[below]{$\cong$} (B);
	\draw [->] (B) -- node[above left]{$\lsh{r^*}$} node[below right]{$\cong$} (C);
	\draw [->] (C) -- node[above right]{$\lsh{t}$} node[below left]{$\cong$} (D);
	\draw [->] (D) -- node[above right]{$\lsh{f_*}$} (E);
	\draw [->] (E) -- node[below right]{$\lsh{\theta}$} node[above left ]{$\cong$} (F);
	\draw [->] (F) -- node[below right]{$\lsh{t^{-1}}$} node[above left]{$\cong$} (G);
	\draw [->] (G) -- node[below left]{$\lsh{(l^{-1})^*}$} node[above right]{$\cong$} (H);
	\draw [->] (H) -- node[below]{$\lsh{\overline{l}^{-1}}$} node[above]{$\cong$} (I);
	\draw [->] (B) -- node[below]{$\lsh{\overline{r}_*} = \lsh{\overline{r}}$} node[above]{$\cong$} (D);
	\draw [->] (F) -- node[above]{$\lsh{\overline{l}^{-1}} = \lsh{(\overline{l}_*)^{-1}}$} node[below]{$\cong$} (H);
	\draw [->] (B) -- node[below left]{$\lsh{f}$} (J);
	\draw [->] (J) -- node[above left]{$1$} (H);
	\draw [->] (J) -- node[above]{$\lsh{\overline{r}}$} node[below]{$\cong$} (E);

	\node at (barycentric cs:B=1,C=1,D=1){Lemma \ref{lem:hom_unit_iso_compatibility_with_t}};
	\node at (barycentric cs:F=1,G=1,H=1){Lemma \ref{lem:hom_unit_iso_compatibility_with_t}};
\end{tikzpicture}
	}
	\caption{Diagram for Proposition~\ref{prop:cotrace_with_unit}}
	\label{fig:CD_cotrace_with_unit}
\end{figure}

The following is analogous to Theorem~\ref{thm:composite_of_dual_pairs_traces}, which is \cite[Proposition~7.5]{Shadows_and_traces}; it says that a composite of cotraces with respect to dualizable objects $M$ and $N$ is the same as a single cotrace with respect to $N \odot M$ (which is dualizable by Proposition~\ref{prop:composite_of_dual_pairs}).

\begin{prop}
\label{prop:cotrace_composite_of_dual_pairs}
Let $f : M \triangleright Q \to P \triangleleft M$ and $g : N \triangleright P \to L \triangleleft N$, where $M$ and $N$ are right dualizable. Then the composite $\cotr(g) \circ \cotr(f) : \lsh{Q} \to \lsh{L}$ is equal to the cotrace (with respect to $N \odot M$) of the composite
\begin{multline}
\label{eq:cotrace_composite_of_dual_pairs}
	(N \odot M) \triangleright Q \underset{\cong}{\xrightarrow{t}} N \triangleright (M \triangleright Q) \xrightarrow{f_*} N \triangleright (P \triangleleft M) \underset{\cong}{\xrightarrow{a^{-1}}} (N \triangleright P) \triangleleft M \\
	\xrightarrow{g_*} (L \triangleleft N) \triangleleft M \underset{\cong}{\xrightarrow{t^{-1}}} L \triangleleft (N \odot M).
\end{multline}
\end{prop}

\begin{proof}
The left and bottom sides of the diagram in Figure~\ref{fig:CD_composite_of_dual_pairs} are $\cotr(g) \circ \cotr(f)$, and the top and right sides are the cotrace of (\ref{eq:cotrace_composite_of_dual_pairs}), with the exception that we have deleted the maps $\lsh{\varepsilon^* \overline{r}} : \lsh{Q} \to \lsh{(M^* \odot M) \triangleright P}$ at the beginning and $\lsh{\overline{l}^{-1} \eta^*} : \lsh{L \triangleleft (N \odot N^*)} \to \lsh{L}$ at the end, since these are common to both sides. Every unlabeled square in the diagram commutes because of functoriality of $- \triangleright -$ and $- \triangleleft -$ or naturality of $\theta$, $t$, $a$, $\overline{r}$, or $\overline{l}$.
\end{proof}

\begin{figure}[h]
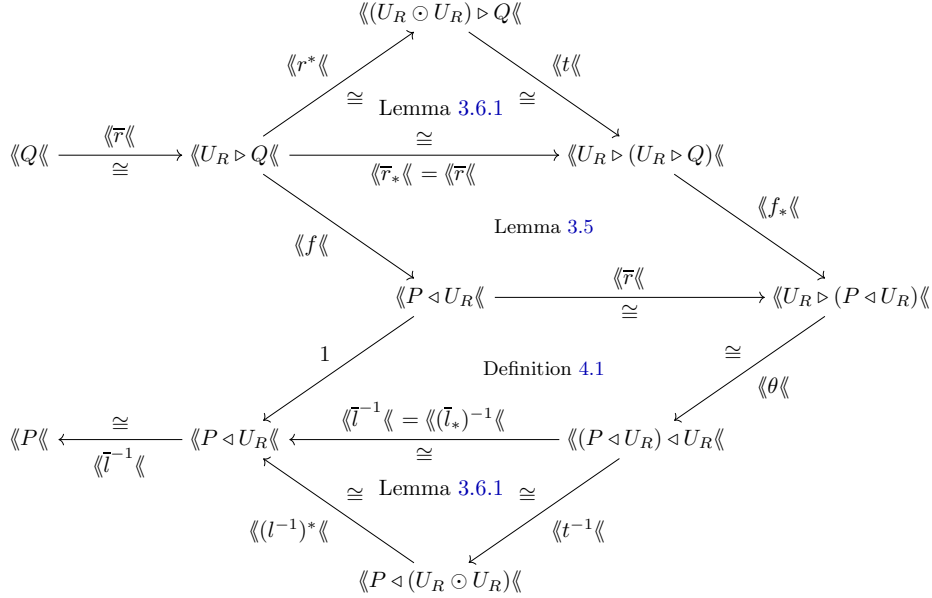

	\centering
	\resizebox{!}{\dimexpr\textheight-0.26in}{\rotatebox{90}{
		\input diagram_composite_of_dual_pairs.tex
	}}
	\caption{Diagram for Proposition~\ref{prop:cotrace_composite_of_dual_pairs}}
	\label{fig:CD_composite_of_dual_pairs}
\end{figure}

This equality is shown graphically in Figure~\ref{fig:SD_composite_of_dual_pairs}; again, the deformation should be fairly easy to visualize.

\begin{figure}[h]
	\centering
	\begin{tabular}{m{36mm}m{20mm}m{48mm}}
	\begin{tikzpicture}
		\bgcylinder{0,0}{9.4}{1.8}{.4}{color1}{color1}

		\coordinate (f) at (1.5, 7.6) {};
		\node (g) at (2.1, 3.1) {};
		\path (f) -- ++(-0.3, 0.5) coordinate (ful);
		\path (f) -- ++(0.3, 0.5) coordinate (fur);
		\path (f) -- ++(-0.3, -0.5) coordinate (fdl);
		\path (f) -- ++(0.3, -0.5) coordinate (fdr);
		\path (g) -- ++(-0.3, 0.5) coordinate (gul);
		\path (g) -- ++(0.3, 0.5) coordinate (gur);
		\path (g) -- ++(-0.3, -0.5) coordinate (gdl);
		\path (g) -- ++(0.3, -0.5) coordinate (gdr);

		\begin{pgfonlayer}{foreground}
			\drawtheta{(f)}{-1.1}{1}{th}{}
			\drawtheta{(g)}{-0.6}{1}{th2}{}
		\end{pgfonlayer}

		\def\ecs{0.1} 
		\filldraw[color2fill] (f) to[out=-135, in=90] (fdl) -- (fdl |- thL) -- ++(0, -0.5) -- node[ed]{$M$} ++(0, -\ecs) arc(0:-180:0.3) -- node[ed]{$M^*$} ++(0, \ecs) to[out=90, in=-90] (thL) -- (thL') -- (thL' |- bot') -- (thR' |- bot') -- (thR') -- (thR) to[out=90, in=-90] ++(-1.2, 1.0) -- ($(fur) + (0.6, 0)$) -- node[ed,swap]{$M^*$} ++(0, \ecs) arc (0:180:0.3) -- node[ed,swap]{$M$} (fur) to[out=-90, in=45] (f);
		\filldraw[color3fill] (g) to[out=-135, in=90] (gdl) -- (gdl |- th2L) -- ++(0, -1.0) -- node[ed]{$N$} ++(0, -\ecs) arc(0:-180:0.3) -- node[ed]{$N^*$} ++(0, \ecs) to[out=90, in=-90] (th2L) -- (th2L') -- (th2L' |- bot') -- (th2R' |- bot') -- (th2R') -- (th2R) to[out=90, in=-90] ++(-0.6, 0.5) -- ($(gur) + (0.6, 0)$) -- node[ed,swap]{$N^*$} ++(0, \ecs) arc (0:180:0.3) -- node[ed,swap]{$N$} (gur) to[out=-90, in=45] (g);
		\draw[ds] (f) to[out=135, in=-90] (ful) -- node[ed]{$Q$} (ful |- ul);
		\draw[ds] (f) to[out=-45, in=90] (fdr) -- node[ed]{$P$} (gul) to[out=-90, in=135] (g);
		\draw[ds] (g) to[out=-45, in=90] (gdr) -- node[ed,near end]{$L$} (gdr |- bot);

		\node[vert] at (f) {$f$};
		\node[vert] at (g) {$g$};
	\end{tikzpicture}
	& \begin{center} = \end{center} &
	\begin{tikzpicture}
		\bgcylinder{0,0}{9.4}{2.4}{.4}{color1}{color1}

		\node (f) at (2.1, 6.55) {};
		\path (f) -- ++(0.6, -1.5) coordinate(g);
		\path (f) -- ++(-0.3, 0.5) coordinate (ful);
		\path (f) -- ++(0.3, 0.5) coordinate (fur);
		\path (f) -- ++(-0.3, -0.5) coordinate (fdl);
		\path (f) -- ++(0.3, -0.5) coordinate (fdr);
		\path (g) -- ++(-0.3, 0.5) coordinate (gul);
		\path (g) -- ++(0.3, 0.5) coordinate (gur);
		\path (g) -- ++(-0.3, -0.5) coordinate (gdl);
		\path (g) -- ++(0.3, -0.5) coordinate (gdr);

		\begin{pgfonlayer}{foreground}
			\drawtheta{(g)}{-1.1}{1}{th}{}
			\drawtheta{(g)}{-1.9}{1}{th2}{}
		\end{pgfonlayer}

		\def\ecs{0.2} 
		\filldraw[color2fill] (f) to[out=-135, in=90] (fdl) -- (fdl |- thL) -- (fdl |- th2L) -- ++(0, -0.5) -- node[ed,pos=1]{$M$} ++(0, -\ecs) arc(0:-180:0.3) -- node[ed,pos=0]{$M^*$} ($(thL) + (1.2, -1.0)$) to[out=90, in=-90] (thL) -- (thL') -- (thL' |- bot') -- (thR' |- bot') -- (thR') -- (thR) to[out=90, in=-90] ++(-0.6, 0.5) -- ($(fur) + (1.8, 0)$) -- node[ed,swap,pos=1]{$M^*$} ++(0, \ecs) arc (0:180:0.9) -- node[ed,swap,pos=0]{$M$} (fur) to[out=-90, in=45] (f);
		\filldraw[color3fill] (g) to[out=-135, in=90] (gdl) -- (gdl |- th2L) -- ++(0, -0.5) -- node[ed,pos=1]{$N$} ++(0, -\ecs) arc(0:-180:0.9) -- node[ed,pos=0]{$N^*$} ++(0, \ecs) to[out=90, in=-90] (th2L) -- (th2L') -- (th2L' |- bot') -- (th2R' |- bot') -- (th2R') -- (th2R) to[out=90, in=-90] ++(-1.2, 1.0) -- ($(fur) + (1.2, 0)$) -- node[ed,swap,pos=1]{$N^*$} ++(0, \ecs) arc (0:180:0.3) -- node[ed,swap,pos=0]{$N$} (gur) to[out=-90, in=45] (g);
		\draw[ds] (f) to[out=135, in=-90] (ful) -- node[ed]{$Q$} (ful |- ul);
		\draw[ds] (f) to[out=-45, in=90] (fdr) -- node[ed]{$P$} (gul) to[out=-90, in=135] (g);
		\draw[ds] (g) to[out=-45, in=90] (gdr) -- node[ed,near end]{$L$} (gdr |- bot);

		\node[vert] at (f) {$f$};
		\node[vert] at (g) {$g$};
	\end{tikzpicture}
\end{tabular}
	\caption{String diagram picture of Proposition~\ref{prop:cotrace_composite_of_dual_pairs}}
	\label{fig:SD_composite_of_dual_pairs}
\end{figure}

If $(M, M^*)$ is a dual pair, a 2-cell $f : M \triangleright Q \to P \triangleleft M$ has a mate $f^* : Q \triangleleft M^* \to M^* \triangleright P$ (Definition~\ref{defn:mate_2}), which has the same cotrace as $f$:

\begin{prop}
\label{prop:cotrace_of_mate}
If $M$ is right dualizable and $f : M \triangleright Q \to P \triangleleft M$ is a 2-cell, then $\cotr(f) = \cotr(f^*)$.
\end{prop}

\begin{proof}
The left and bottom sides of the diagram in Figure~\ref{fig:CD_cotrace_of_mate} are $\cotr(f)$, and the top and right sides are $\cotr(f^*)$. Every unlabeled square commutes because of functoriality of $- \triangleright -$ and $- \triangleleft -$ or naturality of $\theta$, $t$, $a$, $\overline{r}$, or $\overline{l}$.
\end{proof}

\begin{figure}[h]
	\centering
	\resizebox{!}{\dimexpr\textheight-0.26in}{\rotatebox{90}{
		\input diagram_cotrace_of_mate.tex
	}}
	\caption{Diagram for Proposition~\ref{prop:cotrace_of_mate}}
	\label{fig:CD_cotrace_of_mate}
\end{figure}

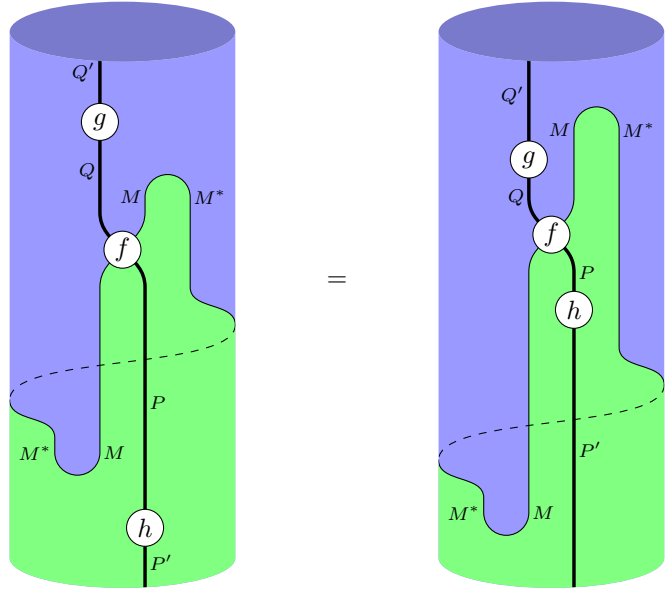
\begin{figure}[h]
	\centering
	\begin{tabular}{m{42mm}m{20mm}m{30mm}}
	\begin{tikzpicture}
		\bgcylinder{0,0}{7.0}{2.1}{.4}{color1}{color1}

		\coordinate (f) at (2.1, 4.2);
		\path (f) -- ++(-0.3, 0.5) coordinate (ful);
		\path (f) -- ++(0.3, 0.5) coordinate (fur);
		\path (f) -- ++(-0.3, -0.5) coordinate (fdl);
		\path (f) -- ++(0.3, -0.5) coordinate (fdr);

		\begin{pgfonlayer}{foreground}
			\drawtheta{(f)}{-2.8}{-1}{th}{}
		\end{pgfonlayer}

		\def\ecs{0.2} 
		\filldraw[color2fill] (f) to[out=45, in=-90] (fur) -- node[ed,pos=1]{$M$} ++(0, \ecs) coordinate (cap) arc(180:0:0.3) -- node[ed,pos=0]{$M^*$} ($(thR) + (-1.2, -0.5)$) -- ++(0, -\ecs) arc(-180:0:0.3) -- node[ed,swap]{$M$} ++(0, \ecs) to[out=90, in=-90] (thR) -- (thR') -- (thR' |- bot') -- (thL' |- bot') -- (thL') -- (thL) to[out=90, in=-90] ++(0.6, 0.5) coordinate(p1) -- (p1 |- cap) -- node[ed,pos=1]{$M$} ++(0, 0.8) arc(180:0:0.3) coordinate (p2) -- node[ed,pos=0]{$M^*$} (p2 |- fdl) -- ++(0, -\ecs) arc(-180:0:0.3) -- node[ed,swap]{$M$} (fdl) to[out=90, in=-135] (f);
		\draw[ds] (ful |- ul) -- node[ed]{$Q$} (ful) to[out=-90, in=135] (f) to[out=-45, in=90] (fdr) -- node[ed,swap,near end]{$P$} (fdr |- bot');

		\node[vert] at (f) {$f$};
	\end{tikzpicture}
	& \begin{center} = \end{center} &
	\begin{tikzpicture}
		\bgcylinder{0,0}{7.0}{1.5}{.4}{color1}{color1}

		\coordinate (f) at (1.5, 4.2) {};
		\path (f) -- ++(-0.3, 0.5) coordinate (ful);
		\path (f) -- ++(0.3, 0.5) coordinate (fur);
		\path (f) -- ++(-0.3, -0.5) coordinate (fdl);
		\path (f) -- ++(0.3, -0.5) coordinate (fdr);

		\begin{pgfonlayer}{foreground}
			\drawtheta{(f)}{-0.8}{1}{th}{}
		\end{pgfonlayer}

		\def\ecs{0.2} 
		\filldraw[color2fill] (f) to[out=-135, in=90] (fdl) -- (fdl |- thL) -- node[ed]{$M$} ++(0, -0.5) -- ++(0, -\ecs) arc (0:-180:0.3) -- node[ed]{$M^*$} ++(0, \ecs) to[out=90, in=-90] (thL) -- (thL') -- (thL' |- bot') -- (thR' |- bot') -- (thR') -- (thR) to[out=90, in=-90] ++(-0.6, 0.5) -- ($(fur) + (0.6, 0)$) -- node[ed,swap]{$M^*$} ++(0, \ecs) arc(0:180:0.3) -- node[ed,swap]{$M$} (fur) to[out=-90, in=45] (f);
		\draw[ds] (f) to[out=135, in=-90] (ful) -- node[ed]{$Q$} (ful |- ul);
		\draw[ds] (f) to[out=-45, in=90] (fdr) -- node[ed,near end]{$P$} (fdr |- bot);

		\node[vert] at (f) {$f$};
	\end{tikzpicture}
\end{tabular}
	\caption{String diagram picture of Proposition~\ref{prop:cotrace_of_mate}}
	\label{fig:SD_cotrace_of_mate}
\end{figure}

This equality is shown graphically in Figure~\ref{fig:SD_cotrace_of_mate}; once again, the deformation should be fairly easy to visualize.

We conclude this list of properties with an analogue of cyclicity of the bicategorical trace \cite[Proposition~7.2]{Shadows_and_traces}, which generalizes the familiar fact from linear algebra that $\tr(AB) = \tr(BA)$ for any matrices such that $AB$ and $BA$ are square matrices (even if $A$ and $B$ themselves are not square).

\begin{prop}
\label{prop:cotrace_cyclicity}
Let $M$ and $N$ be right dualizable 1-cells in a closed bicategory with a coshadow. For maps $f : M \triangleright Q_1 \to P_1 \triangleleft N$ and $g : P_2 \odot M \to N \odot Q_2$, the following diagram commutes:
\[\begin{tikzcd}[column sep=4em, row sep=3em]
	\lsh{Q_1 \triangleleft Q_2}
		\arrow[r, "\cotr(g^* f_*)"]
		\arrow[d, "\theta"', "\cong"]
	& \lsh{P_1 \triangleleft P_2}
		\arrow[d, "\theta", "\cong"']
	\\
	\lsh{Q_2 \triangleright Q_1}
		\arrow[r, "\cotr(f_* g^*)"']
	& \lsh{P_2 \triangleright P_1}
\end{tikzcd}\]
where $g^* f_*$ and $f_* g^*$ mean the following:
\begin{multline*}
	g^* f_* : M \triangleright (Q_1 \triangleleft Q_2) \underset{\cong}{\xrightarrow{a^{-1}}} (M \triangleright Q_1) \triangleleft Q_2 \xrightarrow{f_*} (P_1 \triangleleft N) \triangleleft Q_2 \\ 
	\underset{\cong}{\xrightarrow{t^{-1}}} P_1 \triangleleft (N \odot Q_2) \xrightarrow{g^*} P_1 \triangleleft (P_2 \odot M) \underset{\cong}{\xrightarrow{t}} (P_1 \triangleleft P_2) \triangleleft M
\end{multline*}
\begin{multline*}
	f_* g^* : N \triangleright (Q_2 \triangleright Q_1) \underset{\cong}{\xrightarrow{t^{-1}}} (N \odot Q_2) \triangleright Q_1 \xrightarrow{g^*} (P_2 \odot M) \triangleright Q_1 \\
	\underset{\cong}{\xrightarrow{t}} P_2 \triangleright (M \triangleright Q_1) \xrightarrow{f_*} P_2 \triangleright (P_1 \triangleleft N) \underset{\cong}{\xrightarrow{a^{-1}}} (P_2 \triangleright P_1) \triangleleft N
\end{multline*}
\end{prop}

\begin{proof}
The top and right sides of the diagram in Figure~\ref{fig:CD_cotrace_cyclicity} are $\cotr(g^*f_*) : \lsh{Q_1 \triangleleft Q_2} \to \lsh{P_1 \triangleleft P_2}$, while the left and bottom sides of the diagram in Figure~\ref{fig:CD_cotrace_cyclicity_2} are
\[
	\lsh{Q_1 \triangleleft Q_2} \underset{\cong}{\xrightarrow{\theta}} \lsh{Q_2 \triangleright Q_1} \xrightarrow{\cotr(f_*g^*)} \lsh{P_2 \triangleright P_1} \underset{\cong}{\xrightarrow{\theta}} \lsh{P_1 \triangleleft P_2}.
\]
The two diagrams glue together along their other edges. Every unlabeled square commutes because of functoriality of $- \odot -$, $- \triangleright -$, or $- \triangleleft -$ or naturality of $\theta$, $t$, $a$, $\overline{r}$, or $\overline{l}$.
\end{proof}

\begin{figure}[h]
	\centering
	\resizebox{!}{\dimexpr\textheight-0.26in}{\rotatebox{90}{
		\input diagram_cyclicity.tex
	}}
	\caption{Diagram for Proposition~\ref{prop:cotrace_cyclicity}}
	\label{fig:CD_cotrace_cyclicity}
\end{figure}

\begin{figure}[h]
	\centering
	\resizebox{\textwidth}{!}{
		\begin{tikzpicture}[xscale=3.6, yscale=2.8]
	\newcommand{\Xa}{0}
	\newcommand{\Xb}{\Xa+1.0}
	\newcommand{\Xc}{\Xb+1.0}
	\newcommand{\Xd}{\Xc+1.0}
	\newcommand{\Xe}{\Xd+1.0}
	\newcommand{\Xf}{\Xe+1.0}
	\newcommand{\Xg}{\Xf+1.0}
	\newcommand{\Xh}{\Xg+1.05}
	\newcommand{\Xj}{\Xh+1.0}
	\newcommand{\Ya}{0}
	\newcommand{\Yb}{\Ya+1.0}
	\newcommand{\Yc}{\Yb+1.0}
	\newcommand{\Yd}{\Yc+1.0}
	\newcommand{\Ye}{\Yd+1.0}
	\newcommand{\Yf}{\Ye+1.0}
	\newcommand{\Yg}{\Yf+1.0}
	\newcommand{\Yh}{\Yg+1.0}
	\newcommand{\Yi}{\Yh+1.0}
	\newcommand{\Yj}{\Yi+1.0}

	\node(A) at (\Xd, \Yj){$\lsh{Q_1 \triangleleft Q_2}$};
	\node(B) at (\Xe, \Yi){$\lsh{Q_1 \triangleleft N^*NQ_2}$};
	\node(C) at (\Xf, \Yi){$\lsh{Q_1 \triangleleft N^*P_2M}$};
	\node(D) at (\Xf, \Yh){$\lsh{(Q_1 \triangleleft N^*P_2) \triangleleft M}$};
	\node(E) at (\Xf, \Yg){$\lsh{M \triangleright (Q_1 \triangleleft N^*P_2)}$};
	\node(F) at (\Xf, \Yf){$\lsh{(M \triangleright Q_1) \triangleleft N^*P_2}$};
	\node(G) at (\Xf, \Ye){$\lsh{(P_1 \triangleleft N) \triangleleft N^*P_2}$};
	\node(H) at (\Xg, \Yd){$\lsh{P_1 \triangleleft NN^*P_2}$};
	\node(I) at (\Xh, \Yc){$\lsh{P_1 \triangleleft P_2}$};
	\node(J) at (\Xh, \Yb){$\lsh{P_2 \triangleright P_1}$};
	\node(K) at (\Xc, \Yi){$\lsh{Q_2 \triangleright Q_1}$};
	\node(L) at (\Xb, \Yi){$\lsh{U_R \triangleright (Q_2 \triangleright Q_1)}$};
	\node(M) at (\Xb, \Yh){$\lsh{N^*N \triangleright (Q_2 \triangleright Q_1)}$};
	\node(N) at (\Xb, \Yg){$\lsh{N^* \triangleright (N \triangleright (Q_2 \triangleright Q_1))}$};
	\node(O) at (\Xb, \Yf){$\lsh{N^* \triangleright (NQ_2 \triangleright Q_1)}$};
	\node(P) at (\Xb, \Ye){$\lsh{N^* \triangleright (P_2M \triangleright Q_1)}$};
	\node(Q) at (\Xb, \Yd){$\lsh{N^* \triangleright (P_2 \triangleright (M \triangleright Q_1))}$};
	\node(R) at (\Xb, \Yc){$\lsh{N^* \triangleright (P_2 \triangleright (P_1 \triangleleft N))}$};
	\node(S) at (\Xb, \Yb){$\lsh{N^* \triangleright ((P_2 \triangleright P_1) \triangleleft N)}$};
	\node(T) at (\Xc, \Ya){$\lsh{((P_2 \triangleright P_1) \triangleleft N) \triangleleft N^*}$};
	\node(U) at (\Xe, \Ya){$\lsh{(P_2 \triangleright P_1) \triangleleft NN^*}$};
	\node(V) at (\Xg, \Ya){$\lsh{(P_2 \triangleright P_1) \triangleleft U_S}$};
	\node(W) at (\Xd, \Yh){$\lsh{N^*NQ_2 \triangleright Q_1}$};
	\node(Z) at (\Xd, \Yg){$\lsh{N^*P_2M \triangleright Q_1}$};
	\node(A1) at (\Xd, \Yf){$\lsh{N^*P_2 \triangleright (M \triangleright Q_1)}$};
	\node(B1) at (\Xd, \Ye){$\lsh{N^*P_2 \triangleright (P_1 \triangleleft N)}$};
	\node(C1) at (\Xc, \Yb){$\lsh{(P_2 \triangleright (P_1 \triangleleft N)) \triangleleft N^*}$};
	\node(D1) at (\Xd, \Yc){$\lsh{P_2 \triangleright ((P_1 \triangleleft N) \triangleleft N^*)}$};
	\node(E1) at (\Xe, \Yb){$\lsh{P_2 \triangleright (P_1 \triangleleft NN^*)}$};
	\node(G1) at (\Xe, \Yd){$\lsh{((P_1 \triangleleft N) \triangleleft N^*) \triangleleft P_2}$};
	\node(H1) at (\Xf, \Yc){$\lsh{(P_1 \triangleleft NN^*) \triangleleft P_2}$};
	\node(I1) at (\Xg, \Yc){$\lsh{(P_1 \triangleleft U_S) \triangleleft P_2}$};
	\node(F1) at (\Xg, \Yb){$\lsh{P_2 \triangleright (P_1 \triangleleft U_S)}$};

	\draw [->] (A) -- node[above right]{$\lsh{(\varepsilon \odot 1)^*}$} (B);
	\draw [->] (B) -- node[above](BtoC){} (C); \node at ($(BtoC) + (0, 0.1)$) {$\lsh{(1 \odot g)^*}$};
	\draw [->] (C) -- node[right]{$\lsh{t}$} node[below left]{$\cong$} (D);
	\draw [->] (D) -- node[right]{$\theta$} node[left]{$\cong$} (E);
	\draw [->] (E) -- node[right]{$\lsh{a^{-1}}$} node[left]{$\cong$} (F);
	\draw [->] (F) -- node[above right]{$\lsh{f_*}$} (G);
	\draw [->] (G) -- node[above right]{$\lsh{t^{-1}}$} node[below left]{$\cong$} (H);
	\draw [->] (H) -- node[above right]{$\lsh{(\eta \odot 1)^*}$} (I);
	\draw [->] (J) -- node[right]{$\theta$} node[left]{$\cong$} (I);
	\draw [->] (A) -- node[above left]{$\theta$} node[below right]{$\cong$} (K);
	\draw [->] (K) -- node[above]{$\lsh{\overline{r}}$} node[below]{$\cong$} (L);
	\draw [->] (L) -- node[left]{$\lsh{\varepsilon^*}$} (M);
	\draw [->] (M) -- node[left]{$\lsh{t}$} node[right]{$\cong$} (N);
	\draw [->] (N) -- node[left]{$\lsh{t^{-1}_*}$} node[right]{$\cong$} (O);
	\draw [->] (O) -- node[left]{$\lsh{(g^*)_*}$} (P);
	\draw [->] (P) -- node[left]{$\lsh{t_*}$} node[right]{$\cong$} (Q);
	\draw [->] (Q) -- node[left]{$\lsh{f_{**}}$} (R);
	\draw [->] (R) -- node[left]{$\lsh{a^{-1}_*}$} node[right]{$\cong$} (S);
	\draw [->] (S) -- node[below left]{$\theta$} node[above right]{$\cong$} (T);
	\draw [->] (T) -- node[below]{$\lsh{t^{-1}}$} node[above]{$\cong$} (U);
	\draw [->] (U) -- node[below]{$\lsh{\eta^*}$} (V);
	\draw [->] (V) -- node[below right]{$\lsh{\overline{l}^{-1}}$} node[above left]{$\cong$} (J);
	\draw [->] (B) -- node[above left]{$\theta$} node[below right]{$\cong$} (W);
	\draw [->] (K) -- node[above right]{$\lsh{(\varepsilon \odot 1)^*}$} (W);
	\draw [->] (W) -- node[above]{$\lsh{t}$} node[below]{$\cong$} (M);
	\draw [->] (W) -- node[right,pos=0.3]{$\lsh{(1 \odot g)^*}$} (Z);
	\draw [->] (W) -- node[above left]{$\lsh{t}$} node[below right]{$\cong$} (O);
	\draw [->] (Z) -- node[above left]{$\lsh{t}$} node[below right]{$\cong$} (P);
	\draw [->] (C) -- node[above left]{$\theta$} node[below right]{$\cong$} (Z);
	\draw [->] (Z) -- node[left]{$\lsh{t}$} node[right]{$\cong$} (A1);
	\draw [->] (A1) -- node[above]{$\theta$} node[below]{$\cong$} (F);
	\draw [->] (A1) -- node[above left]{$\lsh{t}$} node[below right]{$\cong$} (Q);
	\draw [->] (B1) -- node[above left]{$\lsh{t}$} node[below right]{$\cong$} (R);
	\draw [->] (C1) -- node[above left]{$\lsh{a}$} node[below right]{$\cong$} (D1);
	\draw [->] (A1) -- node[right]{$\lsh{f_*}$} (B1);
	\draw [->] (B1) -- node[above]{$\theta$} node[below]{$\cong$} (G);
	\draw [->] (R) -- node[above right]{$\theta$} node[below left]{$\cong$} (C1);
	\draw [->] (D1) -- node[above left]{$\theta$} node[below right]{$\cong$} (G1);
	\draw [->] (G) -- node[above left]{$\lsh{t}$} node[below right]{$\cong$} (G1);
	\draw [->] (G1) -- node[above right]{$\lsh{t^{-1}_*}$} node[below left]{$\cong$} (H1);
	\draw [->] (D1) -- node[above right]{$\lsh{t^{-1}_*}$} node[below left]{$\cong$} (E1);
	\draw [->] (C1) -- node[left]{$\lsh{a^{-1}_*}$} node[right]{$\cong$} (T);
	\draw [->] (E1) -- node[left]{$\lsh{a^{-1}}$} node[right]{$\cong$} (U);
	\draw [->] (H) -- node[above left]{$\lsh{t}$} node[below right]{$\cong$} (H1);
	\draw [->] (E1) -- node[above left]{$\theta$} node[below right]{$\cong$} (H1);
	\draw [->] (H1) -- node[below]{$\lsh{(\eta^*)_*}$} (I1);
	\draw [->] (I1) -- node[below]{$\lsh{\overline{l}^{-1}_*}$} node[above]{$\cong$} (I);
	\draw [->] (F1) -- node[left]{$\theta$} node[right]{$\cong$} (I1);
	\draw [->] (F1) -- node[above]{$\lsh{\overline{l}^{-1}_*}$} node[below]{$\cong$} (J);
	\draw [->] (E1) -- node[above]{$\lsh{(\eta^*)_*}$} (F1);
	\draw [->] (F1) -- node[left]{$\lsh{a^{-1}}$} node[right]{$\cong$} (V);

	\node at (barycentric cs:L=1,K=1,M=1,W=1)[align=center]{Lemma \ref{lem:hom_unit_iso_compatibility_with_t_2} \\ and naturality of $t$};
	\node at (barycentric cs:M=2,N=2,W=1,Z=1){Lemma \ref{lem:pentagon_a_and_t_1}};
	\node at (barycentric cs:Q=1,Z=1){Lemma \ref{lem:pentagon_a_and_t_1}};
	\node at (barycentric cs:Z=1,E=1){Definition \ref{defn:coshadow}};
	\node at (barycentric cs:T=1,E1=1){Lemma \ref{lem:pentagon_a_and_t_2}};
	\node at (barycentric cs:B1=1,D1=1){Definition \ref{defn:coshadow}};
	\node at (barycentric cs:F1=1.5,J=1.5,V=1){Lemma \ref{lem:hom_unit_iso_compatibility_with_a}};
	\node at (barycentric cs:H=1,I1=1.5)[align=center]{Lemma \ref{lem:hom_unit_iso_compatibility_with_t} \\ and naturality of $t$};
\end{tikzpicture}
	}
	\caption{Diagram for Proposition~\ref{prop:cotrace_cyclicity}}
	\label{fig:CD_cotrace_cyclicity_2}
\end{figure}

\begin{figure}[h]
	\centering
	\resizebox{\textwidth}{!}{
		\begin{tikzpicture}[xscale=3.5, yscale=2]
	\node(F2) at (0, 4){$\lsh{M^*M \triangleright (Q_1 \triangleleft N^*P_2M)}$};
	\node(G2) at (2, 4){$\lsh{M^* \triangleright (M \triangleright (Q_1 \triangleleft N^*P_2M))}$};
	\node(H2) at (4, 4){$\lsh{M^* \triangleright ((M \triangleright Q_1) \triangleleft N^*P_2M)}$};
	\node(L2) at (0, 2){$\lsh{M^*M \triangleright ((Q_1 \triangleleft N^*P_2) \triangleleft M)}$};
	\node(M2) at (4, 2){$\lsh{M^* \triangleright (((M \triangleright Q_1) \triangleleft N^*P_2) \triangleleft M)}$};
	\node(Q2) at (0, 0){$\lsh{(M^*M \triangleright (Q_1 \triangleleft N^*P_2)) \triangleleft M}$};
	\node(R2) at (2, 0){$\lsh{(M^* \triangleright (M \triangleright (Q_1 \triangleleft N^*P_2))) \triangleleft M}$};
	\node(S2) at (4, 0){$\lsh{(M^* \triangleright ((M \triangleright Q_1) \triangleleft N^*P_2)) \triangleleft M}$};
	\node(W2) at (1, 3){$\lsh{(M^*M \triangleright Q_1) \triangleleft N^*P_2M}$};
	\node(Z2) at (3, 3){$\lsh{(M^* \triangleright (M \triangleright Q_1)) \triangleleft N^*P_2M}$};
	\node(A3) at (1, 1){$\lsh{((M^*M \triangleright Q_1) \triangleleft N^*P_2) \triangleleft M}$};
	\node(B3) at (3, 1){$\lsh{((M^* \triangleright (M \triangleright Q_1)) \triangleleft N^*P_2) \triangleleft M}$};

	\draw [->] (F2) -- node[above]{$\lsh{t}$} node[below]{$\cong$} (G2);
	\draw [->] (G2) -- node[above]{$\lsh{a^{-1}_*}$} node[below]{$\cong$} (H2);
	\draw [->] (H2) -- node[right]{$\lsh{t_*}$} node[left]{$\cong$} (M2);
	\draw [->] (M2) -- node[right]{$\lsh{a^{-1}}$} node[left]{$\cong$} (S2);
	\draw [->] (F2) -- node[left]{$\lsh{t_*}$} node[right]{$\cong$} (L2);
	\draw [->] (L2) -- node[left]{$\lsh{a^{-1}}$} node[right]{$\cong$} (Q2);
	\draw [->] (Q2) -- node[below]{$\lsh{t_*}$} node[above]{$\cong$} (R2);
	\draw [->] (R2) -- node[below]{$\lsh{a^{-1}_{**}}$} node[above]{$\cong$} (S2);

	\draw [->] (W2) -- node[above]{$\lsh{t_*}$} node[below]{$\cong$} (Z2);
	\draw [->] (W2) -- node[left, pos=0.25]{$\lsh{t}$} node[right, pos=0.25]{$\cong$} (A3);
	\draw [->] (Z2) -- node[right, pos=0.75]{$\lsh{t}$} node[left, pos=0.75]{$\cong$} (B3);
	\draw [->] (A3) -- node[below]{$\lsh{t_{**}}$} node[above]{$\cong$} (B3);
	\draw [->] (F2) -- node[above right, pos=0.8]{$\lsh{a^{-1}}$} node[below left]{$\cong$} (W2);
	\draw [->] (B3) -- node[below left, pos=0.2]{$\lsh{a_*}$} node[above right]{$\cong$} (S2);
	\draw [->] (Z2) -- node[above left, pos=0.2]{$\lsh{a}$} node[below right]{$\cong$} (H2);
	\draw [->] (Q2) -- node[below right, pos=0.8]{$\lsh{a^{-1}_*}$} node[above left]{$\cong$} (A3);

	\node at (barycentric cs:F2=1,H2=1,W2=1,Z2=1){Lemma \ref{lem:pentagon_a_and_t_2}};
	\node at (barycentric cs:A3=1,B3=1,Q2=1,S2=1){Lemma \ref{lem:pentagon_a_and_t_2}};
	\node at (barycentric cs:A3=1.7,F2=1,Q2=1.7,W2=1){Lemma \ref{lem:pentagon_a_and_t_2}};
	\node at (barycentric cs:H2=1.7,Z2=1.7,B3=1,S2=1){Lemma \ref{lem:pentagon_a_and_t_2}};
\end{tikzpicture}
	}
	\caption{Subdiagram of Figure~\ref{fig:CD_cotrace_cyclicity}}
	\label{fig:CD_cotrace_cyclicity_3}
\end{figure}

This equality is shown graphically in Figure~\ref{fig:SD_cyclicity}; the deformation is achieved by sliding $f$ to the right and around the back of the cylinder.

\begin{figure}[h]
	\centering
	\begin{tabular}{m{36mm}m{20mm}m{36mm}}
	\begin{tikzpicture}
		\bgcylinder{0,0}{8.4}{1.8}{.4}{color1}{color1}

		\coordinate (f) at (2.1, 6.5);
		\path (f) -- ++(-0.6, -1.3) coordinate (g);
		\path (f) -- ++(-0.3, 0.5) coordinate (ful);
		\path (f) -- ++(0.3, 0.5) coordinate (fur);
		\path (f) -- ++(-0.3, -0.5) coordinate (fdl);
		\path (f) -- ++(0.3, -0.5) coordinate (fdr);
		\path (g) -- ++(-0.3, 0.5) coordinate (gul);
		\path (g) -- ++(0.3, 0.5) coordinate (gur);
		\path (g) -- ++(-0.3, -0.5) coordinate (gdl);
		\path (g) -- ++(0.3, -0.5) coordinate (gdr);

		\begin{pgfonlayer}{foreground}
			\drawtheta{(g)}{-0.7}{1}{th}{}
			\drawtheta{(g)}{-3.2}{1}{th2}{ds}
		\end{pgfonlayer}

		\def\ecs{0.1} 
		\filldraw[color2fill] (ful |- ul) -- (ful) to[out=-90, in=135] (f) to[out=-135, in=90] (fdl) -- (gur) to[out=-90, in=45] (g) to[out=135, in=-90] (gul) -- node[ed,pos=0.75]{$Q_2$} (gul |- ul);
		\filldraw[color3fill] (f) to[out=45, in=-90] (fur) -- node[ed,pos=1]{$M$} ++(0, \ecs) arc(180:0:0.3) -- node[ed,pos=0]{$M^*$} ($(thR) + (-0.6, 0.5)$) to[out=-90, in=90](thR) -- (thR') -- (thR' |- bot') -- (thL' |- bot') -- (thL') -- (thL) to[out=-90, in=90] ++(0.6, -0.5) -- node[ed,swap,pos=1]{$M^*$} ++(0, -\ecs) arc(-180:0:0.3) -- node[ed,swap,pos=0]{$M$} (gdl) to[out=90, in=-135] (g) to[out=45, in=-90] (gur) -- node[ed]{$N$} (fdl) to[out=90, in=-135] (f);
		\filldraw[color4fill] (gdr |- bot') -- node[ed,swap,pos=0.07]{$P_2$} (gdr) to[out=90, in=-45] (g) to[out=45, in=-90] (gur) -- (fdl) to[out=90, in=-135] (f) to[out=-45, in=90] (fdr) -- ($(th2R) + (-1.2, 1.0)$) to[out=-90, in=90] (th2R) -- (th2R') -- (th2R' |- bot');
		\filldraw[color4fill] (th2L) to[out=-90, in=90] ++(1.2, -1.0) -- (gdl |- bot') -- (th2L' |- bot') -- (th2L');

		\draw[ds] (ful |- ul) -- node[ed]{$Q_1$} (ful) to[out=-90, in=135] (f) to[out=-45, in=90] (fdr) -- node[ed]{$P_1$} ($(th2R) + (-1.2, 1.0)$) to[out=-90, in=90] (th2R);
		\draw[ds] (th2L) to[out=-90, in=90] ++(1.2, -1.0) -- node[ed,swap,pos=0.2]{$P_1$} (gdl |- bot');

		\node[vert] at (f) {$f$};
		\node[vert] at (g) {$g$};
	\end{tikzpicture}
	& \begin{center} = \end{center} &
	\begin{tikzpicture}
		\bgcylinder{0,0}{8.4}{1.8}{.4}{color4}{color1}

		\coordinate (g) at (2.1, 4.1);
		\path (g) -- ++(-0.6, -1.3) coordinate (f);
		\path (f) -- ++(-0.3, 0.5) coordinate (ful);
		\path (f) -- ++(0.3, 0.5) coordinate (fur);
		\path (f) -- ++(-0.3, -0.5) coordinate (fdl);
		\path (f) -- ++(0.3, -0.5) coordinate (fdr);
		\path (g) -- ++(-0.3, 0.5) coordinate (gul);
		\path (g) -- ++(0.3, 0.5) coordinate (gur);
		\path (g) -- ++(-0.3, -0.5) coordinate (gdl);
		\path (g) -- ++(0.3, -0.5) coordinate (gdr);

		\begin{pgfonlayer}{foreground}
			\drawtheta{(g)}{2.5}{1}{th}{ds}
			\drawtheta{(f)}{-0.7}{1}{th2}{}
		\end{pgfonlayer}

		\def\ecs{0.1} 
		\filldraw[color3fill] (fdr |- bot') -- (fdr) to[out=90, in=-45] (f) to[out=45, in=-90] (fur) -- (gdl) to[out=90, in=-135] (g) to[out=-45, in=90] (gdr) -- node[ed,pos=0.9]{$P_2$} (gdr |- bot');
		\filldraw[color2fill] (th2L) to[out=-90, in=90] ++(0.6, -0.5) -- node[ed,swap,pos=1]{$N^*$} ++(0, -\ecs) arc(-180:0:0.3) -- node[ed,swap,pos=0]{$N$} (fdl) to[out=90, in=-135] (f) to[out=45, in=-90] (fur) -- (gdl) to[out=90, in=-135] (g) to[out=45, in=-90] (gur) -- node[ed,pos=1]{$N$} ++(0, \ecs) arc(180:0:0.3) -- node[ed,pos=0]{$N^*$} ($(th2R) + (-0.6, 0.5)$) to[out=-90, in=90] (th2R) -- (th2R') -- (ur') -- (ul') -- (th2L');
		\filldraw[color1fill] (thL) to[out=-90, in=90] ++(1.2, -1.0) -- node[ed,swap]{$Q_1$} (ful) to[out=-90, in=135] (f) to[out=45, in=-90] (fur) -- node[ed]{$M$} (gdl) to[out=90, in=-135] (g) to[out=135, in=-90] (gul) -- node[ed,pos=0.8]{$Q_2$} (gul |- ul') -- (ul') -- (thL');
		\filldraw[color1fill] (thR) to[out=90, in=-90] ++(-1.2, 1.0) -- node[ed,swap,pos=0.2]{$Q_1$} (gur |- ur') -- (ur') -- (thR');

		\draw[ds] (thR) to[out=90, in=-90] ++(-1.2, 1.0) -- (gur |- ur');
		\draw[ds] (thL) to[out=-90, in=90] ++(1.2, -1.0) -- (ful) to[out=-90, in=135] (f);
		\draw[ds] (fdr |- bot') -- node[ed,pos=0.12]{$P_1$} (fdr) to[out=90, in=-45] (f);

		\node[vert] at (f) {$f$};
		\node[vert] at (g) {$g$};
	\end{tikzpicture}
\end{tabular}
	\caption{String diagram picture of Proposition~\ref{prop:cotrace_cyclicity}}
	\label{fig:SD_cyclicity}
\end{figure}

\section{Interplay between traces and cotraces}
\label{sec:interplay}

The original motivation for our study of cotraces was \cite{Lipman1987}, in which both traces and cotraces arise and interact with each other. This interaction is mediated by a ``pairing'' map from a shadow and coshadow to a second shadow.

\begin{example}
\label{ex:pairing_example}
If $M$ and $N$ are $(R, R)$-bimodules, we have
\[
	\HH^0(R, M) \otimes \HH_0(R, N) \xrightarrow{\rho} \HH_0(R, M \otimes_R N)
\]
taking $m \otimes [n]$ to $[m \otimes n]$. In fact, there is a pairing
\[
	\HH^n(R, M) \otimes \HH_n(R, N) \xrightarrow{\rho} \HH_0(R, M \otimes_R N)
\]
for any $n \geq 0$.
\end{example}

The main result of \cite{Lipman1987}, Proposition 4.5.4, is a relation between traces and cotraces. In the case that $n = 0$, it takes the following form, where $R$ is a commutative ring, $M$ is an $R$-module considered as an $(R, R)$-bimodule, and $F$ is a finitely generated projective right $R$-module viewed as an $(S, R)$-bimodule, where $S := \Hom_R(F, F)$:
\begin{equation}
\label{eq:prop_4_5_4}
\begin{tikzpicture}[xscale=3.25, yscale=1.5]
	\node (A) at (1, 3) {$\HH^0(R, M) \otimes \HH_0(S, S)$};
	\node (B) at (0, 2) {$\HH^0(S, \Hom_R(F, F \otimes_R M)) \otimes \HH_0(S, S)$};
	\node (C) at (0, 1) {$\HH_0(S, \Hom_R(F, F \otimes_R M))$};
	\node (D) at (1, 0) {$\HH_0(R, M)$};
	\node (E) at (2, 2) {$\HH^0(R, M) \otimes \HH_0(R, R)$};
	\node (F) at (2, 1) {$\HH_0(R, M)$};

	\draw [->] (A) -- node[above left]{$\cotr \otimes 1$} (B);
	\draw [->] (B) -- node[left]{$\rho$} (C);
	\draw [->] (C) -- node[below left]{$\tr$} (D);
	\draw [->] (A) -- node[above right]{$1 \otimes \tr$} (E);
	\draw [->] (E) -- node[right]{$\rho$} (F);
	\draw [double equal sign distance] (F) -- (D);
\end{tikzpicture}
\end{equation}
The trace on the right side is the trace of $\id_F$, the trace on the left side is the trace of $\ev : (F \triangleright (F \odot M)) \odot F \to F \odot M$, and the cotrace takes $m \in \HH^0(R, M)$ to the homomorphism $x \mapsto x \otimes m$ in
\[
	\HH^0(S, \Hom_R(F, F \otimes_R M)).
\]
If $M = R$, this reduces to the diagram of (\ref{eq:interplay_basic_example}). Symbolically, (\ref{eq:prop_4_5_4}) asserts that if $m \in \HH^0(R, M)$ and $[\varphi] \in \HH_0(S, S)$, then
\[
\tr(\rho(\cotr(m), [\varphi])) = \rho(m, \tr([\varphi])),
\]
which looks like a kind of ``adjointness'' between trace and cotrace.

More generally, suppose we have shadows $\sh{-}$ and $\bluesh{-}$ and a coshadow $\bluelsh{-}$, all taking values in the same monoidal category. Suppose also that there are maps
\[
\rho_{M, N} : \bluelsh{M} \otimes \bluesh{N} \to \sh{M \odot N}
\]
which are natural in $M$ and $N$. If $\rho$ is appropriately compatible with the cyclicity isomorphisms $\theta$ for the coshadow and shadows, then cotraces with respect to $\bluelsh{-}$ and traces with respect to $\bluesh{-}$ and $\sh{-}$ satisfy the same sort of adjointness that Lipman's traces and cotraces do. We record the necessary compatibility in the following definition.

\begin{defn}
\label{defn:pairing}
A \KEYWORD{pairing} $\rho$ between a coshadow $\bluelsh{-}$, a shadow $\bluesh{-}$, and another shadow $\sh{-}$, all taking values in the same monoidal category, is a family of maps $\rho_{M, N} : \bluelsh{M} \otimes \bluesh{N} \to \sh{M \odot N}$ which are natural in $M$ and $N$ and such that the following diagram commutes whenever it makes sense:
\[\begin{tikzcd}
	\bluelsh{Z \triangleleft X} \otimes \bluesh{Y \odot X}
		\arrow[r, "\theta \otimes \theta", "\cong"']
		\arrow[d, "\rho"']
	& \bluelsh{X \triangleright Z} \otimes \bluesh{X \odot Y}
		\arrow[d, "\rho"]
	\\
	\sh{(Z \triangleleft X) \odot Y \odot X}
		\arrow[d, "\theta"', "\cong"]
	& \sh{(X \triangleright Z) \odot X \odot Y}
		\arrow[d, "\sh{\ev \odot 1}", "\cong"']
	\\
	\sh{X \odot (Z \triangleleft X) \odot Y}
		\arrow[r, "\sh{\ev \odot 1}"']
	& \sh{Z \odot Y}
\end{tikzcd}\]
\end{defn}

With $\HH^0$ as the coshadow and $\HH_0$ playing the role of both shadows, the maps of Example~\ref{ex:pairing_example} form a pairing. We are finally in a position to state and prove a precise version of Theorem~\ref{thm:main_result_intro}.

\begin{thm}
\label{thm:interplay}
Let $\mathbf{T}$ be a monoidal category and $\mathscr{B}$ a closed bicategory with shadow functors $\sh{-}$ and $\bluesh{-}$ and coshadow $\bluelsh{-}$, all with target $\mathbf{T}$. Suppose also that there is a pairing $\rho : \bluelsh{-} \otimes \bluesh{-} \to \sh{- \odot -}$. Let $F$ and $G$ be right dualizable 1-cells in $\mathscr{B}$, and let
\begin{align*}
	\xi &: Q \odot F \to F \odot P \\
	\gamma &: F \triangleright M \to N \triangleleft F \\
	\zeta &: N \odot Q \odot F \odot G \to F \odot G \odot Z \\
	\delta &: M \odot P \odot G \to G \odot Z
\end{align*}
be 2-cells in $\mathscr{B}$ such that the following commutes:
\[\begin{tikzcd}
	F \odot (F \triangleright M) \odot Q \odot F \odot G
		\arrow[r, "1 \odot \gamma \odot 1^3"]
		\arrow[d, "1^2 \odot \xi \odot 1"']
	& F \odot (N \triangleleft F) \odot Q \odot F \odot G
		\arrow[d, "\ev \odot 1^3"]
	\\
	F \odot (F \triangleright M) \odot F \odot P \odot G
		\arrow[d, "1 \odot \ev \odot 1^2"']
	& N \odot Q \odot F \odot G
		\arrow[d, "\zeta"]
	\\
	F \odot M \odot P \odot G
		\arrow[r, "1 \odot \delta"']
	& F \odot G \odot Z
\end{tikzcd}\]
Then the following commutes:
\[\begin{tikzcd}[column sep=1.5cm]
	& \bluelsh{M} \otimes \bluesh{Q}
		\arrow[dl, "\cotr_F(\gamma) \otimes 1"']
		\arrow[dr, "1 \otimes \tr_F(\xi)"]
	\\
	\bluelsh{N} \otimes \bluesh{Q}
		\arrow[d, "\rho"']
	&& \bluelsh{M} \otimes \bluesh{P}
		\arrow[d, "\rho"]
	\\
	\sh{N \otimes Q}
		\arrow[dr, "\tr_{F \odot G}(\zeta)"']
	&& \sh{M \odot P}
		\arrow[dl, "\tr_G(\delta)"]
	\\
	& \sh{Z}
\end{tikzcd}\]
\end{thm}

\begin{proof}
In the diagram in Figure~\ref{fig:interplay}, the left-hand side is
\[
	\tr_G(\delta) \circ \rho \circ (1 \otimes \tr_F(\xi))
\]
and the right-hand side is
\[
	\tr_{F \odot G}(\zeta) \circ \rho \circ (\cotr_F(\gamma) \otimes 1).
\]
We follow the convention discussed at the start of Section~\ref{sec:cotrace_basic_properties}, omitting the symbol $\odot$. Every unlabeled square commutes because of naturality of $\rho$, $\theta$, or $\ev$ or functoriality of $- \odot -$ or $- \otimes -$.
\end{proof}

\begin{figure}[h]
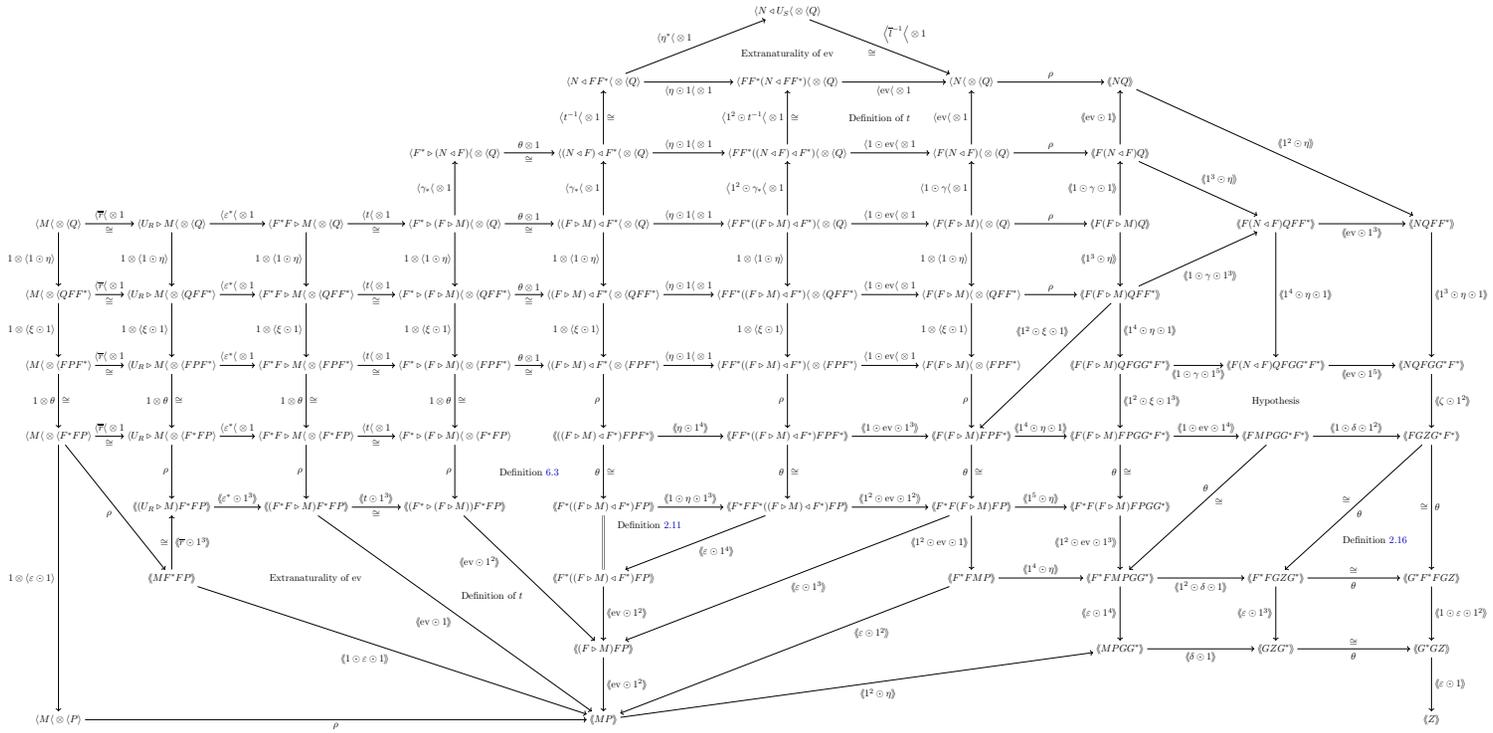

	\centering
	\resizebox{!}{\dimexpr\textheight-0.26in}{\rotatebox{90}{
		\input diagram_interplay.tex
	}}
	\caption{Diagram for Theorem \ref{thm:interplay}}
	\label{fig:interplay}
\end{figure}

The situation in which Lipman's result arises is as follows. If we let $Q$, $P$, and $G$ be unit 1-cells, and let $Z = M$, $\xi = \id_F$, and $\delta = \id_M$, then the hypothesis of Theorem~\ref{thm:interplay} reduces to the following:
\[\begin{tikzcd}
	F \odot (F \triangleright M) \odot F
		\arrow[r, "1 \odot \gamma \odot 1"]
		\arrow[d, "1 \odot \ev"']
	& F \odot (N \triangleleft F) \odot F
		\arrow[d, "\ev \odot 1"]
	\\
	F \odot M
	& N \odot F
		\arrow[l, "\zeta"]
\end{tikzcd}\]
For an example of a collection of objects and maps making this diagram commute, start with any 1-cells $F : S \pto R$ and $M : R \pto R$ (with $F$ right dualizable), and let $N = F \triangleright (F \odot M)$. Then let $\zeta$ be
\[
	\ev : (F \triangleright (F \odot M)) \odot F \to F \odot M
\]
and let
\[
	\gamma : F \triangleright M \to (F \triangleright (F \odot M)) \triangleleft F
\]
be the adjoint of the map $\mu : F \odot (F \triangleright M) \to F \triangleright (F \odot M)$ of Section~\ref{sec:duality_closed_bicat}. The hypothesis of Theorem~\ref{thm:interplay} is satisfied since the following diagram commutes:
\[\begin{tikzcd}
	F \odot (F \triangleright M) \odot F
		\arrow[r, "1 \odot \overline{\mu} \odot 1"]
		\arrow[dr, "\mu \odot 1"]
		\arrow[d, "1 \odot \ev"']
	& F \odot ((F \triangleright (F \odot M)) \triangleleft F) \odot F
		\arrow[d, "\ev \odot 1"]
	\\
	F \odot M
	& (F \triangleright (F \odot M)) \odot F
		\arrow[l, "\ev"']
\end{tikzcd}\]

\begin{example}
Start with a ring $R$, a finitely generated projective right $R$-module $F$, and an $(R, R)$-bimodule $M$. Consider the setup above: $N := F \triangleright (F \odot M) = \Hom_R(F, F \otimes_R M)$, $\zeta = \ev$, $\gamma = \overline{\mu}$, $\xi = \id_F$, $\delta = \id_M$. Choosing $\HH^0$ and $\HH_0$ as the (co)shadows and using the pairing of Example~\ref{ex:pairing_example}, Theorem~\ref{thm:interplay} recovers (\ref{eq:prop_4_5_4}), which is \cite[Proposition 4.5.4]{Lipman1987} in the case $n = 0$.
\end{example}

\section{Functoriality of cotrace}
\label{sec:functoriality}

One of the most important properties of trace is that it is preserved by lax functors, or at least those which preserve the dual pair relevant to the trace in question \cite[Proposition~8.3]{Shadows_and_traces}. We state a similar result for cotraces, after describing the structure of a lax functor between closed bicategories.

A lax functor $F : \mathscr{B} \to \mathscr{C}$ between bicategories is compatible with the horizontal composition in $\mathscr{B}$ and $\mathscr{C}$, in the sense that it comes equipped with coherence 2-cells $F(M) \odot F(N) \to F(M \odot N)$ in $\mathscr{C}$ for each pair of 1-cells $M$ and $N$ in $\mathscr{B}$. This might lead us to ask a lax functor between \emph{closed} bicategories to include similar compatibility with the internal hom-functors, but in fact this is automatic: for example, we get a map $F(M \triangleright N) \to F(M) \triangleright F(N)$ as the transpose of
\[
	F(M \triangleright N) \odot F(M) \to F((M \triangleright N) \odot M) \xrightarrow{F(\ev)} F(N).
\]
Thus a lax functor between closed bicategories is nothing more than a lax functor between the underlying bicategories. However, the functoriality result we want and the concomitant notion of lax coshadow functor use the coherence 2-cells for $\triangleright$ and $\triangleleft$ rather than the ones for $\odot$, so we present a definition of lax functor between closed bicategories in terms of the former.

In the following definition we make use of the transpose $\overline{\circ} : Y \triangleright Z \to ((X \triangleright Y) \triangleright (X \triangleright Z))$ of the ``composition'' map $\circ : (Y \triangleright Z) \odot (X \triangleright Y) \to X \triangleright Z$, which itself is is the transpose of
\[
	(Y \triangleright Z) \odot (X \triangleright Y) \odot X \xrightarrow{1 \odot \ev} (Y \triangleright Z) \odot Y \xrightarrow{\ev} Z.
\]

\begin{defn}
\label{defn:lax_closed_functor}
Let $\mathscr{B}$ and $\mathscr{C}$ be closed bicategories. A \KEYWORD{lax closed functor} $F : \mathscr{B} \to \mathscr{C}$ is
\begin{itemize}
	\item A function $F_0 : \mathrm{ob} \mathscr{B} \to \mathrm{ob} \mathscr{C}$
	\item For each $R, S \in \mathrm{ob} \mathscr{B}$, a functor $F_{R,S} : \mathscr{B}(R, S) \to \mathscr{C}(F_0(R), F_0(S))$
	\item Natural transformations $c : F_{R,S}(N \triangleright P) \to F_{S,T}(N) \triangleright F_{R,T}(P)$
	\item Natural transformations $c : F_{S,T}(P \triangleleft M) \to F_{R,T}(P) \triangleleft F_{R,S}(M)$
	\item Maps $i : U_{F_0(R)} \to F_{R,R}(U_R)$
\end{itemize}
such that the following diagrams commute whenever they make sense (we usually suppress the subscripts of $F$ when they are clear from context):
\[\begin{tikzcd}[column sep=0.95em]
	F(N \triangleright P)
		\arrow[r, "\overline{\circ}"]
		\arrow[d, "c"']
	& F((M \triangleright N) \triangleright (M \triangleright P))
		\arrow[r, "c"]
	& F(M \triangleright N) \triangleright F(M \triangleright P)
		\arrow[d, "c_*"]
	\\
	F(N) \triangleright F(P)
		\arrow[r, "\overline{\circ}"]
	& (F(M) \triangleright F(N)) \triangleright (F(M) \triangleright F(P))
		\arrow[r, "c^*"]
	& F(M \triangleright N) \triangleright (F(M) \triangleright F(P))
\end{tikzcd}\]
\[\begin{tikzcd}
	U_{F(R)}
		\arrow[r, "i"]
		\arrow[d, "\overline{l}"']
	& F(U_R)
		\arrow[d, "F(\overline{l})"]
	\\
	F(M) \triangleright F(M)
	& F(M \triangleright M)
		\arrow[l, "c"]
\end{tikzcd}
\hspace{1cm}
\begin{tikzcd}
	F(M)
		\arrow[r, "F(\overline{r})", "\cong"']
		\arrow[d, "\overline{r}"', "\cong"]
	& F(U_S \triangleright M)
		\arrow[d, "c"]
	\\
	U_{F(S)} \triangleright F(M)
	& F(U_S) \triangleright F(M)
		\arrow[l, "i^*"]
\end{tikzcd}\]
\end{defn}
\noindent along with similar diagrams for the other hom-functor $- \triangleleft -$, and the following diagram relating the maps $c$ for the two hom-functors:
\[\begin{tikzcd}
	F((M \triangleright N) \triangleleft P)
		\arrow[r, "c"]
		\arrow[d, "F(a)"', "\cong"]
	& F(M \triangleright N) \triangleleft F(P)
		\arrow[r, "c_*"]
	& (F(M) \triangleright F(N)) \triangleleft F(P)
		\arrow[d, "a", "\cong"']
	\\
	F(M \triangleright (N \triangleleft P))
		\arrow[r, "c"']
	& F(M) \triangleright F(N \triangleleft P)
		\arrow[r, "c_*"']
	& F(M) \triangleright (F(N) \triangleleft F(P))
\end{tikzcd}\]

\begin{defn}
\label{defn:lax_coshadow_functor}
Let $\mathscr{B}$ and $\mathscr{C}$ be closed bicategories equipped with coshadows $\lsh{-}_{\mathscr{B}}$ and $\lsh{-}_{\mathscr{C}}$ with target categories $\mathbf{T}$ and $\mathbf{Z}$, respectively. A \KEYWORD{lax coshadow functor} is a lax closed functor $F : \mathscr{B} \to \mathscr{C}$ together with a functor $F_{\cosh} : \mathbf{T} \to \mathbf{Z}$ and a natural transformation
\[
\phi : F_{\cosh} \circ \lsh{-}_{\mathscr{B}} \to \lsh{-}_{\mathscr{C}} \circ F
\]
such that the following diagram commutes whenever it makes sense:
\[\begin{tikzcd}
	F_{\cosh}\lsh{M \triangleright N}
		\arrow[r, "F_{\cosh}(\theta)", "\cong"']
		\arrow[d, "\phi"']
	& F_{\cosh}\lsh{N \triangleleft M}
		\arrow[d, "\phi"]
	\\
	\lsh{F(M \triangleright N)}
		\arrow[d, "\lsh{c}"']
	& \lsh{F(N \triangleleft M)}
		\arrow[d, "\lsh{c}"]
	\\
	\lsh{F(M) \triangleright F(N)}
		\arrow[r, "\theta"', "\cong"]
	& \lsh{F(N) \triangleleft F(M)}
\end{tikzcd}\]
\end{defn}

We will make use of the following result in proving that lax closed functors preserve cotraces.

\begin{lem}
\label{lem:closed_functor_and_tensor_hom_adjunction}
If $F : \mathscr{B} \to \mathscr{C}$ is a lax closed functor, the following commutes:
\[\begin{tikzcd}
	F((M \odot N) \triangleright P)
		\arrow[r, "c"]
		\arrow[d, "F(t)"', "\cong"]
	& F(M \odot N) \triangleright F(P)
		\arrow[r, "c^*"]
	& (F(M) \odot F(N)) \triangleright F(P)
		\arrow[d, "t", "\cong"']
	\\
	F(M \triangleright (N \triangleright P))
		\arrow[r, "c"']
	& F(M) \triangleright F(N \triangleright P)
		\arrow[r, "c_*"']
	& F(M) \triangleright (F(N) \triangleright F(P))
\end{tikzcd}\]
\end{lem}

\begin{proof}
Both sides are the transpose (via the tensor-hom adjunction for $F(M)$) of the transpose (via the tensor-hom adjunction for $F(N)$) of
\[
F((M \odot N) \triangleright P) \odot F(M) \odot F(N) \xrightarrow{c \circ (1 \odot c)} F(((M \odot N) \triangleright P) \odot M \odot N) \xrightarrow{F(\ev)} F(P).
\]
\end{proof}

\begin{prop}
\label{prop:cotrace_functoriality}
Let $F : \mathscr{B} \to \mathscr{C}$ be a lax coshadow functor and $M \in \mathscr{B}(R, S)$ a right dualizable 1-cell with right dual $M^*$.
\begin{enumerate}
	\item \cite[Proposition 4.3.6]{Ponto_thesis} If the maps $c : F(M) \odot F(M^*) \to F(M \odot M^*)$ and $i : U_{F(S)} \to F(U_S)$ are isomorphisms, then $F(M)$ is right dualizable with right dual $F(M^*)$.
	\item If, furthermore, the map $c_{M,Q} : F(M \triangleright Q) \to F(M) \triangleright F(Q)$ is an isomorphism, then for any 2-cell $f : M \triangleright Q \to P \triangleleft M$, the following commutes:
	\[\begin{tikzcd}[column sep=3.5cm]
		F\lsh{Q}_{\mathscr{B}}
			\arrow[r, "F(\cotr(f))"]
			\arrow[d, "\phi"']
		& F\lsh{P}_{\mathscr{B}}
			\arrow[d, "\phi"]
		\\
		\lsh{F(Q)}_{\mathscr{C}}
			\arrow[r, "\cotr(c_{M,P} \circ F(f) \circ c_{M,Q}^{-1})"']
		& \lsh{F(P)}_{\mathscr{C}}
	\end{tikzcd}\]
\end{enumerate}
\end{prop}

\begin{proof}
Part (i) is \cite[Proposition 8.3(i)]{Shadows_and_traces}. For part (ii), the diagram in Figure~\ref{fig:cotrace_functoriality} has $\phi \circ F(\cotr(f))$ along the left side and $\cotr(cF(f)c^{-1}) \circ \phi$ along the right side. Every unlabeled square commutes because of naturality of $\phi$, $c$, or $\theta$.
\end{proof}

\begin{figure}[h]
	\centering
	\resizebox{\textwidth}{!}{
		\begin{tikzpicture}[xscale=3.25, yscale=2]
	\newcommand{\Xa}{0}
	\newcommand{\Xb}{\Xa+1.2}
	\newcommand{\Xc}{\Xb+1.0}
	\newcommand{\Xcc}{\Xc+0.5}
	\newcommand{\Xd}{\Xc+1.0}
	\newcommand{\Xe}{\Xd+1.2}

	\node(A) at (\Xa, 9){$F\lsh{Q}$};
	\node(B) at (\Xb, 9){$\lsh{FQ}$};
	\node(C) at (\Xc, 9){$\lsh{U_{FS} \triangleright FQ}$};
	\node(D) at (\Xd, 9){$\lsh{FU_S \triangleright FQ}$};
	\node(E) at (\Xe, 9){$\lsh{F(M^* \odot M) \triangleright FQ}$};

	\node(F) at (\Xa, 8){$F\lsh{U_S \triangleright Q}$};
	\node(G) at (\Xb, 8){$\lsh{F(U_S \triangleright Q)}$};
	\node(H) at (\Xe, 8){$\lsh{(FM^* \odot FM) \triangleright FQ}$};

	\node(I) at (\Xa, 7){$F\lsh{(M^* \odot M) \triangleright Q}$};
	\node(J) at (\Xb, 7){$\lsh{F((M^* \odot M) \triangleright Q)}$};
	\node(K) at (\Xe, 7){$\lsh{FM^* \triangleright (FM \triangleright FQ)}$};

	\node(M) at (\Xa, 6){$F\lsh{M^* \triangleright (M \triangleright Q)}$};
	\node(N) at (\Xb, 6){$\lsh{F(M^* \triangleright (M \triangleright Q))}$};
	\node(O) at (\Xe, 6){$\lsh{FM^* \triangleright F(M \triangleright Q)}$};

	\node(P) at (\Xa, 5){$F\lsh{M^* \triangleright (P \triangleleft M)}$};
	\node(Q) at (\Xb, 5){$\lsh{F(M^* \triangleright (P \triangleleft M))}$};
	\node(R) at (\Xe, 5){$\lsh{FM^* \triangleright F(P \triangleleft M)}$};

	\node(S) at (\Xa, 4){$F\lsh{(P \triangleleft M) \triangleleft M^*}$};
	\node(T) at (\Xb, 4){$\lsh{F((P \triangleleft M) \triangleleft M^*)}$};
	\node(U) at (\Xcc, 4){$\lsh{F(P \triangleleft M) \triangleleft FM^*}$};
	\node(V) at (\Xe, 4){$\lsh{FM^* \triangleright (FP \triangleleft FM)}$};

	\node(W) at (\Xa, 3){$F\lsh{P \triangleleft (M \odot M^*)}$};
	\node(X) at (\Xb, 3){$\lsh{F(P \triangleleft (M \odot M^*))}$};
	\node(Y) at (\Xe, 3){$\lsh{(FP \triangleleft FM) \triangleleft FM^*}$};

	\node(AA) at (\Xe, 2){$\lsh{FP \triangleleft (FM \odot FM^*)}$};

	\node(AB) at (\Xa, 1){$F\lsh{P \triangleleft U_R}$};
	\node(AC) at (\Xb, 1){$\lsh{F(P \triangleleft U_R)}$};
	\node(AD) at (\Xe, 1){$\lsh{FP \triangleleft F(M \odot M^*)}$};

	\node(AE) at (\Xa, 0){$F\lsh{P}$};
	\node(AF) at (\Xb, 0){$\lsh{FP}$};
	\node(AG) at (\Xcc, 0){$\lsh{FP \triangleleft U_{FR}}$};
	\node(AH) at (\Xe, 0){$\lsh{FP \triangleleft FU_R}$};

	\node at (barycentric cs:N=1,H=1){\small Lemma \ref{lem:closed_functor_and_tensor_hom_adjunction}};
	\node at (barycentric cs:T=1,AA=1){\small Lemma \ref{lem:closed_functor_and_tensor_hom_adjunction}};

	\draw [->] (A) -- node[above]{$\phi$} (B);
	\draw [->] (F) -- node[above]{$\phi$} (G);
	\draw [->] (I) -- node[above]{$\phi$} (J);
	\draw [->] (M) -- node[above]{$\phi$} (N);
	\draw [->] (P) -- node[above]{$\phi$} (Q);
	\draw [->] (S) -- node[above]{$\phi$} (T);
	\draw [->] (W) -- node[above]{$\phi$} (X);
	\draw [->] (AB) -- node[above]{$\phi$} (AC);
	\draw [->] (AE) -- node[above]{$\phi$} (AF);
	\draw [->] (B) -- node[above]{$\lsh{\overline{r}}$} node[below]{$\cong$} (C);
	\draw [->] (C) -- node[above]{$\lsh{(i^{-1})^*}$} node[below]{$\cong$} (D);
	\draw [->] (D) -- node[above]{$\lsh{(F\varepsilon)^*}$} (E);
	\draw [->] (G) -- node[below right]{$\lsh{c}$} (D);
	\draw [->] (J) -- node[below right]{$\lsh{c}$} (E);
	\draw [->] (E) -- node[right]{$\lsh{c^*}$} (H);
	\draw [->] (H) -- node[right]{$\lsh{t}$} node[left]{$\cong$} (K);
	\draw [->] (K) -- node[right]{$\lsh{c^{-1}_*}$} node[left]{$\cong$} (O);
	\draw [->] (N) -- node[above]{$\lsh{c}$} (O);
	\draw [->] (A) -- node[left]{$F\lsh{\overline{r}}$} node[right]{$\cong$} (F);
	\draw [->] (B) -- node[left]{$\lsh{F\overline{r}}$} node[right]{$\cong$} (G);
	\draw [->] (F) -- node[left]{$F\lsh{\varepsilon^*}$} (I);
	\draw [->] (G) -- node[left]{$\lsh{F(\varepsilon^*)}$} (J);
	\draw [->] (I) -- node[left]{$F\lsh{t}$} node[right]{$\cong$} (M);
	\draw [->] (J) -- node[left]{$\lsh{Ft}$} node[right]{$\cong$} (N);
	\draw [->] (M) -- node[left]{$F\lsh{f_*}$} (P);
	\draw [->] (N) -- node[left]{$\lsh{F(f_*)}$} (Q);
	\draw [->] (O) -- node[left]{$\lsh{(Ff)_*}$} (R);
	\draw [->] (Q) -- node[above]{$\lsh{c}$} (R);
	\draw [->] (P) -- node[left]{$F(\theta)$} node[right]{$\cong$} (S);
	\draw [->] (T) -- node[above]{$\lsh{c}$} (U);
	\draw [->] (R) -- node[above left]{$\theta$} node[below right]{$\cong$} (U);
	\draw [->] (R) -- node[right]{$\lsh{c_*}$} (V);
	\draw [->] (V) -- node[right]{$\theta$} node[left]{$\cong$} (Y);
	\draw [->] (U) -- node[below left]{$\lsh{c_*}$} (Y);
	\draw [->] (S) -- node[left]{$F\lsh{t^{-1}}$} node[right]{$\cong$} (W);
	\draw [->] (Y) -- node[right]{$\lsh{t^{-1}}$} node[left]{$\cong$} (AA);
	\draw [->] (T) -- node[left]{$\lsh{F(t^{-1})}$} node[right]{$\cong$} (X);
	\draw [->] (X) -- node[above right]{$\lsh{c}$} (AD);
	\draw [->] (AA) -- node[right]{$\lsh{(c^{-1})^*)}$} node[left]{$\cong$} (AD);
	\draw [->] (W) -- node[left]{$F\lsh{\eta^*}$} (AB);
	\draw [->] (X) -- node[left]{$\lsh{F(\eta^*)}$} (AC);
	\draw [->] (AC) -- node[above right]{$\lsh{c}$} (AH);
	\draw [->] (AB) -- node[left]{$F\lsh{\overline{l}^{-1}}$} node[right]{$\cong$} (AE);
	\draw [->] (AC) -- node[left]{$\lsh{F(\overline{l}^{-1})}$} node[right]{$\cong$} (AF);
	\draw [->] (AD) -- node[left]{$\lsh{(F\eta)^*}$} (AH);
	\draw [->] (AH) -- node[below]{$\lsh{i^*}$} (AG);
	\draw [->] (AG) -- node[below]{$\lsh{\overline{l}^{-1}}$} node[above]{$\cong$} (AF);

	\node at (barycentric cs:P=1,U=1){Definition \ref{defn:lax_coshadow_functor}};
	\node at (barycentric cs:G=1,C=1.5,B=0.5){Definition \ref{defn:lax_closed_functor}};
	\node at (barycentric cs:AC=1,AF=0.5,AG=1){Definition \ref{defn:lax_closed_functor}};
\end{tikzpicture}
	}
	\caption{Diagram for Proposition~\ref{prop:cotrace_functoriality}}
	\label{fig:cotrace_functoriality}
\end{figure}

\section{Morita invariance}
\label{sec:Morita_invariance}

An important property of Hochschild homology is that it is Morita invariant, meaning $\HH_n(R) \cong \HH_n(S)$ whenever $R$ and $S$ are Morita equivalent rings. Moreover, there is a trace map which is an isomorphism between the Hochschild homologies of $R$ and $S$. In fact, this is an example of a general notion of Morita invariance satisfied by all shadow functors. After reviewing the classical notion of Morita equivalence and its generalization to bicategories and shadows, we will demonstrate that coshadows also satisfy Morita invariance.

\begin{defn}
\label{defn:Morita_equiv_classical}
Rings $R$ and $S$ are \KEYWORD{Morita equivalent} if their module categories $\MOD_R$ and $\MOD_S$ are equivalent.
\end{defn}

If there are bimodules ${}_RP_S$ and ${}_SQ_R$ such that $P \otimes_S Q \cong R$ (as $(R, R)$-bimodules) and $Q \otimes_R P \cong S$ (as $(S, S)$-bimodules), then
\[
- \otimes_R P : \MOD_R \rightleftarrows \MOD_S : - \otimes_S Q
\]
are mutually inverse equivalences of categories, and it turns out that any equivalence between $\MOD_R$ and $\MOD_S$ arises this way. This leads to a notion of Morita equivalence in a bicategory:

\begin{defn}[\cite{CP_THH}]
\label{defn:Morita_equiv}
Two 0-cells $R$ and $S$ in a bicategory $\mathscr{B}$ are \KEYWORD{Morita equivalent} if there are 1-cells $P \in \mathscr{B}(R, S)$ and $Q \in \mathscr{B}(S, R)$ and isomorphisms
\[
	\eta : U_R \xrightarrow{\cong} P \odot Q \qquad\text{and}\qquad \varepsilon : Q \odot P \xrightarrow{\cong} U_S
\]
satisfying the triangle identities.
\end{defn}

This means that $(P, Q)$ is a dual pair, and $\eta^{-1}$ and $\varepsilon^{-1}$ exhibit $(Q, P)$ as a dual pair as well. In fact, it is not necessary to assume that $\eta$ and $\varepsilon$ satisfy the triangle identities; we can replace one of them with another that does satisfy the triangle identities, since an equivalence of categories can always be upgraded to an adjoint equivalence. In the Morita bicategory $\MORITABICAT$, Definition~\ref{defn:Morita_equiv} recovers the classical notion of Morita equivalence (Definition~\ref{defn:Morita_equiv_classical}).

Shadows are Morita invariant:

\begin{prop}[{\cite[Proposition 4.8]{CP_THH}}]
Let $(P, Q)$ be a Morita equivalence between 0-cells $R$ and $S$ in a bicategory $\mathscr{B}$ with a shadow $\sh{-}$. Then $\sh{M} \cong \sh{Q \odot M \odot P}$ for any 1-cell $M \in \mathscr{B}(R, R)$, and moreover, there is a bicategorical trace witnessing this equivalence.
\end{prop}

\begin{proof}
The trace with respect to $P$ of $\eta \odot 1^2 : M \odot P \xrightarrow{\cong} P \odot Q \odot M \odot P$ is an isomorphism $\sh{M} \xrightarrow{\cong} \sh{Q \odot M \odot P}$.
\end{proof}

In the case that $M = U_R$, we obtain the following.

\begin{cor}
If $R$ and $S$ are Morita equivalent 0-cells in a bicategory with a shadow $\sh{-}$, then $\sh{U_R} \cong \sh{U_S}$.
\end{cor}

Coshadows are also Morita invariant:

\begin{prop}
\label{prop:cotrace_Morita_invariance}
Let $(P, Q)$ be a Morita equivalence between 0-cells $R$ and $S$ in a closed bicategory $\mathscr{B}$ with a coshadow $\lsh{-}$. Then $\lsh{M} \cong \lsh{Q \odot M \odot P}$ for any 1-cell $M \in \mathscr{B}(R, R)$, and moreover, there is a bicategorical cotrace witnessing this equivalence.
\end{prop}

\begin{proof}
The desired map is the composite
\begin{equation}
\label{eq:Morita_equiv_cotrace}
	Q \triangleright M \underset{\cong}{\to} M \odot P \underset{\cong}{\xrightarrow{\eta \odot 1^2}} P \odot Q \odot M \odot P \underset{\cong}{\to} (Q \odot M \odot P) \triangleleft Q,
\end{equation}
whose first and third maps come from Corollary~\ref{cor:hom_from_dualizable}. Since the cotrace of an isomorphism using a Morita equivalence as the dual pair is an isomorphism, the cotrace of (\ref{eq:Morita_equiv_cotrace}) is an isomorphism $\lsh{M} \xrightarrow{\cong} \lsh{Q \odot M \odot P}$.
\end{proof}

\begin{example}
The cotrace from Proposition~\ref{prop:cotrace_Morita_invariance}, using the categorical trace of \cite{GK2006} as the coshadow (Example~\ref{ex:categorical_trace}), is the map
\[
	\Hom_{\mathscr{B}(R, R)}(U_R, M) \to \Hom_{\mathscr{B}(S, S)}(U_S, Q \odot M \odot P)
\]
taking $\varphi : U_R \to M$ to
\begin{equation}
\label{eq:GK_cotrace}
	U_S \underset{\cong}{\xrightarrow{\varepsilon^{-1}}} Q \odot P \cong Q \odot U_R \odot P \xrightarrow{1 \odot \varphi \odot 1} Q \odot M \odot P.
\end{equation}
This is the isomorphism of \cite[Proposition~3.8(a)]{GK2006}. The reader may object that (\ref{eq:GK_cotrace}) clearly depends on the choice of $\varepsilon$, despite the claim that cotraces are independent of such choices (Lemma~\ref{lem:cotrace_well_defined}). This dependence comes not from the dual pair but from the map (\ref{eq:Morita_equiv_cotrace}) whose cotrace we have taken.
\end{example}

\FloatBarrier

\bibliographystyle{amsalpha}
\bibliography{bibliography.bib}

\end{document}